%% file: stratBunGellipticv4.tex
\documentclass[10pt, reqno]{amsart}
\numberwithin{equation}{section}

\usepackage{amsmath,amsthm, amsfonts, amssymb, mathrsfs}

\usepackage[hidelinks]{hyperref}
\usepackage[capitalize]{cleveref} 
\usepackage{tikz, tikz-cd} 
\usetikzlibrary{positioning}
\usetikzlibrary{decorations.markings}
\usepackage{dynkin-diagrams}

\usepackage{parskip, verbatim}
\usepackage{microtype,fixltx2e,lmodern} 
\usepackage[utf8]{inputenc} 
\usepackage[T1]{fontenc} 
\usepackage{enumerate,comment,braket,xspace,csquotes} 
\usepackage[all,cmtip]{xy} 
\usepackage{multirow}
\usepackage{stmaryrd}

\usepackage{xcolor}
\hypersetup{
	colorlinks,
	linkcolor={red!50!black},
	citecolor={blue!50!black},
	urlcolor={blue!80!black}
}
\tikzset{middlearrow/.style={
		decoration={markings,
			mark= at position 0.5 with {\arrow{#1}} ,
		},
		postaction={decorate}
	}
}

\newtheorem{proposition}{Proposition}[section]
\newtheorem{lemma}[proposition]{Lemma}
\newtheorem{theorem}[proposition]{Theorem}
\newtheorem*{theorem*}{Theorem}
\newtheorem{corollary}[proposition]{Corollary}
\newtheorem*{corollary*}{Corollary}

\newtheorem{definition}[proposition]{Definition}
\theoremstyle{definition}
\newtheorem{remark}[proposition]{Remark}
\newtheorem{example}[proposition]{Example}

\input{mynewcommands.tex}

\newcommand{\GIT}{\git}
\newcommand{\adjquot}{{/_{\hspace{-0.2em}ad}\hspace{0.1em}}}

\renewcommand{\red}{\mathsf{red}}
\renewcommand{\nil}{\mathsf{uni}} 
\newcommand{\uni}{\mathsf{uni}}
\newcommand{\fr}{\mathsf{fr}}
\newcommand{\strreg}{\mathsf{str\mhyphen reg}}

\newcommand{\maxreg}{\mathsf{max\mhyphen reg}}

\newcommand{\Vect}{\mathrm{Vect}}

\newcommand{\GE}{{G}_{E}}
\newcommand{\HE}{{H}_{E}}
\newcommand{\TE}{{T}_{E}}

\newcommand{\uGE}{\underline{G}_{E}} 

\newcommand{\uGEone}{\underline{G}_{E_1}}
\newcommand{\uGEtwo}{\underline{G}_{E_2}}
\newcommand{\uGEi}{\underline{G}_{E_i}}
\newcommand{\uGpE}{\underline{G}'_{E}}

\newcommand{\F}{\mathbb{F}}

\newcommand{\uTE}{\underline{T}_{E}}  
\newcommand{\uHE}{\underline{H}_{E}}

\newcommand{\ass}{\mathrm{ass}}
\newcommand{\Hilb}{\mathrm{Hilb}}

\newcommand{\Ga}{\mathbb{G}_{a}}
\newcommand{\Gm}{\mathbb{G}_{m}}

\newcommand{\QC}{\mathrm{QC}}

\newcommand{\Stab}{\mathrm{Stab}}

\newcommand{\Ind}{\mathrm{Ind}}
\newcommand{\Res}{\mathrm{Res}}

\newcommand{\Perv}{\mathrm{Perv}}

\newcommand{\uPE}{\underline{P}_E}
\newcommand{\uLE}{\underline{L}_E}

\newcommand{\MGE}{\mathcal{M}_E(G)}

\newcommand{\MGpE}{\mathcal{M}_{E}(G')} 
\newcommand{\MHE}{\mathcal{M}_E(H)} 

\newcommand{\one}{\mathbf{1}}
\newcommand{\node}{\mathrm{node}}
\newcommand{\cusp}{\mathrm{cusp}}

\newcommand{\rvline}{\hspace*{-\arraycolsep}\vline\hspace*{-\arraycolsep}}

\mathchardef\mhyphen="2D

\title[Geometry of semistable $G$-bundles]{The Jordan--Chevalley decomposition for $G$-bundles on elliptic curves}

\address{IRMA 7 rue René Descartes, Strasbourg, France}
\email{fratila@math.unistra.fr}
\author{Dragos Fratila}

\address{Department of Mathematics, King's College London, London, UK}
\email{sam.gunningham@kcl.ac.uk}
\author{Sam Gunningham}

\address{YMSC, Tsinghua University, Beijing, China}
\email{lipenghui@mail.tsinghua.edu.cn}
\author{Penghui Li}

\begin{document}
\begin{abstract}
	We study the moduli stack of degree $0$ semistable $G$-bundles on an irreducible curve $E$ of arithmetic genus $1$, where $G$ is a connected reductive group. 
	Our main result describes a partition of this stack indexed by a certain family of connected reductive subgroups $H$ of $G$ (the $E$-pseudo-Levi subgroups), where each stratum is computed in terms of bundles on $H$ together with the action of the relative Weyl group. 
	We show that this result is equivalent to a Jordan--Chevalley theorem for such bundles equipped with a framing at a fixed basepoint. 
	In the case where $E$ has a single cusp (respectively, node), this gives a new proof of the Jordan--Chevalley theorem for the Lie algebra $\fg$ (respectively, group $G$). 
	
	We also provide a Tannakian description of these moduli stacks and use it to show that if $E$ is an ordinary elliptic curve, the collection of framed unipotent bundles on $E$ is equivariantly isomorphic to the unipotent cone in $G$. 
	Finally, we classify the $E$-pseudo-Levi subgroups using the Borel--de Siebenthal algorithm, and compute some explicit examples.
\end{abstract}


\maketitle
\tableofcontents

\section{Introduction}
\subsection{Overview and main results}\label{SS:intro-overview}

Fix an irreducible projective curve $E$ of arithmetic genus 1 over an algebraically closed field $k$. 
There are three possibilities:
\begin{enumerate}
	\item $E$ has a single ordinary cusp;
	\item $E$ has a single node;
	\item $E$ is smooth.
\end{enumerate}
We also fix a basepoint $x_0$ in the smooth locus of $E$.
We will refer to these cases as the \emph{cuspidal}, \emph{nodal}, or \emph{elliptic} cases respectively. \footnote{Note that in the first two cases, unlike the third, the pair $(E,x_0)$ has no moduli.}

Fix also a connected reductive group $G$ over $k$. We consider the moduli stack
\[
\uGE := \Bun_G^{0,\ss}(E)
\]
of degree $0$, semistable $G$-bundles over $E$. 
We denote by $\GE$ the moduli stack of such bundles $\cP$ together with a framing - that is, a trivialization of the fiber $\cP_{x_0}$.

We will see that $\GE$ is a smooth algebraic variety with an action of $G$ (changing the trivialization), and that $\uGE=\GE/G$. In fact, by results of Friedman--Morgan \cite{FriedMorgIII}, we have:
\begin{itemize}
	\item $\GE \cong \fg = \Lie(G)$ if $E$ is cuspidal, and
	\item $\GE \cong G$ if $E$ is nodal.
\end{itemize}
In these cases, the action of $G$ on $\GE$ corresponds to the adjoint/conjugation action. 
In the elliptic case, we will see that the stack $\uGE$ shares many of the properties of the adjoint quotient stacks $\fg/G$ and $G/G$.

\subsubsection{The Jordan--Chevalley Decomposition}
	One of our main results is a form of the Jordan--Chevalley decomposition for $G_E$, which recovers the usual Jordan--Chevalley decomposition for $\fg$ and $G$ in the cuspidal and nodal cases respectively. 
	
	An element $p \in \GE$ is called \emph{semisimple} if its $G$-orbit is closed. 
	We say that an element $p\in \GE$ is \emph{unipotent} if its orbit closure contains the trivial bundle $p_0$. 
	We denote by $\GE^\uni$ the subvariety of unipotent elements of $\GE$ and by $\uGE^\uni$ the corresponding substack in $\uGE$. 
	We note that these definitions recover the usual notion of semisimple and unipotent/nilpotent elements in the nodal/cuspidal cases.
	
	The Jordan--Chevalley decomposition essentially states that every framed bundle $p\in \GE$ can be uniquely decomposed as $p = p_s \cdot p_u$ where $p_s$ is semisimple, $p_u$ is unipotent, and $p_s$ commutes with $p_u$. 
	As it stands this statement is not well-formed as it is not clear what it means to ``multiply'' elements in $\GE$ (the usual statement of Jordan decomposition in the nodal and cuspidal cases uses multiplication in $G$ and addition in $\fg$). 
	However, it does make sense to multiply a $Z(G)$-bundle $\cP'$ and a $G$-bundle $\cP$: we define
	$\cP' \cdot \cP$
	to be the bundle induced from the external product $\cP' \times_E \cP$ via the multiplication map $Z(G) \times G \to G$ (which is a group homomorphism). Moreover, this construction is naturally compatible with framings, giving rise to an abelian algebraic group structure on $Z(G)_E$ and an action of $Z(G)_E$ on $\GE$.
	
	We will also show that for a reductive subgroup $H$ of $G$, the induction map on framed semistable bundles $\HE \to \GE$ is a closed embedding (see \cref{P:framed H semistable injects in G semistable}).
	Therefore we can identify $\HE$ with the corresponding closed subvariety of $\GE$. 
	We are now ready to state the first result, a form of Jordan--Chevalley decomposition:
	
	\begin{theorem}\label{thm:jordan intro} (see \cref{thm:jordan})
		Given $p\in \GE$, there is a unique triple $(H,p_s,p_u)$, where 
		$H$ is a connected reductive subgroup of $G$, $p_s \in Z(H)_E$ with $\Stab_G(p_s)^\circ=H$, $p_u \in \HE^\uni$, and $p=p_s \cdot p_u$.
	\end{theorem}
	
	The subgroups of $G$ which occur as connected stabilizers of semisimple elements of $\GE$ will be called \emph{$E$-pseudo-Levi subgroups}. 
	In the cuspidal case, these are precisely the Levi subgroups (centralizers of semisimple elements of $\fg$) and in the nodal case, these are pseudo-Levi subgroups  (connected centralizers of semisimple elements of $G$). 
	We give a classification of $E$-pseudo-Levi subgroups in \cref{sec:classification-of-elliptic-closed-subsets}: in the elliptic case we get precisely the intersections of two pseudo-Levi subgroups.
	
	\begin{remark}
		For simply connected groups over $\bC$, a similar Jordan--Chevalley decomposition was proved in \cite[Theorem 5.6]{BaEvGi_rep-qtori-bdles-ell} using an algebraic uniformization of $\uGE$ through loop groups (see \cite{BaGi_conj}).
	\end{remark}

\subsubsection{The partition according to $E$-pseudo-Levis}\label{SS:intro-partition-E-pseudoLevi}
Typically, in the statements of the Jordan--Chevalley decomposition for $G$ and for $\fg$, the subgroup $H$ is 
not explicitly mentioned, though it may be easily recovered as the connected centralizer of the semisimple element. 

However, in our proof of \cref{thm:jordan intro}, the subgroup $H$ will be the key player and the decomposition in to semisimple and unipotent elements will play a subsiduary role. In fact, the semisimple and unipotent elements may be recovered from the subgroup $H$ in the following sense.

Given a subgroup $H$ of $G$ we define the loci
\[
(Z(H)_E)^\reg = \left\{p \in Z(H)_E \mid \Stab_G(p)^\circ = H \right\} \subseteq Z(H)_E
\]
and
\[
(\HE)_\heartsuit^\reg = (Z(H)_E)^\reg \cdot \HE^\uni \subseteq \HE.
\]
The subgroup $H$ is an $E$-pseudo-Levi precisely when $(Z(H)_E)^\reg$ is nonempty. 
Note that the conditions in \cref{thm:jordan intro} are precisely stipulating that $p\in (\HE)_\heartsuit^\reg$. 
On the other hand, any element $p \in (\HE)_\heartsuit^\reg$ may be uniquely written as a product $p_s \cdot p_u$ 
where $p_s \in (Z(H)_E)^\reg$ and $p_u \in \HE^\uni$ (see \cref{P:heart locus is a product}). 
We denote by $(\GE)_H$, the image of $(H_E)_\heartsuit^\reg$ in $\GE$. 

In other words, \cref{thm:jordan intro} states that every element $p\in \GE$ lies in the image of $(\HE)_\heartsuit^\reg$ for a unique $E$-pseudo-Levi subgroup $H$. Thus we may rephrase \cref{thm:jordan intro} as follows:

\begin{theorem}\label{thm:jordan partition intro}
There is a locally closed partition:
	\[
	\GE = \bigsqcup_{H} (\GE)_H
	\]
	indexed by $E$-pseudo-Levi subgroups $H \subseteq G$. Moreover, the natural embeddings $\HE \to \GE$ restrict to isomorphisms
\[
(\HE)_\heartsuit^\reg \xrightarrow{\sim} (\GE)_H.
\]
\end{theorem}

We will reformulate this result in one final way (in the form that will actually be proved in \cref{sec:the-jordan--chevalley-theorem}). Recall that $\uGE$ denotes the quotient stack $\GE/G$, and write $(\uHE)_\heartsuit^\reg$ for $(\HE)_\heartsuit^\reg/H$. We write $W_{G,H}$ for the relative Weyl group $N_G(H)/H$ (a finite group).

\begin{theorem}\label{thm:galois intro}
	The stack $\uGE$ carries a locally closed partition
	\[
	\uGE = \bigsqcup _{[H]} (\uGE)_{[H]}
	\]
	indexed by conjugacy classes $[H]$ of $E$-pseudo-Levi subgroups $H \subseteq G$. 
	Moreover, the natural induction maps $\uHE \to \uGE$ restrict to equivalences
	\[
	(\uHE)_\heartsuit^\reg/W_{G,H} \xrightarrow{\sim} (\uGE)_{[H]}.
	\]
\end{theorem}
Our proof of \cref{thm:galois intro} (and hence of \cref{thm:jordan intro} and \cref{thm:jordan partition intro}) 
involves a geometric analysis of the induction map $\uHE \to \uGE$ (see \cref{SS:main results}).

\begin{remark}
A key difference between the statements of \cref{thm:galois intro} and \cref{thm:jordan partition intro} is that, 
unlike the subvarieties $(\GE)_H$ in $\GE$, the substacks $(\uGE)_{[H]}$ have an a priori definition that doesn't 
make reference to the induction map $\uHE \to \uGE$. 
More precisely, given a $G$-bundle $\cP\in\uGE$, pick a framed lift $p\in\GE$ and then choose any semisimple bundle in the closure of its $G$-orbit.
The underlying $G$-bundle defines a conjugacy class $[H]$ that is independent of the chosen framing and is hence canonically associated to $\cP$.
The content of \cref{thm:jordan partition intro} is that, in the presence of a framing, there is a canonical choice of subgroup $H$ within its conjugacy 
class, and (equivalently) a canonical choice of semisimple element $p_s$ in the orbit closure of $p$.
\end{remark}

	\subsubsection{Unipotent bundles}
	\cref{thm:galois intro} and \cref{thm:jordan intro} allow us to reduce the study of degree $0$, semistable $G$-bundles on $E$ to semisimple and unipotent bundles.
	
	Our next result shows that, under certain hypotheses, the collection of unipotent bundles in $\GE$ is insensitive to the isomorphism type of $E$. 
	We let $J(E)$ denote the Jacobian of $E$, which is either isomorphic to $\Ga$, $\Gm$, or $E$ itself in the cuspidal, nodal, or elliptic cases respectively. 
	We denote by $G^\uni$ the unipotent cone of $G$ (which is the same as $\GE^\uni$ for $E$ a nodal curve).
	
	\begin{theorem}\label{thm:unipotent G bundles intro}
		An isomorphism of formal group $\hat J(E)\simeq \hat \bG_m$ induces an isomorphism of $G$-varieties
		\[
		\GE^\uni \cong G^\uni.
		\]
		Moreover this isomorphism extends over a formal neighbourhood of $\GE^\uni$ in $\GE$:
		\[
		(\GE)^\wedge_\uni \cong G^\wedge_\uni.
		\]
	\end{theorem}
	\begin{remark}
		In characteristic zero, there is always an isomorphism $\hat J(E)\simeq \hat \bG_m$. 
		In fact, there is also an isomorphism $\hat J(E)\simeq \hat \bG_a$ which gives the same result but with the unipotent cone in $G$ replaced by the nilpotent cone $\cN$ in $\fg$.
		
		In characteristic $p>0$, if $E$ is an elliptic curve, an isomorphism $\hat J(E)\simeq \hat \bG_m$ exists precisely when $E$ is ordinary (i.e. not supersingular). 
	\end{remark}
	
	\begin{remark}
		As a special case of \cref{thm:unipotent G bundles intro} we recover a $G$-equivariant isomorphism $\cN \cong \cU$ between the unipotent and nilpotent cones for each isomorphism $\hat \bG_a \cong \hat \bG_m$. The latter isomorphisms exist only in characteristic zero, and are given by exponential maps. 
		On the other hand, there exist $G$-equivariant isomorphisms (the so-called Springer isomorphisms) $\cN \cong \cU$ under very mild conditions on the characteristic, even though $\hat \bG_a \ncong \hat \bG_m$ in positive characteristic. 
		It seems reasonable to expect that $\GE^\uni$ is isomorphic to $\cU$ (and $\cN$) under much more general conditions than in \cref{thm:unipotent G bundles intro}.
		From \cite[Theorem 3.11]{Groj-ShBar} one can deduce that the varieties $\GE^\uni$ and  $G^\uni$ are smoothly equivalent for uniformizable elliptic curves under some restrictions on $G$ and the characteristic. 
		See also \cite[Corollary 8.8]{Groj-ShBar} where they show it fails for $G=E_8$, $E$ supersingular in characteristic 2,3 or 5.
	\end{remark}

	\subsubsection{Semisimple bundles and the classification of $E$-pseudo-Levi subgroups}
	
	Fix a maximal torus $T$ of $G$ and let $\Phi \subset \bX^*$ denote the corresponding space of roots sitting inside the character lattice $\bX^* = \bX^*(T)$. Then 
	\[
	\TE \cong \Hom_{gp}(\bX^*,J(E)) \cong J(E)^r
	\]
	where $r$ is the rank of $T$. 
	
	Note that a $G$-bundle $\cP \in \uGE$ is semisimple if and only if it admits a reduction to $T$. 
	We may understand the partition in to $E$-pseudo-Levi subgroups root theoretically as follows.
	
	First note that any character $\alpha \in \bX^*$ defines a homomorphism $\alpha_\ast\colon\TE \to J(E)$, taking a $T$-bundle on $E$ to its induced line bundle via $\alpha$. 
	Given $p\in \TE$, we let $\Sigma_p = \{\alpha \in \TE \mid \alpha_\ast(p) = \one_{J(E)}\}$. 
	The subsets $\Sigma$ of $\Phi$ which occur in this way will be called $E$-root subsystems of $\Phi$.
	
	It turns out that any such $E$-root subsystem $\Sigma\subset\Phi$ is a closed root subsystem and so it corresponds to a connected reductive subgroup $H$ of $G$ (see also  \cref{SS:thm of Borel de Siebenthal} for Borel--de Siebenthal theory). 
	In fact, we have $H=\Stab_G(p)^\circ$ and

	\begin{proposition}\label{P:pseudo-Levi bijection closed subsets}
	There is a one-to-one correspondence between:
	\begin{itemize}
		\item $E$-pseudo-Levi subgroups $H$ of $G$ containing $T$, and
		\item $E$-root subsystems $\Sigma \subseteq \Phi$.
	\end{itemize}
	Moreover, in each of the cases cuspidal, nodal and elliptic one can characterize precisely the $E$-root subsystems of $\Phi$ (see the Appendix and \cref{P:the set A_ell} for the elliptic case).
	\end{proposition}

\begin{remark}
The theory of semisimple bundles becomes increasingly complicated as one passes from the cuspidal to the nodal and then to the elliptic cases. For example, the centralizer of a semisimple element of $\fg$ is a Levi subgroup, and in particular connected. The centralizer of a semisimple element in $G$ is connected (but not necessarily simply-connected) whenever $G$ is simply connected. 
On the other hand, the automorphism group of a semisimple $G$-bundle $\cP \in \uGE$ where $E$ is smooth may be disconnected, even if $G$ is simply connected! 
An example is given for $G$ of type $D_4$, see \cref{Eg:type D4} (this example also appears in \cite[p. 18]{BaEvGi_rep-qtori-bdles-ell}). 

Fortunately, our result provides control over the component groups of automorphisms of semisimple bundles in terms of Weyl group combinatorics (just as Lusztig's stratification does in the group case). More precisely, let $p \in \TE$ with $\Sigma_p=\Sigma$  for some $E$-root subsystem $\Sigma$ of $\Phi$, and let $\cP_G$ be the induced $G$-bundle. 
Then \cref{thm:galois intro} implies that 
\[
\pi_0\Aut(\cP_G) \cong \Stab_{N_W(\Sigma)}(p)/W_\Sigma
\]
where $N_W(\Sigma) = \{ w\in W \mid w(\Sigma)=\Sigma\}$.
\end{remark}
	
	\subsubsection{The Lusztig stratification}
	Putting these results together, we can refine the partition of $\GE$ in \cref{thm:galois intro} as follows.
	
	\begin{corollary}
		Suppose either that $char(k)=0$, or that $E$ is ordinary. There is a stratification 
		\[
		\uGE = \bigsqcup_{[H,\cO]} (\uGE)_{[H,\cO]}
		\]
		indexed by $G$-conjugacy classes of pairs $(H,\cO)$ where $H$ is an $E$-pseudo-Levi subgroup of $G$, and $\cO \subseteq \HE^\uni$ is a unipotent $H$-orbit. For each such pair $[H,\cO]$ we have an isomorphism:
		\[
		(\uGE)_{[H,\cO]} \cong (Z(H)_E^\reg \times \cO)/N_G(H) \cong (Z(H)_E^\reg \times \underline{\cO})/W_{G,H}.
		\]
	\end{corollary}

\begin{remark}
		To make the comparison with Lusztig's work more evident, note that to each $E$-pseudo-Levi subgroup $H$, one can associate a Levi $L=L_H = C_G(Z(H)^\circ)$. This is the smallest Levi subgroup which contains $H$. 
		A bundle $\cP \in \uGE$ is called \emph{isolated} if it is contained in $(\uGE)_{[H]}$ and $H$ is not contained in any proper Levi. (It follows from \cref{P:pseudo-Levi bijection closed subsets} that there are finitely many isomorphism classes of isolated bundles.) Instead of parameterizing the strata using subgroups of elliptic type and unipotent bundles, one may use Levi subgroups and isolated bundles. This is how Lusztig describes the stratification of $G$ in \cite{Lu_gen_spr}.
\end{remark}

\subsection{Motivation: the geometric Langlands program and elliptic Springer theory} 

\subsubsection{Global geometric Langlands}\label{sec:global-geometric-langlands-on-e}
Recall that in the global geometric Langlands program, one aims to describe the derived category $\cD(\Bun_G(C))$ of $D$-modules or constructible sheaves on the moduli stack of $G$-bundles on a smooth projective curve $C$ in terms of the derived category of quasi-coherent (or more general Ind-coherent) sheaves on the moduli stack of $\LG$-local systems, where $\LG$ is the Langlands dual group to $G$.

Recall that for every parabolic subgroup $P \subseteq G$ with Levi factor $L$, we have \emph{Eisenstein} and \emph{constant term} functors:
\[
\Eis_{L,P}^G: \cD(\Bun_L(C)) \leftrightarrows \cD(\Bun_G(C)):\CT_{L,P}^G
\]
Generally speaking, these functors are defined via a pull-push construction involving the diagram:\footnote{More precisely, one must consider a fiberwise compactification of $\Bun_P(C)$ over $\Bun_G(C)$. One must also specify which functors (star or shriek) are employed for this process. We will ignore these distinctions for the purposes of this informal discussion.}
\begin{equation}\label{Diag:Eisenstein and constant}
 \xymatrix{
	& \Bun_P(C) \ar[rd]^p\ar[ld]_q &\\
	\Bun_L(C) &&\Bun_G(C)
} 
\end{equation}
An object of $\cD(\Bun_G(C))$ is called \emph{cuspidal} if it is in the kernel of the constant term functor for every proper parabolic subgroup $P$ of $G$. As per Harish-Chandra's philosophy of cusp forms, one may think of the category $\cD(\Bun_G(C))$ as built up from Eisenstein series of cuspidal objects in $\cD(\Bun_L(C))$ as $L$ ranges over Levi subgroups of $G$. In this way one can hope to understand the category $\cD(\Bun_G(C))$ inductively in terms of smaller reductive groups.

Typically, the Eisenstein series from different Levi subgroups will interact in a complicated way, and this does not lead to a straightforward description of $\cD(\Bun_G(C))$. We will now describe a closely related category for which the Harish-Chandra approach yields a complete description.

\subsubsection{Springer Theory and Character Sheaves}\label{sec:springer-theory-and-character-sheaves}
Let us turn to the category $\cD(G/G)$ of ``class sheaves'' on $G$, i.e. conjugation equivariant constructible sheaves or $D$-modules on the group $G$. Analogous to the Eisenstein and constant term functors above, one has functors of parabolic induction and restriction:
\[
\Ind_{P,L}^G : \cD(L/L) \leftrightarrows \cD(G/G): \Res_{P,L}^G
\]
defined by pull-push along the diagram:

\begin{equation}\label{Diag:ind and res for G}
\xymatrix{
	&P/P \ar[dr]\ar[dl]\\
 L/L  & &G/G.
}
\end{equation}

Just as in \cref{sec:global-geometric-langlands-on-e} one defines the cuspidal objects as those whose parabolic restriction to any proper Levi is zero. Again, one aims to describe the category $\cD(G/G)$ in terms of parabolic induction from cuspidal objects in $\cD(L/L)$ as $L$ ranges over Levi subgroups. 

In his seminal paper \cite{Lu_gen_spr}, Lusztig obtained a block decomposition of the category of equivariant 
perverse sheaves on the unipotent cone $G^\uni \subseteq G$:
\begin{equation}\label{eq:gen spr}
\Perv_G(G^\uni) = \bigoplus_{(L,\cO,\cE)} \Rep(W_{G,L}).
\end{equation}
Here, the blocks are indexed by \emph{cuspidal data}, consisting of a pair of a Levi subgroup $L$ together with a simple cuspidal local system $\cE$ on the unipotent cone of $L$. The classical Springer correspondence for representations of the Weyl group $W=W_{G,T}$ is recovered inside of \cref{eq:gen spr} as one of these blocks, corresponding to the unique cuspidal datum with $L=T$, a maximal torus.

More recently, more general forms of the decomposition \cref{eq:gen spr} have been obtained for the derived 
category of nilpotent orbital sheaves by Rider--Russell \cite{RiderRussell}, the category of 
equivariant $D$-modules on the Lie algebra $\fg$ by S.G. 
\cite{gunningham_generalized_2018,gunningham_derived_2017} and the derived category of character sheaves by  P.L. \cite{li_derived_2018}. 

In the the setting of $D$-modules on $\fg/G$ or sheaves supported on the unipotent cone of $G$, the set of cuspidal data indexing the decomposition is the same as in \cref{eq:gen spr}. We will call this set the \emph{unipotent} cuspidal data for $G$. In the case of character sheaves on $G/G$, however, this set must be expanded to account for unipotent cuspidal data for pseudo-Levi subgroups $H$ of $G$. More generally, one can consider $E$-cuspidal data for any arithmetic genus 1 curve as we explain further now.

\subsubsection{Elliptic Springer theory}
The main object of study in this paper is the stack $\uGE \subseteq \Bun_G(E)$ of semistable $G$-bundles on a curve of arithmetic genus 1. 

The stack $\uGE$ which we study in the present paper forms a bridge between the situations described in \cref{sec:global-geometric-langlands-on-e} and \cref{sec:springer-theory-and-character-sheaves}. On the one hand $\uGE$ sits inside $\Bun_G(E)$ as the locus of degree $0$ semistable bundles. On the other hand, when $E$ is taken to be a nodal curve one has an isomorphism $\uGE \cong G/G$. 

Viewing the category $\cD(\uGE)$ through the lens of \cref{sec:springer-theory-and-character-sheaves}, one defines functors of parabolic induction and restriction using the correspondence
\[
\uLE \leftarrow \uPE \rightarrow \uGE
\]
and study the corresponding Harish-Chandra (or generalized Springer decomposition). The ordinary Springer correspondence in the elliptic setting was studied by Ben-Zvi and Nadler \cite{BZN}. 

One can also formulate a generalized Springer correspondence for the (various flavours of) category $\cD(\uGE)$ which recovers the standard patterns for orbital and character sheaves in the cuspidal and nodal cases. This will be expanded on in future work; for now, let us just note that the indexing set of the generalized Springer decomposition (which we are calling $E$-cuspidal data) involves a choice of an $E$-pseudo-Levi subgroup $H$ of $G$, together with a $W_{G,H}$-equivariant simple unipotent cuspidal local system for $H$.

In this way \cref{thm:galois intro} may be thought of as a geometric antecedent to the generalized Springer correspondence for $\cD(\uGE)$: it expresses the geometry of the stack $\uGE$ in terms of unipotent orbits for $E$-pseudo-Levi subgroups. We note that the work of P.L. with David Nadler \cite{li2015uniformization} is another expression of the idea that $\uGE$ is glued together from data indexed by $E$-pseudo-Levi subgroups. However, the techniques of \emph{loc. cit.} involve an analytic uniformization of $\uGE$, whereas the present paper stays within the realm of algebraic geometry.

\subsubsection{Elliptic geometric Langlands}
We may view the category $\cD(\uGE)$ as a subcategory (via extension by zero from the semistable locus, say) of the automorphic geometric Langlands category $\cD(\Bun_G(E))$ for an elliptic curve $E$. In this way, (generalized) Springer theory is embedded in to the geometric Langlands correspondence. This perspective is exposited in \cite{li2015uniformization}[Section 1.3.1.2]. 

As a cautionary remark: there are now two different notions of cuspidal for an object of $\cD(\uGE)$: one using the 
constant term functor via diagram \cref{Diag:Eisenstein and constant} (we will call this Langlands cuspidal) and the 
other using the parabolic restriction functor via diagram \cref{Diag:ind and res for G} (we will call this Springer 
cuspidal). If an object of $\cD(\uGE)$ is Langlands cuspidal it is necessarily Springer cuspidal, but it is not clear if 
the converse holds: the constant term may be supported on non-zero components of $\Bun_T(E)$. 

Despite these difficulties, our results in this paper can be used to obtain strong restrictions on the existence and support of Langlands cuspidals on $\Bun_G$. As a simple example, it can be shown that any Langlands cuspidal Hecke-eigensheaf in $\cD(\Bun_{SL_2})$ must restrict to one of the four Springer cuspidal objects in $(\underline{SL}_2)_E$.

\subsubsection{Quantum geometric Langlands} 

However, the relationship between Springer theory and geometric Langlands is the most direct when one considers the quantum deformation. Namely, one studies the category $\cD_{\kappa}(\Bun_G)$ of $\cL^{\otimes \kappa}$-twisted $D$-modules on $\Bun_G$ where $\cL$ is the determinant line bundle (for this discussion let us restrict to the case where the ground field $k$ has characteristic zero). Here $\kappa$ can be taken to be any complex number. The quantum geometric Langlands conjecture posits an equivalence (assuming $\kappa$ is not the critical level which we normalize to $\kappa = 0$)
\[
\cD_{\kappa}(\Bun_G) \simeq \cD_{-1/\kappa}(\Bun_{G^\vee})
\]
When $\kappa$ is irrational it can be shown that any non-zero object of $\cD_\kappa$ is cleanly supported on the semistable locus. In this way, the quantum geometric Langlands equivalence at irrational level $\kappa$ reduces to a statement about cuspidal data for $G$ and $\LG$. We plan to return to this in future work.

There is also a Betti formulation of quantum geometric Langlands (see \cite{Betti}). For an elliptic curve $E$, this involves\footnote{The Betti formulation of quantum geometric Langlands is not as symmetric as in the de Rham setting; here we are describing the category that naturally lives on the spectral side of geometric Langlands correspondence; the automorphic version will be described in terms of certain twisted sheaves on $\Bun_G$.} the category $\cD_q(G/G)$ of quantum $D$-modules on $G/G$ (see \cite{ben-zvi_integrating_2018}). (Here, the parameter $q$ is roughly an exponential of the $\kappa^\vee$ appearing in the de Rham formulation above.) There is a conjectural generalized Springer correspondence for $\cD_q(G/G)$. Interestingly, though the categories are very different, one expects the same (discrete) cuspidal data to appear for the generalized Springer decomposition of $\cD(\uGE)$ and $\cD_q(G/G)$. S.G. hopes to expand on this point in forthcoming work with David Jordan and Monica Vazirani. We note that the ``Springer block'' of such a quantum generalized Springer would involve $W$-equivariant modules for the algebra of $q$-difference modules on the torus $T$, which may explain the connection with the work of Baranovsky, Evans and Ginzburg \cite{BaEvGi_rep-qtori-bdles-ell}.

\subsubsection{Eisenstein Sheaves and elliptic Hall categories}

The main motivation for D.F. to study the stratification described in this paper concerns not cuspidal sheaves but rather Eisenstein sheaves. 
More precisely, one can define the category $\cQ_G(E)$ of principal spherical Eisenstein sheaves on an elliptic curve $E$ as the category generated by
$\Eis_{T,B}^G(\underline{\bQ_l})$ inside $\cD^b(\Bun_G(E))$. 
This can be thought of as a categorical version of the space of spherical automorphic functions for the field of functions on a smooth projective curve over a finite field. 
The latter, at least for the groups $\GL_n$, is nothing else but the degree $n$ part of the spherical Hall algebra of the curve. 
See \cite{Lau,Sch_Vas_McD,Sch2,Frat_cusp} for more details about the relationship to automorphic functions. 
Therefore, one can think of $\cQ_G(E)$ for the groups $\GL_n$ as a categorical version of the spherical Hall algebra. It can be actually proved  \cite{Sch2} that one obtains in this way a categorification of the elliptic Hall algebra. 
The situation for higher genus curves is not well understood.

A sensible question to ask (for elliptic curves) is the classification of simple objects of the spherical category $\cQ_G(E)$. 
Actually the proof of the above mentioned result on categorification goes by first establishing such a classification. 
The main result of \cite{BZN} implies that there is an injection $\Irr(W)\hookrightarrow \Irr(\cQ_G(E))$, where $W$ is the Weyl group of $G$. 
Previously, in \cite{Sch2} this was shown to be a bijection for $\GL_n$.
The precise expectation for simply connected groups is that the above map is a bijection. 

Our main results in this article allow us to rule out some of the sheaves that could appear in $\Irr(\cQ_G(E))$ and that don't arise from $\Irr(W)$. 
In the case of $\SL_n$ and some groups of small rank this is enough to confirm the above sought for bijection. 
However, at the moment, we don't know if it's true for all simply connected groups.

\subsubsection{Affine character sheaves and local geometric Langlands}
The local geometric Langlands program provides yet another interpretation of the category $\cD(\uGE)$. Roughly 
speaking, one is motivated by local Langlands to study the category $\cD(LG/LG)$ of class sheaves for the loop 
group $LG$ of $G$. This may be considered as a natural home for what one might call affine character sheaves. 
This category is technically very difficult to study (or even define). However, a slight deformation $LG/_q LG$ of the 
stack $LG/LG$ is closely related to the moduli stack $\Bun_G(E_q)$ for a certain elliptic curve $E_q$. Thus one may 
consider $\cD(\Bun_G(E_q))$ as an avatar of affine class sheaves (this idea appears in \cite{BZNP_affch}).

The most natural formulation of the relationship between $LG/LG$ and $\Bun_G$ is analytic. In (unpublished) work of Looijenga, it was shown that there is an equivalence of complex analytic stacks 
\[
L^{hol}G/_q L^{hol}G \simeq \Bun_G^{an}(E_q)
\]
Here $q\in \bC^\times$, $|q|\neq 1$, $E_q = \bC^\times/q^\bZ$ is the corresponding elliptic curve, and $L^{hol}G = \mathrm{Map}_{hol}(\bC^\times,G)$ is the holomorphic loop group which acts on itself by $q$ twisted conjugation:
\[
\Ad_q(g(z)) h(z) = g(qz)h(z)g(z)^{-1}
\]
This idea of analytic unformization was generalized by P.L. with David Nadler \cite{li2015uniformization}, leading to analytic proofs of results closely related to those in this paper.

\subsection{Acknowledgements}
This project grew from discussions in 2014 while D.F. and S.G. were attending the Geometric Representation Theory 
program at MSRI, Berkeley, and P.L. was a graduate student at U.C. Berkeley. S.G. also worked on this project while 
attending the Higher Categories and Categorification program at MSRI in 2020. We thank MSRI for their support 
during the writing of this paper. 

S.G. was partially supported by Royal Society grant RGF\textbackslash EA\textbackslash 181078, and the European Research Council (ERC) under the European Union’s Horizon 2020 research and innovation programme (grant agreement no. 637618).

We would like to thank Michel Brion, Dougal Davis, Ian Grojnowski, David Nadler, Mauro Porta and 
Nick-Shepherd-Barron for helpful conversations.

\section{Preliminaries}
\subsection{Notation and generalities on $G$-bundles}

Let $X$ be a scheme and $G$ an affine algebraic group over an algebraically closed field $k$.
By a $G$-bundle (or principal $G$-bundle, or $G$-torsor) over a scheme $X$ we mean a scheme $\pi\colon\cP\to X$ over $X$ with a $\pi$-invariant right $G$-action such that, étale locally on $X$, $\pi\colon\cP\to X$ is $G$-equivariantly isomorphic to $\pi_1:X\times G\to X$. 
In other words, there exists $X'\to X$ étale and surjective such that the pullback of $\cP$ to $X'$ is $G$-equivariantly isomorphic to $X'\times G$
\[ \begin{tikzcd}
X'\times G \arrow[dr,"\pi_1",swap] \rar{\simeq} & X'\times_X \cP \dar{\pi'}\\
& X'
\end{tikzcd} \]

When there's no danger of confusion we will simply write $\cP$ for a $G$-bundle and omit the mention of $\pi\colon\cP\to X$.

If $\cP$ is a $G$-bundle over a scheme $X$ and $Y$ is a quasi-projective variety with a left $G$-action, then we can form the associated fiber space over $X$ with fiber $Y$ as
$ Y_\cP=\cP\times^G Y:=(\cP\times Y)/G $. If moreover $Y$ has a right $H$-action for some group $H$, then $Y_\cP$ is naturally endowed with an $H$-action.

We will apply the above construction in two particular cases: 
\begin{itemize}
	\item if $\rho\colon G\to H$ is a morphism of groups, then to a $G$-bundle $\cP$ we associate an $H$-bundle: $H_\cP=\cP\times^G H$. We'll also denote it by $\rho_*(\cP)$.
	\item if $V$ is a representation of $G$ (viewed as an affine space with a left $G$-action) then to a $G$-bundle $\cP$ we associate the vector bundle $V_\cP=(\cP\times V)/G$.
\end{itemize}

\begin{example}\label{Eg:product with a central bundle}
A particularly important case of the first situation above is for the group morphism $m\colon Z(G)\times G\to G$.
If $\cP'$ is a $Z(G)$-bundle and $\cP$ is a $G$-bundle we denote the induced $G$-bundle $m_*(\cP'\times\cP)$ by $\cP'\cdot\cP$.
In the case of $\GL_n$ it corresponds to tensoring a vector bundle by a line bundle.
\end{example}

\begin{example}	\hfill
	\begin{enumerate}
		\item 	If $G=\bG_m^2$ then a $G$-bundle is simply a pair of line bundles $(\cL,\cM)$. 
		If $\alpha\colon\bG_m^2\to \bG_m$ is given by $\alpha(t,s) = ts^{-2}$ then $\alpha_*(\cL,\cM) = \cL\otimes\cM^{-2}$.
		
		\item If $G=\GL_n$ then the category of $G$-bundles is equivalent ot the category of vector bundles of rank $n$. 
		The correspondence is given by associating to a $\GL_n$-bundle $\cP$ the vector bundle $\cP\times^{\GL_n}\bA^n$. 
		We will use this correspondence tacitly especially in the case $G=\bG_m$ where we think of a $\bG_m$-bundle as a line bundle using the natural representation of $\bG_m$ on $\bA^1$.
		
		\item Consider the Lie algebra $\fg$ with the adjoint action of $G$. For a $G$-bundle $\cP$ we have its adjoint bundle $\cP\times^G\fg=:\fg_\cP$ that will play an important role in the text. If $G=\GL_n$ then the adjoint bundle is none other than the vector bundle of endomorphisms.
	\end{enumerate}
\end{example}

\subsection{Semistability}
First let us recall the definition of slope: for a vector bundle $\cV$ on a smooth curve $X/k$ we put $\mu(\cV):=\deg(\cV)/\rank(\cV)\in\bQ$ and call it the slope of $\cV$.

\begin{definition}\label{D:semistability} \hfill
	\begin{itemize}
		\item A vector bundle $\cV$ on $X$ is semistable (respectively, stable) if for all proper subvector bundles $\cW\le \cV$ we have $\mu(\cW)\le \mu(\cV)$ (respectively $\mu(\cW)< \mu(\cV)$).
		
		\item A $G$-bundle $\cP$ is semistable if the induced adjoint vector bundle $\fg_\cP$ is a semistable vector bundle.	
	\end{itemize}
\end{definition}

\begin{remark}\label{r:semistability associated}
	In general one doesn't define semistability of $G$-bundles through the adjoint representation but rather using reduction to parabolic subgroups and a slope map. In characteristic $0$ all the definitions are equivalent, however in characteristic $p$ they aren't. See \cite{Ram-modI,Schi} for the definitions of semistability. 
	Nevertheless, for an elliptic curve all these definitions coincide due to \cite[Thm 2.1, Cor 1.1]{Sun_sstab}. 
	Therefore, for an elliptic curve, a $G$-bundle is semistable if and only if all associated vector bundles $\cP \times^G V$ are semistable, where $V$ is a highest weight finite dimensional $G$-representation.
\end{remark}

It will also be useful to consider a  notion of semistability for $G$-bundles on certain singular curves, namely the nodal and cuspidal curve $E_{\node}$ and $E_{\cusp}$ considered in the introduction \cref{SS:intro-overview}.\footnote{For such curves, the normalization is isomorphic to $\mathbb{P}^1$, where the only degree $0$-semistable bundle is the trivial bundle.}
We caution the reader that this is only a working definition for us and we don't pretend it is the good notion of semistability over a singular curve.

\begin{definition}
	A $G$-bundle on a singular curve $X$ is said to be semistable if its pullback under the normalization map $\widetilde{X} \to X$ is a semistable $G$-bundle on $\widetilde{X}$.
\end{definition}

\begin{remark}
See \cite[Section 2]{Bhosle-nodal} and \cite{Balaji-singularcurves,Schmitt-principal-nodal} for a more in depth discussion of semistability for $G$-bundles on singular curves. 
\end{remark}

\begin{remark}
	It follows from results of \cite[Thm 3.3.1]{FriedMorgIII} that over a cuspidal or nodal curve the following two conditions are equivalent for a $G$-bundle $\cP$: 
	\begin{enumerate}
		\item $\cP$ is trivial when pulled-back to the normalization $\bP^1$,
		\item for any representation $V$ of $G$ the associated vector bundle $V_\cP$ is slope semistable.
	\end{enumerate}
	This justifies our choice of calling a $G$-bundle on a nodal/cuspidal semistable provided its pull-back to the normalization is trivial.
\end{remark}
\subsection{Moduli spaces and stacks}

Let $X/k$ be a projective curve. 
For each test scheme $S$, we write $\Bun_G(X)(S)$ for the groupoid of $G$-bundles on $X_S = X \times S$. 
This has an open dense substack $\Bun_G^{\ss}(X)$ whose $k$-points consist of semistable $G$-bundles. 

The stack $\Bun_G(X)$ decomposes in to connected components indexed by the algebraic fundamental group $\pi_1(G)$ (the quotient of the cocharacter lattice by the coroot lattice). 
We write $\Bun_G^0(X)$ for the component containing the trivial bundle; such bundles are said to have degree $0$.

Now let $E$ be a genus 1 curve over $k$ with one marked smooth point which may either have a simple node or cusp singularity. 
Thus, $E$ is either a smooth elliptic curve $E$, a genus 1 curve with a single node $E_{\node}$ or a genus 1 curve with a single cusp $E_{\cusp}$.

We write 
$
\uGE := \Bun^0_G(E)^{\ss}
$
for the moduli stack of degree $0$ semistable $G$-bundles on $E$. 
The case when $E$ is an elliptic curve is the main object of study in this paper. 
As noted in the introduction \cref{SS:intro-overview}, the nodal and cuspidal cases may be expressed in more concrete terms as follows.

\begin{proposition} \cite[Thm 3.1.5, Thm 3.2.4]{FriedMorgIII}\label{P:Fried-Morg-nodal-cuspidal}
There is an equivalence of stacks
\[
\underline{G}_{E_\node} \cong G\adjquot G
\]
and 
\[
\underline{G}_{E_{\cusp}} \cong \fg\adjquot G
\]
\end{proposition}

\subsection{Reductions of the structure group}\label{SS:reductions str gr} Suppose we are given $\rho:H\to G$ a morphism of groups and a $G$-bundle $\cP$ on a scheme $E$. 
A reduction of $\cP$ to $H$ is a pair $(\cP',\phi)$ of an $H$-bundle together with an isomorphism of $G$-bundles $\phi\colon\cP'\times^HG\simeq \cP$.
We say that two reduction $(\cP_1,\phi_1)$ and $(\cP_2,\phi_2)$ of $\cP$ are isomorphic if there exists an isomorphism of $H$-bundles $\cP_1\simeq \cP_2$ that intertwines $\phi_1$ and $\phi_2$.

The collection of possibles reductions of a $G$-bundles $\cP$ to an $H$-bundle forms naturally a groupoid. If moreover $\rho$ is injective, this groupoid is equivalent to a set as can be easily checked.

An alternative way of giving a reduction of $\cP$ to $H$ is to give a section $s\colon E\to \cP/H$ of the bundle\footnote{if $\rho$ is not injective this bundle is actually a gerbe with fiber $B(\ker\rho)$} $\cP/H\to E$. 
Indeed, to such a section we associate the pullback $\cP'$ together with $\phi\colon\cP'\rightarrow \cP$:
\[ \begin{tikzcd}
\cP' \ar[r,"\phi"]\ar[d]&\cP\ar[d]\\
E\ar[r,"s"] & \cP/H
\end{tikzcd} \]
Since $\cP\to\cP/H$ is an $H$-bundle, $\cP'\to E$ is one as well. Moreover, the map $\phi$ is $H$-equivariant and induces a $G$-equivariant isomorphism $\cP'\times^H G\to \cP$.

To say that two such sections $s_1$ and $s_2$ are isomorphic we need to look at $E$ and $\cP/H$ as groupoids (the first one is discrete but the second one is non-discrete when $\rho$ is not injective). 
Then $s_1$ and $s_2$ are isomorphic if there exists a natural transformation $\eta\colon s_1\Rightarrow s_2$ that is an isomorphism. 
In the case of a subgroup $H\le G$ the groupoids are discrete and saying that $s_1$ and $s_2$ are isomorphic is the same as saying they are equal (as functions).

Conversely, if $(\cP',\phi)$ is a reduction of $\cP$ to $H$ then quoting out by $H$ the composition $\cP'\stackrel{f\mapsto (f\times 1)}{\longrightarrow} \cP'\times^H G\stackrel{\phi}{\simeq} \cP$ we get a section $E\to \cP/H$.

The above correspondence is an equivalence of categories and we will use freely either way of looking at a reduction.

The most important cases for us are $H\le G$ and $B\twoheadrightarrow T$.

\begin{remark}
	From the above discussion we have that $\Bun_H(E)\to\Bun_G(E)$ is a representable morphism of stacks when $H$ is a subgroup of $G$. 
	This is used silently  throughout the text.
\end{remark}

\subsection{Framed bundles}\label{S:framed bundles}
We start by recalling some basic notions for framed bundles over a pointed projective curve $(E,x_0)$ and then consider the moduli stack $\GE$ of framed degree 0 semistable $G$-bundles.
We show, through \cref{L:automorphisms framed G bundle land inside G}, that it is a variety that is a $G$-bundle over the corresponding non-framed moduli stack. 
Then we proceed to the main point of this section, namely we show that for a closed subgroup $H\le G$ we have a closed embedding $\HE\hookrightarrow\GE$ (see \cref{P:framed H semistable injects in G semistable}).

A framed $G$-bundle $p=(\cP_G,\theta)$ is a $G$-bundle on $E$ together with a $G$-equivariant isomorphism $\theta\colon\cP_G|_{x_0}\simeq G$ of the fiber of $\cP_G$ over $x_0$ with the group $G$.
The degree and semistability are defined in terms of the underlying $G$-bundle. 

Sometimes, for convenience, we will omit the mention of $\theta$ and simply say that $p$ is a framed $G$-bundle. 

It is clear that the moduli stack of framed $G$-bundles $\Bun_G^{\fr}(E)$ is a $G$-torsor over the moduli stack of $G$-bundles $\Bun_G(E)$.

Let $\cP_G$ be a $G$-bundle and $\theta,\theta'$ two framings. 
Then there exists a unique $g\in G$ such that $\theta' = g\cdot\theta$, more precisely such that the following diagram commutes
\[ 
\xymatrix{
	\cP_G|_{x_0}\ar[r]^\theta \ar[dr]_{\theta'} & G \ar[d]^{g\cdot }\\
	& G}
\]

For $(\cP_G,\theta)$ a framed $G$-bundle  we have an induced map $i_\theta\colon\Aut(\cP_G)\to G$ defined by restricting the automorphism to the fiber $\cP_G|_{x_0}$ and using the basic fact recalled above that any two framings differ by an element of $G$. 
If $(\cP_G,\theta')$ is another framing on $\cP_G$, then the induced map $i_{\theta'}\colon \Aut(\cP_G)\to G$ satisfies $i_{\theta'} = gi_\theta g\inv$, where
$g\in G$ is such that  $\theta'=g\cdot\theta$.

In the case that interests us we have moreover
\begin{lemma}\label{L:automorphisms framed G bundle land inside G}
	Let $(\cP_G,\theta)$ be a degree $0$, semistable, framed $G$-bundle over an elliptic curve $E$. 
	Then the induced morphism $\Aut(\cP_G)\to G$ is injective.
\end{lemma}

To prove \cref{L:automorphisms framed G bundle land inside G} we will make use of a property of degree $0$ semistable vector bundles that holds in any genus. 
Recall that for the purposes of this paper, a line bundle on a singular curve is defined to be semistable of degree $0$ if and only if its pullback to the normalization is semistable of degree $0$.
\begin{lemma}\label{l:semistable lemma}
	Let $X$ be a projective curve, and $\cV$, $\cW$ semistable vector bundles of degree $0$. Then any map $f:\cV \to \cW$ is of constant rank.
\end{lemma}
\begin{proof}
	First let us note that we may reduce to the case when $X$ is smooth. 
	Indeed, given such a map $f:\cV \to \cW$ on $X$, we may pullback to a map $\widetilde{f}:\widetilde{\cV} \to \widetilde{\cW}$ on the normalization $\eta:\widetilde{X} \to X$, and restricting to the fiber of $\widetilde{f}$ at $\widetilde{x} \in \widetilde{X}$ identifies with the fiber of $\eta$ at $x=\eta(x)$.
	
	Now suppose that $X$ is smooth. 
	Then $\ker(f)$ and $\mathrm{im}(f)$ must both have non-positive degree by semistability, and as $\cV$ has degree $0$, we must have that both $\ker(f)$ and $\mathrm{im}(f)$ have degree $0$. 
	Now, if $\mathrm{coker}(f)$ had torsion then its preimage would be a positive degree subbundle in $\cV$ contradicting semistability. Thus we must have that $f$ is of constant rank as required. 
\end{proof}

\begin{proof}[Proof of \cref{L:automorphisms framed G bundle land inside G}]
	Let us first start with $G=\GL_n$, so $\cP_G$ is equivalent to a semistable vector bundle, say $\cV$, of degree $0$. 
	
	We'll show that restriction to the fiber at $x_0$ gives an inclusion $\Aut(\cV)\hookrightarrow \GL_n$. 
	Suppose there is an automorphism $\psi$ in the kernel. 
	Then $\psi-\Id$ is an endomorphism of $\cV$ which has a zero at $x_0$. Thus by \cref{l:semistable lemma}, $\psi-\Id$ must be identically zero as required.
	
	Back to the general case. 
	Consider $G\hookrightarrow \GL(V)$ a faithful highest weight representation. 
	Then we know from \cite[Cor 1.1, Thm 2.1]{Sun_sstab} that $V_\cP$ is a semistable, degree 0 vector bundle. 
	
	Therefore by the above, its automorphisms lie inside $\GL(V)$. 
	
	We have the following commutative diagram 
	
	\[ 
	\xymatrix{
		\Aut(\cP) \ar[r]^-{=} & \Hom_{G\mhyphen\equi}(\cP,G)\ar@{^(->}[r] \ar[d] &  \Hom_{G\mhyphen\equi}(\cP,\GL(V))\ar@{^(->}[d] \ar[r]^-{=} & \Aut(V_\cP)\\
		& G\ar@{^(->}[r] & \GL(V) &}.
	\]
	From the injectivity of the three marked maps we deduce that the first vertical morphism is  injective as well, which is what we wanted.
\end{proof}
\begin{remark}
	\cref{L:automorphisms framed G bundle land inside G} is also true for higher genus curves.
	Given a semistable $G$-bundle $\cP$ of degree $0$ one needs to pick up a faithful representation $V$ of $G$ such that $V_\cP$ is semistable.
	For example, one of the fundamental representations will do.
	It is known that $V_\cP$ can become unstable (in characteristic $p$) only if $V$ contains a Frobenius twist and this ensures the semistability of $V_\cP$ for this choice of $V$.
	However, we'll never use this result and so we don't provide the details.
\end{remark}

When $(E,x_0)$ is a pointed, irreducible projective curve of arithmetic genus 1 we denote by $\GE$ the moduli stack of degree 0, semistable framed $G$-bundles over $E$ where the framing is at $x_0\in E$.
\cref{L:automorphisms framed G bundle land inside G} shows that $\GE$ is a variety that is a $G$-torsor over the stack $\uGE$.

Let now $H\le G$ be a closed subgroup of $G$ and consider the induction map from $H$-bundles to $G$-bundles.
First we show that the induction preserves semistability:

\begin{lemma}\label{L:induced of semistable is semistable}
	The map $\Bun_H^0(E)\to \Bun_G^0(E)$ sends semistable bundles to semistable bundles.
\end{lemma}
\begin{proof}
	Let $\cP$ be a semistable $H$-bundle of degree $0$.
	In genus 1 we can use the definition of semistability through associated vector bundles (see \cref{r:semistability associated}) and this simplifies the argument.
	
	Let $V$ be a representation of $G$ and restrict it to $H$. 
	Then $(\cP\timesH G)\timesG V = \cP\timesH V$ is a semistable vector bundle of degree $0$.	
	Hence $\cP\timesH G$ is a semistable $G$-bundle of degree $0$.	
\end{proof}
\begin{remark}
	Note that the above lemma is false without assuming degree $0$.
	For example, for $T=\bG_m^2\hookrightarrow \GL(2)$ we have that $(\cO(-1), \cO(1))$ is a semistable $T$-bundle of degree $(-1,1)$ but the induced vector bundle $\cO(-1)\oplus\cO(1)$ is not semistable even if of degree $0$.
\end{remark}
The following proposition, used silently in the sequel, is conceptually important in understanding the partition of the moduli stack $\GE$ in terms of subsgroups and it paves the way to the Jordan--Chevalley decomposition as formulated in \cref{thm:jordan partition intro}.

\begin{proposition}\label{P:framed H semistable injects in G semistable}
	For any closed subgroup $H\le G$ the induction map between moduli spaces of framed bundles $\HE\to \GE$ is a closed embedding.
\end{proposition}

To prove \cref{P:framed H semistable injects in G semistable}, we will make use of the following corollary of \cref{l:semistable lemma} and of an additional technical lemma on equivariant embeddings:
\begin{lemma}\label{L:line bundles inside semistable}
	Suppose $X$ is a projective curve and $\cV$ a semistable vector bundle of degree $0$.
	\begin{enumerate}
		\item Suppose $\cL$ is a degree $0$ line bundle on $X$. 
				Then any injective map of sheaves $\cL \to \cV$ is necessarily the inclusion of a subbundle.
		\item Suppose $\cL_1,\cL_2$ are line sub-bundles of $\cV$ such that $\cL_1|_x = \cL_2|_x$ for some $x\in X$.
				Then $\cL_1 = \cL_2$.
	\end{enumerate}
\end{lemma}
\begin{proof}
	Part 1 follows immediately from \cref{l:semistable lemma}, and part 2 follows by considering the canonical morphism $\cL_1 \oplus \cL_2 \to \cV$.
\end{proof}

\begin{lemma}\label{L:embed G/H with compl Cartier}\footnote{We thank M. Brion for providing us the proof.}
	Let $H\le G$ be a closed subgroup.
	Then we can find a representation $V$ of $G$ and a line $L\subset V$ such that
	\begin{enumerate}
		\item $\Stab_G(L)=H$,
		\item $G/H\hookrightarrow \bP(V)$ equivariantly,
		\item the complement $\overline{G/H} \setminus (G/H)$ is empty or the support of a Cartier divisor.
	\end{enumerate}	
\end{lemma}
\begin{proof}
	By Chevalley's theorem we can find a representation $W$ and a line $L\subset W$ with properties (i), (ii) above.
	If $H$ is a parabolic subgroup then we're done.
	If not, put $X:=\overline{G/H}$ and denote by $Z$ the complement $X\setminus (G/H)$.
	Let $\tilde{X}$ be the normalization of the blow-up of $X$ along $Z$.
	The $G$-action on $X$ extends by universal properties to $\tilde{X}$ and the complement of $G/H$ in $\tilde{X}$ is, by construction, the support of a Cartier divisor.
	
	Now we use Sumihiro's theorem \cite{Sumi_equiv_emb} (or \cite[Theorem 5.3.3.]{Brion_linearization}) to embed $\tilde{X}\hookrightarrow \bP(V)$ equivariantly in the projectivization of a representation $V$ of $G$.
\end{proof}

\begin{proof}[Proof of \cref{P:framed H semistable injects in G semistable}]
	
By \cref{L:embed G/H with compl Cartier} we may pick a representation $V$ of $G$ and a line $L\subseteq V$ such that the morphism $g \mapsto g \cdot L$ defines an equivariant embedding in to the projective space $G/H \hookrightarrow \bP(V)$ and such that the complement of $G/H$ in its closure is the support of a Cartier divisor.

Now suppose $\cP_G$ is a $G$-bundle on $E$ and put $\cV:=V_{\cP_G}$. 
We have an embedding of associated bundles:
\[
\cP_G/H=\cP_G \times^G (G/H) \hookrightarrow \cP_G \times^G \bP(V) = \bP(\cV).
\]
The data of an $H$-reduction $\cP_H$ of $\cP_G$ is equivalent to a section $s:E\to \cP_G/H$ which we will consider as a section of the projective bundle $\bP(\cV)$ via the embedding above. In other words, an $H$-reduction corresponds to a certain line subbundle $\cL = \cP_H \times^H L \hookrightarrow \cV$. Note that if $\cP_G$ and $\cP_H$ are semistable and of degree $0$, then so are $\cV$ and $\cL$.

Let us first show that $\HE \to \GE$ is injective. Given a framed $G$-bundle $p=(\cP_G,\theta)$, we suppose it has two reductions to a framed $H$-bundle, corresponding to two line subbundles $\cL_1,\cL_2$ as explained above.  Under the framing isomorphism
\[
\cV|_{x_0} = \cP_G|_{x_0} \times^G V \cong G\times^G V = V
\]
we have that $\cL_1|_x$ and $\cL_2|_x$ both correspond to the line $L \subseteq V$. Thus, by \cref{L:line bundles inside semistable}, we must have $\cL_1= \cL_2$ as required.

Now let us show that $\HE \to \GE$ is proper. We use the valuative criterion of properness. Thus, let $S$ denote the spectrum of a discrete valuation ring, $U$ the spectrum of its generic point, and $Z$ the spectrum of its closed point, and let $E_S, E_U, E_Z$ be the base change of $E$ to $S$, $U$, and $Z$ respectively. 

Given a family $(\cP_{G,S}, \theta)$ of framed, degree $0$, semistable $G$-bundles on $E_S$ together with a compatibly framed $H$-reduction on $E_U$, we must show that it extends to a compatibly framed $H$-reduction on $E_S$.

As before, we have the associated vector bundle $\cV_S$ over $E_S$, and we record the data of the $H$-reduction as a line subbundle $\cL_U \hookrightarrow \cV_U$. 
We must show that
\begin{enumerate}
	\item $\cL_U$ extends to a line subbundle $\cL_S \subseteq \cV_S$ over $E_S$ and
	\item The corresponding section $E_S \to \bP(\cV_S)$ lands inside the subbundle $\cP_{G,S}/H$.
\end{enumerate}

For the first part, we note that the line subbundle extends to a sub\emph{sheaf} $\cL_S \to \cV_S$. Indeed, by the properness of $\bP(\cV_S) \to E_S$, we may extend the section $s: E_U \to \bP(\cV_U)$ over the generic point of $E_Z$ to define a line subbundle $\cL'$ on $E_S'$, where $E_S'$ is an open subset of $E_S$ whose complement is of codimension 2. This line subbundle is the restriction of a unique line bundle $\cL_S$ on $E_S$ (e.g. by taking the closure of a Weil divisor representing it), and the map $\cL' \to \cV_{S}|_{E_S'}$ necessarily extends over the codimension 2 locus. 

Now we observe that the restriction $\cL_Z$ of $\cL_S$ to $E_Z$ is a degree $0$ line bundle (as it deforms to the degree $0$ line bundle $\cL_U$) with a non-zero map to $\cV_Z$. It follows from \cref{L:line bundles inside semistable} that $\cL_Z \to \cV_Z$ must be a subbundle as required.

To prove the second part, remember that the complement of $G/H$ in $\overline{G/H}\subset \bP(V)$ is the support of a Cartier divisor or empty.

We will show that the section $s: E_S \to \bP(\cV_S)$ defining the line subbundle $\cL_S$ above lands inside $\cP_{G,S}/H$. 
If $G/H$ is closed in $\bP(V)$ then $\cP_{G,S}/H$ is closed in $\bP(\cV_S)$ and so the image of the section is contained entirely in it.

In case $G/H$ is not projective, notice that the complement of $\cP_{G,S}/H$ in its fiberwise closure is the support of a Cartier divisor;
thus the set of points $x\in E_S$ for which $s(x) \notin \cP_{G,S}/H$ is possibly a divisor $C$ in $E_S$. 
As $s(x) \in \cP_{G,S}/H$ for all $x\in E_U$, we must have that $C=E_Z$ (possibly with multiplicity). 
But now note that the framing on $\cP_{G,S}$ defines a trivialization 
\[
\bP(\cV_S)|_{\{x_0\}\times S} \cong \bP(V\otimes \cO_S)
\]
and we have that $s$ takes the constant value $[L]$ along the entire slice $\{x_0\}\times S$ of $E_S$. 
In particular, the value of $s$ is contained inside $\cP_{G,S}/H$ at at least one point in $E_Z=C$, and thus it must be so contained at every point, as required.
\end{proof}

\subsection{The coarse moduli space and the characteristic polynomial map}
The coarse moduli \emph{space} of degree 0, semistable $G$-bundles, which is none other than the GIT quotient $\GE\git G$, was identified by Laszlo \cite{Lasz} to be isomorphic to the GIT quotient $\TE\git W$. 
There is a natural $G$-invariant morphism from the framed moduli stack to the moduli space of $G$-bundles that we think of as the characteristic polynomial map (in analogy to Lie theory)
\begin{align*}
\chi = \chi_G\colon \GE \to \MGE.
\end{align*}

Notice that the $G$-invariance of $\chi$ is equivalent to $\chi$ factorizing through the moduli stack 
\[ 
\underline{\chi}\colon\uGE\to\MGE.
\]

The map $\chi$ maps a semisimple bundle $\cP$ into its equivalence class in $\TE\git W$. 
\begin{definition}\label{D:stack of unip bdls}
	The moduli stack/variety of unipotent $G$-bundles are defined to be the preimages of $\one\in\MGE$:
	 \begin{align*}
	 \uGE^\uni:=&\underline\chi\inv(\one), \\
	 \GE^\uni : = &\chi\inv(\one).
	 \end{align*}
\end{definition}

\section{The partition of semisimple bundles by Lusztig type}\label{sec:the-partition-of-semisimple-bundles-by-lusztig-type}

In this section we construct and study a certain locally closed partition of the coarse moduli space $\MGE$.

Recall that the points of the coarse moduli space $\MGE$ are in bijection with isomorphism classes of semisimple objects of $\uGE$. Let $\cP_G \in \uGE$ be a semisimple object and let $p \in \GE$ be a framed lift. Then the automorphism group $\Aut(\cP_G)$ is identified with the reductive subgroup $\Stab_G(p)$ of $G$. Different choices of framing define conjugate subgroups of $G$, and thus the conjugacy class of $\Stab_G(p)$ in $G$ is a well-defined invariant of the bundle $\cP_G$.

In this way, we may partition $\MGE$ according to the corresponding conjugacy class in $G$ of its automorphism group. It will be convenient to encode the data of $\Stab_G(p)$ in two stages:
\begin{itemize}
	\item The neutral component $H =\Stab_G(p)^\circ$, which is a connected reductive subgroup of $G$.
	\item The component group $\pi_0\Stab_G(p)$, which is a subgroup of the finite group $N_G(H)/H$.
\end{itemize}
The goal of this section is to study this partition of $\MGE$, and express it in combinatorial/root-theoretic terms.

\subsection{$E$-pseudo-Levi subgroups}\label{sec:e-pseudo-levi-subgroups}
Denote by $\cE$ the set of conjugacy classes of connected reductive subgroups of $G$ containing a maximal torus. 
Consider the map (of sets)
\[
h\colon \MGE \to \cE
\]
which takes the isomorphism class of a semisimple bundle $\cP_G$ to the $G$-conjugacy class of $\Stab_G(p)^\circ$, where $p$ is a framed lift of $\cP_G$. We denote by
\[
\MGE_{[H]} := h^{-1}([H])
\]
the fibres of the map $h$.

\begin{definition}\label{D:E-pseudo-Levi}
	We say that a connected reductive subgroup $H$ of $G$ is an \emph{$E$-pseudo-Levi subgroup} if it is of the form $\Stab_G(p)^\circ$ for some semisimple $p\in \GE$. We write $\cE_E \subset \cE$ for the set of conjugacy classes of $E$-pseudo-Levi subgroups of $E$.
\end{definition}

\begin{remark}
	As mentioned in the introduction (see also \cref{P:Fried-Morg-nodal-cuspidal}), if $E$ is cuspidal (respectively, nodal) then $G_E \cong \fg$ (respectively $\GE \cong G$). 
	Thus $E$-pseudo-Levi subgroups correspond to connected centralizers of semisimple elements of $\fg$ (respectively $G$). 
	These are precisely the Levi (respectively, pseudo-Levi) subgroups of $G$ (see \cref{sec:classification-of-elliptic-closed-subsets}).
\end{remark}

In \cref{SS:refined partial order on ell sets} we will give an alternative description of this partition, which allows us to understand the closure relations. 
In particular, at the end of \cref{sec:connecting-e-psuedo-levis-and-e-subsets-of-roots} we will establish the following result (which may also be deduced from the theory of Luna stratifications; see \cref{rem:luna}).

\begin{proposition}\label{prop:locally closed}
For each $E$-pseudo-Levi $H$, the subset $\MGE_{[H]}$ is a locally closed subvariety of $\MGE$.
\end{proposition}

We may record the finer partition according to the conjugacy class of the full automorphism group as follows. If $\cP_G$ is a representative of a point in $\MGE_{[H]}$ and $p$ a framed lift, then $\Stab_G(p) \subseteq N_G(H)$ (as every algebraic group normalizes its own neutral component). Thus the component group of $\Stab_G(p)$ is naturally a subgroup of the relative Weyl group:
\[
\pi_0 \Stab_G(p) = \Stab_G(p)/H \subseteq W_{G,H} := N_G(H)/H.
\]
Given a subgroup $A$ of $W_{G,H}$, we write $\MGE_{[H,A]}$ for the subset of $\MGE_{[H]}$ corresponding to semisimple bundles $\cP_G$ such that the component group of $\Aut(\cP_G)$ is identified with a conjugate of $A$ as above. (The pair $[H,A]$ is defined up to simultaneous conjugation in $G$.)

In particular, the subset $\MGE_{[H,1]}$ consists of isomorphism classes of semisimple bundles $\cP_G$ whose full automorphism group $\Aut(\cP_G)$ is identified with the connected group $H$.

\begin{remark}\label{rem:luna}
Suppose we are given a $G$-variety $Y$ such that every point has a $G$-invariant affine chart. Then we have the categorical quotient $Y\git G$ whose points are in bijection with closed $G$-orbits on $Y$. We may partition $Y\git G$ according to the conjugacy class in $G$ of the (necessarily reductive) stabilizer of the corresponding closed orbits. One can show that this defines a locally closed partition (\cite{luna_slices_1973}, see e.g. \cite{kuttler_is_2008} for an overview). In the case of $G$ acting on the framed moduli space $\GE$, this reproduces the partition $\MGE_{[H,A]}$ as described above. For our purposes it is convenient to focus mainly on the coarser partition according to connected reductive subgroups $H$, hence the choice of notation.
\end{remark}

\subsection{Borel--de Siebenthal theory}\label{SS:thm of Borel de Siebenthal}
Fix a maximal torus $T$ of $G$, and let $\Phi$ denote the corresponding set of roots. A subset $\Sigma \subseteq \Phi$ is called \emph{closed} if $\bZ \Sigma \cap \Phi = \Sigma$. We denote by $\cA$ the set of closed subsets of $\Phi$.

The connection between closed subsets of roots and connected reductive subgroups is given by the following theorem of Borel--de Siebenthal:

\begin{theorem}\label{T:Borel de Siebenthal}\cite{BdS}
	The collection of connected reductive subgroups $H\le G$ that contain the maximal torus $T$ is in bijection with $\cA$. The correspondence is given by associating to $H$ its root system and conversely, to $\Sigma\in\cA$ the subgroup $C_G(Z(\Sigma))^\circ$. 
	
	Moreover, under this correspondence we have $G(\Sigma)=C_G(Z(\Sigma))^\circ$ and 
	\begin{align}\label{E:center of C_G is Z_Sigma}
	Z(G(\Sigma))=Z(\Sigma)=Z(C_G(Z(\Sigma)))
	\end{align}
\end{theorem}
\begin{proof}
	Apart the last equality in \cref{E:center of C_G is Z_Sigma}, all is part of Borel--de Siebenthal theory \cite[Th\'eor\`eme 4 and Section 6, page 213]{BdS}.
	
	For the last equality, it is enough to show that $Z(C_G(Z(\Sigma)))$ is contained in the neutral component $C_G(Z(\Sigma))^\circ$. 
	If $h$ is in $Z(C_G(Z(\Sigma)))$, then $h\in C_G(T)=T\subset C_G(Z(\Sigma))^\circ$. 
	Since $Z(C_G(Z(\Sigma))^\circ)=Z(\Sigma)$ we deduce $Z(C_G(Z(\Sigma)))\subset C_G(Z(\Sigma))^\circ$ which moreover implies $Z(C_G(Z(\Sigma)))\subset Z(\Sigma)$. 
\end{proof}

\subsection{$E$-root subsystems}\label{sec:closed-subsets-of-e-type}
In this subsection, we will present an alternative approach to the theory of $E$-pseudo-Levi subgroups in terms of their associated root data.

Recall that $\TE$ denotes the algebraic group parameterizing framed $T$-bundles on the fixed curve $E$. Thus $\TE$ is isomorphic to either $\ft$ or $T$ in the cuspidal and nodal cases respectively. In general 
\[
\TE \cong \Hom(\bX^*(T),J(E))
\]
where $J(E)$ denotes the Jacobian variety of $E$.

Note that each character $\alpha \in \bX^*(T)$ gives rise to a homomorphism $\alpha_\ast\colon\TE \to J(E)$, taking a $T$-bundle to its associated line bundle. For $p \in \TE$, we set 
\[
\Sigma_p := \left\lbrace \alpha\in\Phi\mid p \in \ker(\alpha_\ast) \right\rbrace.
\]

We write $\cA_E$ for the subset of $\cA$ consisting of subsets $\Sigma_p\subset\Phi$ which occur in this way. The elements of $\cA_E$ are called $E$-root subsystems of $\Phi$. 

Thus we have a map 
\[
\kappa\colon\TE \to \cA
\]
which assigns to a point $p \in \TE$ the set $\Sigma_p$. 
The image is $\cA_E$ by definition. 

\begin{lemma}\label{l:locally closed torus}
	The map $\kappa\colon\TE \to \cA_E$ is continuous with respect to the topology on $\cA_E$ induced by the partial order given by inclusion. In other words the partition
	\[
	\TE = \bigsqcup_{\Sigma \in \cA_E} (\TE)_{\Sigma}
	\]
	is locally closed.
\end{lemma}
\begin{proof}
	It suffices to show that the preimage
	\[
	(\TE)_{\geq \Sigma} = \kappa^{-1}(\{\Sigma' \mid \Sigma' \supseteq \Sigma\})
	\]
	is closed. But this subset is just the intersection of root hyperplanes $\ker(\alpha_\ast) \subseteq \TE$ for $\alpha \in \Sigma$.
\end{proof}

As we will see in \cref{sec:connecting-e-psuedo-levis-and-e-subsets-of-roots}, the partition of $T_E$ according to $E$-root subsystems records the conjugacy class of the neutral component the stabilizer of $G$ acting on $T_E$. We may also record the component group of the stabilizer as follows.

Let $W=W_{G,T}$ be the Weyl group. Let $\Sigma \subseteq \Phi$ be an $E$-root subsystem. We write $N_W(\Sigma) \subseteq W$ for the normalizer of $\Sigma$. Let $W_\Sigma$ denote the Weyl group of $\Sigma$ (considered as a root system in its own right). 
We have that $N_W(\Sigma)=N_W(W_\Sigma)$ and so $W_\Sigma$ is a normal subgroup in $N_W(\Sigma)$.

\begin{lemma}\label{l:weyl stabilizer}
	Let $p \in T_E$ and set $\Sigma = \Sigma_p$. Then 
	\[
	W_\Sigma \subseteq \Stab_W(p) \subseteq N_W(\Sigma).
	\]
	We write $A_p$ for the corresponding subgroup of $W_{G,\Sigma} := N_W(\Sigma)/W_{\Sigma}$.
\end{lemma}

Thus to each $p \in T_E$ we have a pair $(\Sigma_p, A_p)$ consisting of an $E$-root subsystem $\Sigma_p$ and a subgroup $A_p$ of $W_{G,\Sigma}$.

\begin{proof}
First note that, by definition of $\Sigma$, $p$ is contained in each of the root hyperplanes corresponding to roots in $\Sigma$. Thus $p$ is fixed by the corresponding reflections in $W_\Sigma$ and thus by all of $W_\Sigma$. This proves the inclusion on the left.

Now let $w\in \Stab_G(p)$, and suppose $\alpha \in \Sigma$. Then 
\[
w(\alpha)_\ast(p) = \alpha_\ast(w\inv p) = \alpha_\ast(p) = \one_{J(E)}
\]
Thus $w\in N_W(\Sigma)$, establishing the inclusion on the right.
\end{proof}

\subsection{Connecting $E$-pseudo-Levis and $E$-root subsystems of roots}\label{sec:connecting-e-psuedo-levis-and-e-subsets-of-roots}
Recall that we have fixed a maximal torus $T\subseteq G$, with associated roots $\Phi$ and Weyl group $W$.

It follows from Borel-de-Siebenthal theory (see \cref{SS:thm of Borel de Siebenthal}) that the assignment $\Sigma \mapsto [G(\Sigma)]$ defines an order preserving bijection between $\cE$ and $\cA/W$. 
The following result states that $\Sigma$ is an $E$-root subsystem if and only if $G(\Sigma)$ is an $E$-pseudo-Levi subgroup.

\begin{proposition}
	The assignment $\Sigma \mapsto G(\Sigma)$ defines an order-preserving bijection
	\[
	\cA_E/W \xrightarrow{\sim} \cE_E.
	\]
\end{proposition}

This follows immediately from the following result:
\begin{lemma}\label{l:automorphisms and root systems}
	Let $p \in \TE $ and view it inside  $\GE$. 
	Then $\Stab_G(p)^\circ = G(\Sigma_p)$.
\end{lemma}
\begin{proof}
	To show that these two connected subgroups of $G$ are equal, it suffices to show that their corresponding Lie algebras are equal inside $\fg$.
	Put $p=(\cP_G,\theta)$. 
	The Lie algebra of $\Stab_G(p)^\circ \cong \Aut(\cP_G)$ is given by the global sections of the adjoint bundle $H^0(E;\fg_{\cP})$. 
	As $\cP_G$ is induced from a $T$-bundle $\cP$, the adjoint bundle $\fg_{\cP_G}$ splits as a direct sum 
	\[
	\fg_{\cP_G} = \fg_{\cP} = \ft_{\cP} \oplus \bigoplus_{\alpha \in \Phi} \fg_{\alpha,\cP}.
	\]
	For each $\alpha \in \Phi$, the line bundle $\fg_{\alpha,\cP}$ is trivial precisely when $\alpha \in \Sigma_p$. 
	In this case, the framing provides a canonical identification 
	\[
	H^0(E;\fg_{\alpha,\cP}) = \fg_\alpha.
	\]
	Similarly, we have $H^0(E;\ft_{\cP})= \ft$.
	On the other hand, if $\alpha \notin \Sigma_p$, then $H^0(E;\fg_{\alpha,\cP}) = 0$ because a line bundle of degree $0$ on a curve has a section if an only if it's trivial.
	Thus we have that 
	\[
	H^0(E; \fg_{\cP}) = \fg(\Sigma_p)
	\]
	as required.
	\end{proof}

Now recall that the map 
\[
q\colon\TE \to \MGE
\]
identifies $\MGE$ with the categorical quotient $\TE\git W$. 
Putting all this together, we have 

\begin{lemma}\label{l:comparing two descriptions of partition}
	The following diagram commutes:
	\[
	\xymatrix{
		\TE \ar[r]^\kappa \ar[d]_q & \cA_E \ar[d] \\
		\MGE \ar[r]_-h & \cE_E
	}
	\]
	\end{lemma}
This proves \cref{prop:locally closed} in view of \cref{l:locally closed torus}.

\subsection{The component group}
Given a semisimple bundle in $\cP_G \in \uGE$, we have shown that the neutral component $\Aut(\cP_G)^\circ$ may be identified with $G(\Sigma_p)$ where $\Sigma_p$ is the set of roots which annihilate a given framed lift $p \in \TE$ of $\cP_G$.

We will now refine this to give a combinatorial description (i.e. in terms of the $W$-action on $T_E$) of the component group of $\Stab_G(p) \ (\cong \Aut(\cP_G))$.

Recall that for a connected reductive subgroup $T \subseteq H \subseteq G$ with root subsystem $\Sigma \subseteq \Phi$, we have an isomorphism
\[
W_{G,H} := N_G(H)/H \cong N_W(\Sigma)/W_\Sigma := W_{G,\Sigma}.
\]
This finite group is referred to as the relative Weyl group of $H$ (or of $\Sigma$) in $G$. 
It naturally acts on the algebraic group $Z(H)_E$ (also written $Z(\Sigma)_E$). 
The following result identifies the component group of a semisimple bundle in $\MGE_{[H]}$ with the stabilizer in the relative Weyl group of a corresponding lift to $p \in Z(H)_\Sigma$. 
\begin{lemma}\label{lem:components}
Let $p\in \TE$, and let $H=G(\Sigma_p)$. Then there are compatible identifications
\[
\xymatrix{
	&\ar[d]_\wr \Stab_{W_{G,\Sigma}}(p) \ar@{^{(}->}[r]  & \ar[d]^\wr N_W(\Sigma)/W_\Sigma \\
	\pi_0\Stab_G(p)  &\ar[l]_{\sim} \Stab_G(p)/H \ar@{^{(}->}[r] & N_G(H)/H.
}
\]
\end{lemma}
\begin{proof}
	Note that we have a commutative diagram with exact rows
	\[
	\xymatrix{
1 \ar[r] &\ar[d] N_H(T) \ar[r] &\ar[d] \Stab_{N_G(H)\cap N_G(T)}(p) \ar[r]&\ar[d] \Stab_{N_W(\Sigma)}(p)/W_\Sigma \ar[r]& 1 \\
1 \ar[r] &  H \ar[r] &\Stab_G(p)/H \ar[r]& \Stab_G(p)/H \ar[r]& 1	
}
\]
It is a straightforward diagram chase to show that the right most vertical morphism is an isomorphism. Similarly, one checks that this isomorphism is compatible with the embeddings as required.
\end{proof}

\subsection{Closure relations for the partition}\label{SS:refined partial order on ell sets}
We shall see presently that the closure of the variety $(\TE)_\Sigma$ is a union of varieties $(\TE)_{\Sigma'}$. However, the $\Sigma'$ that appear in the closure relation are determined by a slightly modified partial order relation which we determine below.

Given $\Sigma \in \cA_{E}$ we consider the subgroup $L(\Sigma) := C_G(Z(\Sigma)^\circ)$ where $Z(\Sigma) = \cap_{\alpha\in\Sigma}\ker(\alpha)$. 
It is a Levi subgroup of $G$ (as it is the centralizer of a torus) which contains $G(\Sigma)$ (as every element of $G(\Sigma)$ centralizes $Z(\Sigma)^\circ$). In fact, it is the smallest such subgroup. We write $\Sigma^\circ$ for the set of roots of $L(\Sigma)$. We may now define a new partial order on $\cA$.

\begin{definition}\label{D:refined partial order}
Given $\Sigma_1, \Sigma_2 \in \cA$, we define the partial order $\succeq$ by
\[\Sigma_1\succeq \Sigma_2 \text{ if } \Sigma_1 \supseteq \Sigma_2 \text{ and } \Sigma_2 = \Sigma_1 \cap \Sigma_2^\circ.\]
We say that a closed subset is \emph{isolated} if it is maximal with respect to this partial order.
\end{definition}
In other words, maximal subsets $\Sigma$ are characterized by the fact that $Z(\Sigma)^\circ = Z(G)^\circ$. 
This notion is useful in order to relate our partition/stratification to the one of Lusztig \cite[3.1]{Lu_gen_spr}

\begin{proposition}\label{P:closure of TE_Sigma in terms of refined order}
	We have the closure relation
	\[
\overline{(\TE)_\Sigma} = \bigsqcup_{\Sigma' \succeq \Sigma} (\TE)_{\Sigma'}.
\]
\end{proposition}
\begin{proof}
First we claim that we have the following description of the closure:
\begin{equation}\label{eq:closure}
\left\{(\TE)_\Sigma\right\}^- = (Z(\Sigma)^\circ)_E \cdot (\TE)_\Sigma.
\end{equation}
Indeed, note that $(\TE)_\Sigma$ is open in $Z(\Sigma)_E= (\TE)_{\ge\Sigma}$, and thus its closure is necessarily a union of connected components of $Z(\Sigma)_E$. 
In particular, the closure must be a union of orbits for the neutral component
$\left(Z(\Sigma)_E\right)^\circ = (Z(\Sigma)^\circ)_E
$, from which the claim follows.

Now suppose $p \in \TE$ is in the closure of $(\TE)_\Sigma$. We must show that 
\[
\Sigma = \Sigma_p \cap \Sigma^\circ.
\]

As $(\TE)_{\geq \Sigma}$ is closed and contains $(\TE)_\Sigma$, it must also contain $\overline{(\TE)_\Sigma}$ and thus $p$. In other words, we have $\Sigma \subseteq \Sigma_p$. By construction we have $\Sigma \subseteq \Sigma^\circ$ and hence $\Sigma \subseteq \Sigma_p \cap \Sigma^\circ$. 

It remains to show the other inclusion. Let $\alpha \in \Sigma_p \cap \Sigma^\circ$.
By \cref{eq:closure} we may write
\[
p = q \cdot r 
\]
where $q\in Z(\Sigma)^\circ_E$ and $r \in (\TE)_\Sigma$. 
As $\alpha \in \Sigma_p$, $\alpha_\ast(p) = \one_{J(E)}$; 
as $\alpha \in \Sigma^\circ$, $\alpha_\ast(q)= \one_{J(E)}$. 
Thus we have that $\alpha_\ast(r) = \one_{J(E)}$ which implies $\alpha\in\Sigma$ as required.
\end{proof}

\subsection{Summary of section}\label{SS:summary partition and strata}
Recall that we have defined a locally closed partition in two ways
\begin{align}\label{Eq:partition of moduli space index elliptic sets}
\MGE =& \bigsqcup_{[\Sigma]\in\cA_E/W} \MGE_{[\Sigma]}\\
		= & \bigsqcup_{[H]\in \cE_E} \MGE_{[H]}\nonumber
\end{align}
We thus obtain a locally closed partition 
\begin{align}\label{Eq:partition of moduli stack index elliptic sets}
\uGE &= \bigsqcup_{[\Sigma]\in\cA_E/W} (\uGE)_{[\Sigma]}\\
		&=\bigsqcup_{[H]\in\cE_E} (\uGE)_{[H]}\nonumber
\end{align}
by pulling back via the characteristic polynomial map $\underline\chi\colon\uGE \to \MGE$. 
Similarly for the framed version $\GE$ using $\chi\colon\GE\to\MGE$.

A stratum that plays a special role for us is the one corresponding to $H=G$. 
We put 
\[ (\uGE)_\heartsuit:=(\uGE)_{[G]}. \]
One should think of this locus as those $G$-bundles whose semisimplification "is central" (i.e. has a reduction to the center $Z(G)$ of $G$).

Note that the set $\cE$ carries partial orders $\leq, \preceq$ induced from the same named orders on $\cA$ (see \cref{D:refined partial order}).

\begin{proposition}
	The partitions from \eqref{Eq:partition of moduli space index elliptic sets}, \eqref{Eq:partition of moduli stack index elliptic sets} of $\MGE$ and $\GE$ have the closure relations determined by the partial order $\preceq$.
\end{proposition}
\begin{proof}
	Follows from \cref{l:comparing two descriptions of partition} and \cref{P:closure of TE_Sigma in terms of refined order}
\end{proof}

\begin{remark}
	The upshot of this section is that there are two approaches to determining the type of a semisimple bundle: either compute its automorphism group, or choose a reduction to $T$ and compute the subset of roots on which the bundle vanishes.
\end{remark}

\section{The Jordan--Chevalley Theorem}\label{sec:the-jordan--chevalley-theorem}
In this section we will prove \cref{thm:jordan intro} and \cref{thm:galois intro} from the introduction. 
A key concept here is the notion of \emph{regularity} which we define in \cref{sec:the-regular-locus}. 
Then we will establish the equivalence of the two main theorems in \cref{sec:equivalence-of-theorems-refthmjordan-and-refthmgalois}. 
Finally we will prove \cref{thm:galois intro} over \cref{sec:proof-of-surjectivity} and \cref{sec:proof-of-injectivity}.

\subsection{The regular locus}\label{sec:the-regular-locus}
Fix $H$ an $E$-pseudo-Levi subgroup of $G$. 
Recall that this means that there exists a semisimple framed $G$-bundle $q\in \GE$ such that $\Stab_G(q)^\circ=H$.

\begin{definition}\label{D:regular and strongly regular}
We say that $p \in \HE$ is
	\begin{enumerate}
		\item \emph{regular} if $\Stab_G(p)^\circ \subseteq H$,
		\item \emph{strongly regular} if $\Stab_G(p) \subseteq N_G(H)$,
		\item \emph{maximally regular} if $\Stab_G(p) \subseteq H$.
	\end{enumerate}
\end{definition}
One may immediately check that for a given $p\in \HE$ we have implications
\[
\text{maximally regular} \implies \text{strongly regular} \implies \text{regular}
\]
(the second implication uses that $N_G(H)^\circ=H$).

If the stabilizers of semisimple elements of $\GE$ are always connected, then all three notions above coincide. For example this happens when $E$ is cuspidal (so $\GE = \fg$) or when $E$ is nodal and $G = \GE$ is simply connected. However, in general, the three notions are all distinct as illustrated by the following example.

\begin{example}\label{ex:strongly regular}
Consider the case where $E$ is a nodal curve, and thus we may identify $\GE=G$. Assume also that $char(k)=0$. 
	\begin{enumerate}
	\item Let $G=PGL_2$ and $T$ the maximal torus represented by the classes of diagonal matrices. Consider the matrix
	\[
	X=
	\begin{pmatrix}
	1&0\\
	0&-1
	\end{pmatrix}
	\]
	One can check that $\Stab_{G}([A]) = N_G(T)$. Thus $[X]$ is a strongly regular (and thus regular) element of $T$, but it is not maximally regular .
	\item Now let $G=PGL_3$ and $H$ the Levi subgroup consisting of classes of block matrices of the form:
	\[
\begin{pmatrix}
\begin{matrix}
\ast & \ast \\
\ast & \ast
\end{matrix}
& \rvline & \begin{matrix}0\\0\end{matrix} \\
\hline
\begin{matrix}0&0\end{matrix} & \rvline &
\begin{matrix}
\ast
\end{matrix}
\end{pmatrix}
\]
We write $T$ for the diagonal maximal torus again.

Consider the matrix
\[
Y = 	
\begin{pmatrix}
1&0&0\\
0&\zeta&0\\
0&0&\zeta^2
\end{pmatrix}
\]
where $\zeta$ is a primitive third root of unity.
Then $\Stab_{G}([Y])$ is the preimage in $N_G(T)$ of the cyclic subgroup group $A_3 \subseteq S_3 \cong N_G(T)/T$. 

In particular 
\[
\Stab_{G}([Y])^\circ = T \subseteq H
\]
and thus $[Y]$ is a regular element of $H$. However, 
\[
\Stab_G([Y]) \nsubseteq H = N_G(H),
\]
so $[Y]$ is not a strongly regular element of $H$ (it is however, a strongly but not maximally regular element of $T$).
\end{enumerate}
\end{example}
We write $\HE^\reg$ (respectively $\HE^\strreg$, respectively $\HE^\maxreg$) for the locus of regular (respectively strongly regular, respectively maximally regular) elements. As these loci are manifestly $H$-invariant (in fact, $N_G(H)$-invariant) we have corresponding loci $\uHE^\reg$, $\uHE^\strreg$, $\uHE^\maxreg$ in the stack $\uHE$.

\begin{remark}\label{rem:subscript-vs-superscript}
Whereas the loci $(\uGE)_{[H]}$ and $(\uGE)_{[H,A]}$ introduced in \cref{sec:the-partition-of-semisimple-bundles-by-lusztig-type} are intrinsic to $G$, the regular locus $(\uHE)^\reg$ and its relatives are defined in terms of how $H$ sits as a subgroup of $G$ (i.e. are not intrinsic to $H$). In what follows, we will often need to consider both the intrinsic loci of $\uHE$ (such as $(\uHE)_{[K]}$ for some $E$-pseudo-Levi $K$ of $H$) alongside the regular loci. To keep track of these notions, we will always label intrinsic loci in the subscript and regularity conditions in the superscript.
\end{remark}

It will be useful to have a few other characterizations of the regularity condition. To state the result we will need to recall some notation.

Let $\cP_H \in \uHE$. Choose an element $p \in \TE$ lifting $\chi(\cP_H)$ under the map
\[
\TE \to \MHE = \TE \git W_H.
\]
This element determines a closed root subsystem $\Sigma_p \subseteq \Phi$ (see \cref{sec:closed-subsets-of-e-type}). Note that we consider $\Sigma_p$ as a subset of the roots $\Phi$ of $G$ and it might not be contained in the root subsystem $\Sigma_H$ corresponding to $H$.

Recall also that there is a unique-up-to-isomorphism semisimple bundle $\cP_H^\ss$ with $\chi(\cP_H^\ss) = \chi(\cP_H)$. One can construct $\cP_H^\ss$ by choosing a reduction $\cP_{B_H}$ to a Borel $B_H$ of $H$, and taking the induced bundle $\cP_T$ under the projection $B_H \to T$, then inducing $\cP_T$ back up to an $H$-bundle. Moreover, we may choose the $B_H$-reduction compatibly with the lift $p\in \TE$ above, so that $p$ is a framed lift (unique-up-to-isomorphism) of $\cP_T$.

\begin{proposition}\label{prop:regular}
	Let $\cP_H \in \uHE$, and choose $p \in \TE$ and $\cP_H^\ss$ as described above. The following are equivalent:
	\begin{enumerate}
		\item $\cP_H$ is regular, \label{E:PH regular}
		\item $\cP_H^{\ss}$ is regular,  \label{E:PHss regular}
		\item The morphism $\pi\colon\uHE \to \uGE$ is \'etale at $\cP_H$, \label{E:HEtoGE is etale at PH}
		\item $\rH^\bullet(E;(\fg/\fh)_{\cP_H})=0$,  \label{E:coh of adjoint of PH is zero}
		\item $\Sigma_{p} \subseteq \Sigma_H$. \label{E:Sigma_p inside Sigma_H}
	\end{enumerate} 
\end{proposition}

\begin{remark}\label{r:regular}
	The equivalence of \eqref{E:PHss regular} and \eqref{E:PH regular} means that the condition of $\cP_H$ being regular depends only on the "characteristic polynomial" $\chi(\cP_H)$. Thus we have a locus
	\[
	\MHE^\reg \subseteq \MHE
	\]
	such that $\cP_H$ is regular if and only if $\chi(\cP_H) \in \MHE^\reg$. 
	According to \eqref{E:Sigma_p inside Sigma_H}, the locus $\MHE^\reg$ is equal to the image of $(\TE)_{\leq \Sigma_H}$ under the map $\TE \to \MHE$. 
	In particular, it follows from \cref{prop:regular} that $\MHE^\reg$ (respectively, $\uHE^\reg$) is open and dense in $\MHE$ (respectively, $\uHE$). 
\end{remark}

\begin{proof}[Proof of \cref{prop:regular}]
	First, let us show that conditions (\ref{E:coh of adjoint of PH is zero}) and (\ref{E:HEtoGE is etale at PH}) are equivalent. 
	The map $\pi$ is \'etale precisely at the points where its differential is a quasi-isomorphism of tangent complexes. 
	Recall that the cohomology of the tangent complex of $\uHE$ at a bundle $\cP_H$ is given by the cohomology of the adjoint bundle of $\cP_H$, that is by $\rH^i(E;\fh_{\cP_H})$.
	
	The differential of $\pi$ at $\cP_H\in\uHE$ is the map
	\[  
	\bT_{\uHE,\cP_H} \to \pi^*\bT_{\uGE,\pi(\cP_H)}
	\]
	which upon taking cohomology groups becomes
	\[
	\rH^i(E;\fh_{\cP_H})\to \rH^i(E;\fg_{\cP_G}), i=0,1. 
	\]
	The cone of this map of complexes is given by $\rH^\bullet(E;(\fg/\fh)_{\cP_H})$. Thus we have that $\rH^\bullet(E;(\fg/\fh)_{\cP_H})=0$ if and only if $\pi$ is \'etale at $\cP_H$ as required.
	
	Let us show the equivalence of \eqref{E:coh of adjoint of PH is zero} and \eqref{E:Sigma_p inside Sigma_H}. 
	Fix $\cP_{B_H}$ a reduction of $\cP_H$ to a Borel subgroup of $H$ such that the induced $T$-bundle is precisely $\cP_T$ such that $\cP_T\timesT H=\cP_H^\ss$. (See paragraph above \cref{prop:regular}.)
	The vector bundle $(\fg/\fh)_{\cP_H} = (\fg/\fh)_{\cP_{B_H}}$ carries a filtration whose associated graded is $(\fg/\fh)_{\cP_T}$. 
	This latter bundle is a direct sum of line bundles $(\fg_\alpha)_{\cP_T}$ corresponding to roots $\alpha \in \Phi \backslash \Sigma$. 
	By definition, the bundle $(\fg_\alpha)_{\cP_T}$ is trivial if and only if $\alpha \in \Sigma_{\cP_T}$. 
	Noting that the cohomology of a degree 0 line bundle on $E$ vanishes if and only if it is trivial, we deduce that (\ref{E:coh of adjoint of PH is zero}) is equivalent to (\ref{E:Sigma_p inside Sigma_H}). 
	The same argument shows the equivalence of \eqref{E:PHss regular} and \eqref{E:Sigma_p inside Sigma_H}.
	
	Notice that the bundle $\cP_H$ is regular if and only if the inclusion $\Aut(\cP_H)^\circ \to \Aut(\cP_G)^\circ$ is an isomorphism, where we put $\cP_G=\cP_H\timesH G$. 
	Taking the corresponding Lie algebras, we see that $\cP_H$ is regular if and only if the map
	\[
	H^0(E;\fh_{\cP_H}) \to H^0(E;\fg_{\cP_H})
	\]
	is an isomorphism. 
	Using the long exact sequence associated to the short exact sequence of $H$-modules
	\[ 0\to \fh\to \fg \to \fg/\fh\to 0 \] and Serre duality (the canonical bundle is trivial), this is equivalent to condition \eqref{E:coh of adjoint of PH is zero}.
	We've shown that \eqref{E:coh of adjoint of PH is zero} is equivalent to \eqref{E:PH regular}.
\end{proof}

We write $Z(H)_E^\reg$ for the intersection $Z(H)_E \cap \HE^\reg$ (similarly for $Z(H)_E^\strreg$, $Z_H(E)^\maxreg$). The various notions of regularity are somewhat simpler here. 

\begin{lemma}\label{l:regular vs strongly regular}
	Fix an element $p \in Z(H)_E$.
	\begin{enumerate}
		\item The element $p$ is regular if and only if it is strongly regular.
		\item Assume $p$ is regular. It is maximally regular if and only if the relative Weyl group $W_{G,H}$ acts freely on the orbit of $p$.
	\end{enumerate}
\end{lemma}
\begin{proof}
		(1) Suppose $p$ is regular, i.e. $\Stab_G(p)^\circ \subseteq H$. But $p\in Z(H)_E$, so $H \subseteq \Stab_G(p)$ and thus $\Stab_G(p)^\circ = H$. As the neutral component of any algebraic group is a normal subgroup, we have $\Stab_G(p) \subseteq N_G(H)$ as required.
		
		(2) By regularity of $p$, $\Stab_G(p)^\circ=H$ as explained above. 
		Thus $p$ is maximally regular if and only if $\Stab_G(p)$ is connected. 
		According to \cref{lem:components}, the component group of $\Stab_G(p)$ is precisely the stabilizer of $p$ in $W_{G,H}$, hence the claimed result.
\end{proof}

\subsection{The main results}\label{SS:main results}
For convenience, we remind the reader of the statements of \cref{thm:jordan intro} and \cref{thm:galois intro}, to be proved in this section.

To state the first result, recall that for an $E$-pseudo-Levi subgroup $H$ of $G$, we defined the regular locus $(\HE)^\reg$ in \cref{sec:the-regular-locus}. 
We denote by $Z(H)^\reg$ the intersection $Z(H)_E \cap \HE^\reg$. 
The locus of unipotent bundles $\HE^\uni$ is defined to be fibre $\chi_H^{-1}(\one)$ of the characteristic polynomial map $\chi_H:\HE \to \MHE$.

\begin{theorem}[Jordan decomposition]\label{thm:jordan}
 Given a semistable, degree $0$, framed $G$-bundle $p \in \GE$, there is a unique triple $(H,p_s,p_u)$ where $H$ is an $E$-pseudo-Levi, $p_s \in Z(H)_E^\reg$, $p_u \in \HE^\uni$, and 
		\[
		p =  p_s \cdot p_u 
		\]
		where the multiplication is defined via the group morphism $m\colon Z(H)\times H\to H$ (see \cref{Eg:product with a central bundle}).
\end{theorem}

To state the next result recall that we have defined a partition of $\uGE$ (and also of $\GE$) in \cref{SS:summary partition and strata}
\[ \uGE = \bigsqcup_{[H]\in\cE_E} (\uGE)_{[H]} \]
and we defined the heart locus to be $(\uGE)_\heartsuit:=(\uGE)_{[G]}$.

\begin{theorem}[Galois theorem]\label{thm:galois} For $H$ an $E$-pseudo-Levi subgroup of $G$ the morphism
		\[
		(\uHE)_\heartsuit^\reg \to \uGE
		\]
		is a $W_{G,H}$-Galois covering onto $(\uGE)_{[H]}$.
\end{theorem}

To begin with we will establish the following result, which expresses the Jordan--Chevalley decomposition on the heart locus (where, in fact, it is simply a direct product). This may be thought of as a baby version of \cref{thm:jordan}.
\begin{proposition}\label{P:heart locus is a product}
	There is an equivalence of stacks
	\[ (\uGE)_\heartsuit \simeq Z(G)_E\times \uGE^\uni. \]
\end{proposition}
\begin{proof}
	The following is the cartesian diagram defining $(\uGE)_\heartsuit$ and $\uGE^\uni$:
	\[ \begin{tikzcd}
		\uGE^\uni\arrow[r,hook]\arrow[d] & (\uGE)_\heartsuit \arrow[r,hook]\arrow[d] & \uGE\arrow[d,"\chi"]\\
		\{\one\}\arrow[r, hook]&Z(G)_E \arrow[r,hook] & \MGE
	\end{tikzcd} \]
	The group $Z(G)_E$ acts on both $\uGE$ and $\MGE$ and the map $\chi$ is equivariant.
	This readily implies the required isomorphism from the statement.
	
	One could also argue as follows: the natural product map 
	\[ Z(G)_E\times \uGE^\uni\to (\uGE)_\heartsuit \]
	is $Z(G)_E$-equivariant and this enables us to define its inverse by the following formula
	\[ \cP\mapsto (\chi(\cP),\chi(\cP)\inv \cdot \cP).\qedhere\]
\end{proof}
\begin{corollary}\label{C:heartsuit regular as a product}
	We have an equivalences of stacks
	\[ (\uHE)_\heartsuit^\reg\simeq Z(H)_E^\reg\times \uHE^\uni \simeq (\uHE)_\heartsuit^\strreg \]
\end{corollary}
\begin{proof}
	The first equivalence follows from \cref{P:heart locus is a product} and from the fact that regularity of an $H$-bundle is governed by it's semisimple part, i.e. by the part in $Z(H)_E$ (see \cref{prop:regular}). The analogous equivalence also holds for the strongly regular locus. The second equivalence then follows from the fact that $Z(H)_E^\reg = Z(H)_E^\strreg$ (\cref{l:regular vs strongly regular}).
\end{proof}

\subsection{Equivalence of Theorems \ref{thm:jordan} and \ref{thm:galois}}\label{sec:equivalence-of-theorems-refthmjordan-and-refthmgalois}
Our next step will be to show that the two main results are mutually equivalent. 

Let us first recast Theorem \ref{thm:galois} in terms of framed bundles. It states that the map
\begin{align}\label{Eq:tildepi is an isomorphism}
\widetilde{\pi}_\heartsuit^\reg\colon G\times^{N_G(H)}(\HE)_\heartsuit^\reg \to \GE
\end{align}
is a ($G$-equivariant) isomorphism onto $(\GE)_{[H]}$.

To prove Theorem \ref{thm:galois}, we must show the following two statements:
\begin{enumerate}
	\item (``Surjectivity'') The image of $\widetilde{\pi}_\heartsuit^\reg$ is precisely $(\GE)_{[H]}$.
	\item (``Injectivity'') The map $\widetilde{\pi}_\heartsuit^\reg$ is injective.
\end{enumerate}

Analogously, to prove Theorem \ref{thm:jordan}, we must show the following two statements:
\begin{enumerate}
	\item (``Existence'') Every $p \in \uGE$ has a Jordan datum $(H,p_s,p_u)$ with $p= p_s \cdot p_u \in (\HE)_\heartsuit^\reg$.
	\item (``Uniqueness'') Given $p\in \uGE$ with two Jordan data $(H,p_s,p_u)$ and $(H',p_s',p_u')$ we have $(H,p_s,p_u)=(H',p_s',p_u')$.
\end{enumerate}

We will show that the ``existence'' (respectively, ``uniqueness'') part of Theorem \ref{thm:jordan} is equivalent to the ``surjectivity'' (respectively, ``injectivity'') in Theorem \ref{thm:galois}.

\subsubsection{Existence implies surjectivity}

 Suppose $p \in (\GE)_{[H]}$ with associated Jordan data $H',p_s,p_u$. 
 It follows that $p \in (\HE')_\heartsuit^\reg \subseteq (\GE)_{[H]}$ and thus $H' = \Ad(g)H$ for some $g\in G$. 
 Then $p$ is the image of
\[
(g, g^{-1} \cdot p) \in G \times^{N_G(H)}(\HE)_\heartsuit^\reg 
\]
as required.

\subsubsection{Surjectivity implies existence}

Given $p \in (\GE)$, we let $[H]$ denote the unique conjugacy class of $E$-pseudo-Levi subgroups such that $p$ is in the locus $(\GE)_{[H]}$. 
Surjectivity means that there exists 
\[
(g,p') \in G \times (\HE)_\heartsuit^\reg
\]
such that $g \cdot p' = p$. 
By replacing $p'$ with $g^{-1} \cdot p'$ and $H$ with $\Ad(g^{-1})H$, we may assume that $p \in (\HE)_\heartsuit^\reg$. 
Recall (\cref{C:heartsuit regular as a product}) that $(\HE)_\heartsuit^\reg \cong Z(H)_E^\reg \times \HE^\uni$. 
Thus $p$ has a Jordan decomposition $p = p_s \cdot p_u$ as required.

\subsubsection{Uniqueness implies injectivity} 

We must show that if $p,p' \in (\HE)_\heartsuit^\reg$ and $g\in G$ such that $g\cdot p = p'$, then $g\in N_G(H)$. As we have assumed $p,p' \in (\HE)_\heartsuit^\reg$, we have Jordan decompositions $p=p_s \cdot p_u$ and $p'= p_s' \cdot p_u'$. 
Note that $H = \Stab_G(p_s)^\circ = \Stab_G(p_s')^\circ$. 
By the uniqueness of the Jordan decomposition, we must have that $g \cdot p_s = p_s'$. 
Thus  $\Ad(g)H = H$, so $g\in N_G(H)$ as required.

\subsubsection{Injectivity implies uniqueness}

Let $p\in \GE$ and suppose we have two Jordan data $(H,p_s,p_u)$ and $(H',p_s',p_u')$. 
To prove the uniqueness of the Jordan decomposition it suffices to show that $H=H'$ (as within $(\HE)_\heartsuit^\reg = Z(H)_E^\reg \times (\HE)^\uni$, every element has a unique Jordan decomposition).

First observe that the Lusztig type of $p$ is well-defined, so $[H] = [H']$, i.e. there exists $g\in G$ such that $\Ad(g)H=H'$.

Now we have $p \in (\HE')_\heartsuit^\reg$, and thus $g^{-1} \cdot p \in (\HE)_\heartsuit^\reg$. 
It follows that $p$ is the image of both $(1,p)$ and $(g, g^{-1} \cdot p)$ under the map
\[
\widetilde{\pi}_\heartsuit^\reg\colon G\times^{N_G(H)} (\HE)_\heartsuit^\reg \to \GE
\]
By injectivity, we must have that $g\in N_G(H)$. 
But then $H'=H$ as required. 

The remainder of this section will be taken up with the proof of Theorem \ref{thm:galois} (and thus Theorem \ref{thm:jordan}).

\subsection{Proof of surjectivity}\label{sec:proof-of-surjectivity}
In this subsection we give a proof of the surjectivity part of \cref{thm:galois}. Specifically, we show the following:
\begin{proposition}\label{l:characterize locus}
	The restriction 
	\[
	\pi_\heartsuit^\reg\colon (\uHE)_\heartsuit^\reg \to \uGE
	\]
	maps surjectively onto $(\uGE)_{[H]}$.
\end{proposition}

The statement splits into two parts:
	\begin{enumerate}
		\item \label{sublemma 1} The image of 
		\[
		\pi^\reg_\heartsuit\colon (\uHE)^\reg_\heartsuit \to \uGE
		\]
		is contained in $(\uGE)_{[H]}$.
		\item \label{sublemma 2} The image of $\pi^\reg_\heartsuit$ contains all of $(\uGE)_{[H]}$.
	\end{enumerate}

\subsubsection{Proof of \eqref{sublemma 1}}	
	Let $\cP_H \in \uHE$ and let $\cP_G = \pi(\cP_H)$ be the induced bundle. 
	Then we will show
	\begin{enumerate}[(a)]
		\item \label{subpart a}if $\cP_H \in \uHE^\reg$ then $\cP_G \in (\uGE)_{\leq[H]}$;
		\item \label{subpart b} $\cP_H \in (\uHE)_\heartsuit$ if and only if $\cP_G \in (\uGE)_{\geq[H]}$.
	\end{enumerate}

\begin{remark}
	The converse of part (\ref{subpart a}) above is false. For example, suppose we are in the group case $E=E_{\node}$ with $G=G_E=GL_3$. Let $H \subseteq G$ denote the subgroup consisting of matrices of the form
	\[
	\left(\begin{array}{@{}c|c@{}}
	\begin{matrix}
	\ast & \ast \\
	\ast & \ast
	\end{matrix}
	&  \begin{matrix}0\\0\end{matrix} \\
	\hline
	\begin{matrix}0&0\end{matrix} &
	\ast
	\end{array}\right)
	\]
	Now let $p \in H$ denote the matrix:
	\[
	\begin{pmatrix}
	1&0&0\\
	0&2&0\\
	0&0&2
	\end{pmatrix}
	\]
	Then $p\in G_{[H]}$ as $\Stab_G(p)$ is conjugate to $H$, but $p \notin H^\reg$ as $\Stab_G(p) \nsubseteq H$.	
\end{remark}

\emph{Proof of \eqref{subpart a}}	First suppose that $\cP_H \in \uHE$ is semisimple. Let $p \in \HE$ denote a lift to a framed bundle.
		
		Then $\cP_H$ is regular if and only if $\Aut(\cP_G)^\circ \cong \Stab_G(p)^\circ \subseteq H$. On the other hand, $\cP_G$ is contained in $(\uGE)_{\leq[H]}$ if and only if $\Stab_G(p)$ is contained in some $G$-conjugate of $H$.
		
		Thus we have the required implication in case $\cP_H$ is semisimple. 
		In general, the result follows from the fact that the conditions $\cP_H \in \uHE^\reg$ and $\cP_G \in (\uGE)_{\leq[H]}$ only depend on the characteristic polynomial of $\cP_H$ (respectively $\cP_G$) and thus only depend on the isomorphism class of the semisimplification.
		
\emph{Proof of \eqref{subpart b})} Again, suppose $\cP_H$ is semisimple. 
Then $\cP_H$ is contained in $(\uHE)_\heartsuit$ if and only if $ H = \Stab_H(p) = \Stab_G(p)\cap H$. 
This in turn is equivalent to $\cP_G \in (\uGE)_{\geq[H]}$. 
As before, the general case follows from the fact that the conditions only depend on the characteristic polynomial.

\subsubsection{Proof of \eqref{sublemma 2}} 
Suppose $\cP_G$ is contained in $(\uGE)_{\leq [H]}$ (respectively, $(\uGE)_{[H]}$). Then we will show that there exists $\cP_H$ in $(\uHE)^\reg$ (respectively, $(\uHE)_\heartsuit^\reg$) such that $\pi(\cP_H) = \cP_G$.
	
	Recall that by \cref{prop:regular}, the substack $\uHE^\reg$ is precisely the locus on which the map $\pi$ is \'etale. In particular, the image $\pi(\uHE^\reg)$ is an open substack of $\uGE$. We must show that this open substack contains all of $(\uGE)_{\leq[H]}$.
	
	First suppose $\cP_G \in (\uGE)_{\leq[H]}$ is semisimple. Let $T$ be a maximal torus of $H$ (and thus of $G$). Then, by \cref{l:automorphisms and root systems}, there is a framed reduction $p \in \TE$ such that $\Stab_G(p) \subseteq H$. Thus the induced $H$-bundle $\cP_H$ is a regular $H$-reduction of $\cP_G$ as required.
	
	Now let us drop the assumption that $\cP_G$ is semisimple, and let $\cP_G^\ss$ denote a semisimplification of $\cP_G$. Then $\cP_G^\ss$ is semisimple and contained in $(\uGE)_{\leq[H]}$, thus by the above argument $\cP_G^\ss$ is contained in $\pi(\uHE^\reg)$. But $\cP_G^\ss$ is contained in the closure of the point $\cP_G$; thus any open neighbourhood of $\cP_G^\ss$ in $\uGE$ contains $\cP_G$. As $\pi(\uGE^\reg)$ is such an open neighbourhood, we must have $\cP_G = \pi(\cP_H)$ for some $\cP_H \in \uHE^\reg$ as required.

	It remains to show that if moreover $\cP_G\in (\uGE)_{[H]}$ then the above constructed $\cP_H$ belongs to $(\uHE)_\heartsuit$.
	If $p$ is a framed lift of $\cP_H$ such that $\Stab_G(p)^\circ\subset H$ then the condition $\cP_G\in (\uGE)_{\ge [H]}$ means $H\subset \Stab_G(p)^\circ$.
	We deduce that $H=\Stab_G(p)^\circ$ which in turn implies $\Stab_H(p)=H$, or in other words $\cP_H\in (\uHE)_{\heartsuit}$.	

\subsection{Proof of injectivity}\label{sec:proof-of-injectivity} In this section we prove the injectivity required (see \eqref{Eq:tildepi is an isomorphism}) in the proof of \cref{thm:galois}.
In other words, we must show that 
\begin{align}\label{Eq:pi bar heart is embedding}
\overline{\pi}_\heartsuit^\reg\colon (\uHE)_\heartsuit^\reg/W_{G,H} \to \uGE
\end{align}
is an embedding.

Let us first sketch an outline of the strategy of proof. 
We wish to apply the following general principle:
\begin{proposition}\label{p:etale generic embedding}
	If an \'etale morphism of schemes $X\to Y$ is an embedding over a dense open subset of $Y$ and $X$ is 
	separated, then it is an open embedding.
\end{proposition}
\begin{proof}
	We reduce it to \cite[Thm 17.9.1]{EGA4.4}.
	Namely, according to loc.cit., a morphism of schemes $U\to V$ is an open immersion if and only if it is flat, 
	locally of finite presentation and a monomorphism in the category of schemes.
	
	In our case we only need to check that the morphism is a monomorphism which follows at once from birationality 
	and separatedness.
\end{proof}
\begin{remark}
	Since being an open immersion is a property that is smooth-local on the target, the proposition can be applied to 
	a morphism of finite type stacks that is representable by separated schemes and this is how it will be used below. 
\end{remark}

While $\overline{\pi}$ is \'etale over the locus $(\uHE)^\reg$, we cannot apply \cref{p:etale generic embedding} directly to the morphism $\overline{\pi}$ as it is not generically an embedding - its generic fiber has cardinality $|W|/|N_W(W_H))|$ (see \cref{l:both coverings of same degree}). 

However, the failure of $\overline{\pi}$ to be generically an embedding is precisely accounted for by the corresponding map 
\[
\overline{\rho}:\MHE \GIT W_{G,H} \to \MGE
\]
on the level of coarse moduli spaces, which also has generic degree $|W|/|N_W(W_H)|$ (\cref{l:both coverings of same degree}). The idea is thus to replace the target $\uGE$ with the base-change $\widetilde{\uGE}$:

\[
\xymatrix{
	\uHE/W_{G,H} \ar@/^2pc/[rr]^{\overline{\pi}} \ar[r]^{\overline{\nu}} \ar[rd]_{\overline{\underline{\chi}}_H} 
	& \widetilde{\uGE} \ar[r]^{\widetilde{\rho}} \ar[d] \ar@{}[rd]|\square &\uGE \ar[d]_{\underline{\chi}_G} \\
&	 \MHE \GIT W_{G,H} \ar[r]_{\overline{\rho}} & \MGE
}
\]

In this way, we obtain a morphism $\overline{\nu}$ which is generically an embedding (\cref{l:embedding}). 
Moreover, we will show that $\overline{\nu}$ is \'etale when restricted to the locus $\uHE^\strreg/W_{G,H}$ (see 
\cref{l:etale}), and thus an open embedding on this locus by \cref{p:etale generic embedding} (we can apply it in this 
situation since $\overline{\nu}$ comes from an obvious $G$-equivariant morphism between varieties). 
As this locus contains the desired substack $(\uHE)_\heartsuit^\reg/W_{G,H}$ (recall that $(\HE)_\heartsuit^\strreg = 
(\HE)_\heartsuit^\reg$ by \cref{C:heartsuit regular as a product}), the morphism 
\[
\overline{\nu}_\heartsuit^\reg\colon (\uHE)_\heartsuit^\reg/W_{G,H} \to \widetilde{\uGE}
\]
is an embedding.

Finally, we show that the restriction of the base change morphism 
\[
\widetilde{\rho}_\heartsuit^\reg:(\widetilde{\uGE})_\heartsuit^\reg \to \uGE
\]
is an embedding (\cref{C:base change embedding}). Thus the composite
\[
\xymatrix{
	(\uHE)_\heartsuit^\reg/W_{G,H} \ar@{^{(}->}[r] &
	(\widetilde{\uGE})_\heartsuit^\reg \ar@{^{(}->}[r] & \uGE
}
\]
is an embedding as required in \cref{Eq:pi bar heart is embedding}.

Now let us go through the steps in the proof one by one. 
We start with the two lemmas below, whose proof will be given in \cref{sec:proofs-creflembedding-crefletale-and-creflinjective-coarse}. 

\begin{lemma}\label{l:etale}
	The following restriction of $\overline\nu$ is étale
	\[
	\overline{\nu}^\strreg\colon (\uHE)^\strreg/W_{G,H} \to \widetilde{\uGE}.
	\]
\end{lemma}

\begin{lemma}\label{l:embedding}
	There is an open dense subset of $\uHE$ on which $\overline{\nu}$ is an open embedding.
\end{lemma}

Assuming \cref{l:etale} and \cref{l:embedding}, we may now establish:

\begin{lemma}\label{l:nu an embedding}
	The restriction 
	\[
	\overline{\nu}_\heartsuit^\reg\colon (\uHE)_\heartsuit^\reg/W_{G,H} \to \widetilde{\uGE}
	\]
	is an embedding.
\end{lemma}
\begin{proof}
	By \cref{l:etale} $\overline{\nu}^\strreg$ is \'etale, and by \cref{l:embedding} it is generically an embedding. Thus by \cref{p:etale generic embedding}, $\overline{\nu}^\strreg$ is an open embedding. The claim then follows immediately from the fact that $(\uHE)_\heartsuit^\reg = (\uHE)_\heartsuit^\strreg$ (\cref{C:heartsuit regular as a product}).
\end{proof}

The following lemma is an analogue on the level of coarse moduli spaces of the desired injectivity part of \cref{thm:galois}. Its proof is given in \cref{sec:proofs-creflembedding-crefletale-and-creflinjective-coarse}.
	
\begin{lemma}\label{l:injective coarse}
	The morphism
	\[
	\overline{\rho}_\heartsuit^\reg\colon \MHE_\heartsuit^\reg \git W_{G,H} \xrightarrow{} \MGE
	\]
	is an embedding.
\end{lemma}

It follows immediately that the base change is also an embedding:

\begin{corollary}\label{C:base change embedding}
The morphism
\[
\widetilde{\rho}_\heartsuit^\reg\colon (\widetilde{\uGE})_\heartsuit^\reg \xrightarrow{} \uGE
\]
is an embedding.
\end{corollary}

Putting \cref{l:nu an embedding} and \cref{C:base change embedding} together we obtain that the composite
\[
\xymatrix{
(\uHE)_\heartsuit^\reg/W_{G,H} \ar@{^{(}->}[r] &
(\widetilde{\uGE})_\heartsuit^\reg \ar@{^{(}->}[r] & \uGE
}
\]
is an embedding. This completes the proof of \cref{thm:galois} (modulo the proofs in the following subsection).

\subsection{Proofs of \cref{l:embedding}, \cref{l:etale}, and \cref{l:injective coarse}}\label{sec:proofs-creflembedding-crefletale-and-creflinjective-coarse}

For this subsection we will fix a maximal torus $T$ of our fixed $E$-pseudo-Levi subgroup $H$ (and thus $T$ is also a maximal torus of $G$). 

The following lemma (which is essentially \cref{thm:galois} for the special case $H=T$) forms a key step in the proof of \cref{l:embedding}.

\begin{proposition}\label{p:galois torus case}
	The induction map
	\[
	\uTE^\reg \to \uGE
	\]
	is an unramified $W$-Galois cover onto $(\uGE)_{[T]}$.
\end{proposition}
\begin{proof}
	The claim is equivalent to the statement that
	\[
	G\times^{N_G(T)}\TE^\reg \to (\GE)_{[T]}
	\]
	is an isomorphism. 
	We have already seen that the map is onto and étale (\cref{l:characterize locus,prop:regular}). 
	For injectivity, we must show that for all $p \in \TE^\reg$ we have $\Stab_G(p) \subseteq N_G(T)$. 
	But by definition $\Stab_G(p)^\circ = T$, so the required statement follows from the fact that $\Stab_G(T)$ normalizes its own neutral component.
\end{proof}	

\begin{definition}\label{d:generic covering}
We say that a morphism $f:X\to Y$ is \emph{generically a covering of degree $d$} if there are dense open subsets $U$ of $X$ and $V=f(U)$ of $Y$ such that $f|_U:U \to V$ is a covering (i.e. finite and \'etale) of degree $d$.
\end{definition}
	 
\begin{lemma}\label{l:both coverings of same degree}
	The following morphisms are generically covering maps of degree $|W|/|N_W(W_H)|$:
	\begin{enumerate}
		\item $	\pi:\uHE/W_{G,H} \to \uGE$,
		\item $\rho:\MHE \GIT W_{G,H} \to \MGE$.
\end{enumerate}
\end{lemma}
\begin{proof}
	(1) Consider the following diagram:
	\[
	\xymatrix{
		\uTE \ar[r]^\beta \ar[rd] \ar@/^2pc/[rr]^\alpha& \uHE \ar[r]^\gamma \ar[d]^\delta& \uGE \\
		& \uHE/W_{G,H} \ar[ru]_\pi &	
	}
	\]
By \cref{p:galois torus case}, we have that $\alpha$ is generically a covering of degree $|W|$ and $\beta$ is generically a covering of degree $|W_H|$. 
Thus $\gamma$ is generically a covering of degree $|W|/|W_H|$. 
By construction, $\delta$ is a covering of degree $|W_{G,H}|$. 
Thus $\pi$ is generically a covering of degree $|W|/|W_H||W_{G,H}| = |W|/|N_W(W_H)|$ as required.

(2) Now consider the diagram:
\[
\xymatrix{
	\TE \ar[r]^{\beta'} \ar@/^2pc/[rr]^{\alpha'}  & \TE \GIT N_W(W_H) \ar[r] \ar[d]^\wr & \TE \GIT W \ar[d]^\wr\\
	& \MHE \GIT W_{G,H} \ar[r]_{\rho} & \MGE
}
\]
As $W$ acts freely on a dense open subset of $\TE$ (namely $\TE^\maxreg$; see \cref{D:regular and strongly regular}), we have that $\alpha'$ is generically a covering of degree $|W|$ and $\beta'$ is generically a covering of degree $|N_W(W_H)|$. Thus $\rho$ is generically a covering of degree $|W|/|N_W(W_H)|$ as required.\qedhere
\end{proof}

\begin{proof}[Proof of \cref{l:embedding}]
Consider again the diagram:
\[
\xymatrix{
	\uHE/W_{G,H} \ar@/^2pc/[rr]^{\overline{\pi}} \ar[r]^{\overline{\nu}} \ar[rd]_{\overline{\underline{\chi}}_H} 
	& \widetilde{\uGE} \ar[r]^{\widetilde{\rho}} \ar[d] \ar@{}[rd]|\square &\uGE \ar[d]_{\underline{\chi}_G} \\
	&	 \MHE \GIT W_{G,H} \ar[r]_{\overline{\rho}} & \MGE
}
\]

By \cref{l:both coverings of same degree}, $\overline{\pi}$ and $\overline{\rho}$ are both generically coverings of degree $|W|/|N_W(W_H)|$. 
Thus, by base change, $\widetilde{\rho}$ is generically a covering of degree $|W|/|N_W(W_H)|$. 
But as $\overline{\pi} = \widetilde{\rho}\circ\overline{\nu}$ we must have that $\overline{\nu}$ is generically a covering of degree $1$, i.e. generically an embedding as required.
\end{proof}

\begin{lemma}\label{lem:coarse-moduli-etale}
	The morphism
	\[
	\MHE\GIT W_{G,H} \to \MGE
	\]
	is \'etale on the locus $\MHE^\strreg \GIT W_{G,H}$.
\end{lemma}
\begin{proof}
	 Recall that 
	\[
	\MGE \simeq \TE \git W
	\]
	and 
	\[
	\MHE \git W_{G,H} \simeq \TE\git N_W(\Sigma)
	\] 
	where $\Sigma$ is the root system of $H$.
	Using \cite[Prop V.2.2]{SGA1} we deduce that if $N_W(\Sigma)$ contains the stabilizer $\Stab_W(p)$ then the 
	map
	\[ T_E\git N_W(\Sigma)\to T_E\git W \]
	is étale at $p$.	
	Thus the claim reduces to \cref{lem:regular-normalizer} below.
	\end{proof} 

\begin{lemma}\label{lem:regular-normalizer}
	Let $p\in \TE \cap \HE^\strreg$. 
	Then $\Stab_W(p) \subseteq N_W(\Sigma)$.
\end{lemma}
\begin{proof}
	By assumption $\Stab_G(p)^\circ \subseteq H$. 
	Suppose $w\in \Stab_W(p)$ and choose a lift to $\widetilde{w} \in N_G(T)$. 
	Then $\widetilde{w} \in N_G(H)$ by assumption, and thus $w$ preserves the root system $\Sigma$ of $H$ as required.
	\end{proof}

We may now proceed with:
\begin{proof}[Proof of \cref{l:etale}]
Consider once more the diagram:
\[
\xymatrix{
	\uHE/W_{G,H} \ar@/^2pc/[rr]^{\overline{\pi}} \ar[r]^{\overline{\nu}} \ar[rd]_{\overline{\underline{\chi}}_H} 
	& \widetilde{\uGE} \ar[r]^{\widetilde{\rho}} \ar[d] \ar@{}[rd]|\square &\uGE \ar[d]_{\underline{\chi}_G} \\
	&	 \MHE \GIT W_{G,H} \ar[r]_{\overline{\rho}} & \MGE
}
\]	
By \cref{lem:coarse-moduli-etale}, $\overline{\rho}$ is \'etale on $\MHE^\strreg\GIT W_{G,H}$. 
Therefore its base change, $\widetilde{\rho}$ is \'etale on $\widetilde{\uGE}^\strreg$. 
We also know that $\uHE^\strreg \subseteq \uHE^\reg$ is \'etale over $\uGE$ by \cref{prop:regular}. 
Thus both the source and target of $\overline{\nu}^\strreg$ is \'etale over $\uGE$, and hence $\overline{\nu}^\strreg$ is itself \'etale as required.
\end{proof}

Finally we come to
\begin{proof}[Proof of \cref{l:injective coarse}]
 We have a commutative diagram
	\[
	\xymatrix{
		& \ar[ld] \TE \ar[rd] &\\
		\MHE \ar[rr] && \MGE
	}
	\]
	which exhibits $\MGE$ as $\TE \git W$ and $\MHE$ as $\TE\git W_H$. 
	The further quotient $\MHE \git W_{G,H}$ is thus identified with $\TE \git N_W(W_H)$ (recall that $W_{G,H} \cong N_W(W_H)/W_H$). 
	
	Let us denote by $\Sigma = \Sigma_H \subseteq \Phi$ the roots of $H$ with respect to $T$. The locus $\MGE_{[H]}$ (respectively, $\MHE_{\heartsuit}^\reg$) is precisely the image of $(\TE)_{\Sigma}$ (see \cref{l:comparing two descriptions of partition}). 
	
	Thus $\MGE_{[H]}$ is identified with the quotient of $(\TE)_{\Sigma}$ by the subgroup of $W$ which preserves the locus $(\TE)_{\Sigma}$. 
	But this subgroup is precisely $N_W(W_H)$. 
	Thus both $\MGE_{[H]}$ and $\MHE_{\heartsuit}^\reg \git W_{G,H}$ are identified with $(\TE)_\Sigma \git N_W(W_H)$. 
	In particular, the map
	\[
	\MHE_{\heartsuit}^\reg\git W_{G,H} \to \MGE_{[H]}
	\]
	is an isomorphism as required.
\end{proof}

\section{The Tannakian approach and unipotent bundles}
The goal of this section is to understand the geometry of the locus of unipotent bundles $\GE^\uni$ in $\GE$. 
Let $(\GE)^\wedge_\uni$ denote the formal neighbourhood of the unipotent locus. 
Similarly, we have the unipotent cone $G^\uni$ in $G$ and its formal neighbourhood $G^\wedge_\uni$. 
We will prove:

 \begin{theorem}\label{t:unipotent}
Let $E,E'$ be two pointed curves of arithmetic genus 1. Then any isomorphism of formal groups $\widehat{J(E)} \cong \widehat{J(E')}$ defines $G$-equivariant isomorphisms
 \[
 \xymatrix{
 	(G_{E'})^\wedge_\uni \ar[r]^\sim & (\GE)^\wedge_\uni \\
 	G_{E'}^\uni \ar[r]^\sim \ar@{^{(}->}[u] & \GE^\uni \ar@{^{(}->}[u]
 }
 \]
 \end{theorem}
\begin{corollary}\label{c:unipotent}	
If $E$ is an ordinary elliptic curve over $k$ (in particular, if $char(k)=0$) then we have isomorphisms
\[
\xymatrix{
	G^\wedge_\uni \ar[r]^\sim & (\GE)^\wedge_\uni \\
	G^\uni \ar[r]^\sim \ar@{^{(}->}[u] & \GE^\uni \ar@{^{(}->}[u]
}
\]

\end{corollary}

In order to prove \cref{t:unipotent} we will use the following result.
\begin{theorem}\label{t:tannaka}
	Given a genus 1 marked curve $E$ as above, there is an equivalence of stacks
	\[
	\uGE \simeq \Fun^\otimes(\Rep_k(G),\Tor(J(E)))	.
	\]
\end{theorem}

\cref{t:tannaka} is a combination of a Tannaka duality statement, expressing $G$-bundles in terms of their associated vector bundles, and a Fourier-Mukai duality, relating vector bundles on $E$ and torsion sheaves on $J(E)$. 

The idea of the proof of \cref{t:unipotent} is to use the isomorphism of formal groups $\widehat{E'} \cong \widehat{E}$ to identify torsion sheaves on $E$ and $E'$ which are supported in a formal neighborhood of the identity. 

\begin{remark}
	Note that in characteristic zero, we can identify $\widehat{E} \cong \widehat{\Gm} \cong \widehat{\Ga}$. Thus, under these conditions we obtain isomorphisms between the nilpotent, unipotent, and elliptic unipotent cones (and their formal neighbourhoods). These isomorphisms arise from exponential maps in characteristic $0$.  

	In characteristic $p>0$, there is no isomorphism of formal groups $\widehat{\Ga} \cong \widehat{\Gm}$. Nevertheless, under very mild conditions on the characteristic, there are $G$-equivariant isomorphisms (the so-called Springer isomorphisms) between the nilpotent cone in $\fg$ and the unipotent cone in $G$. It is natural to conjecture that there are also Springer isomorphisms between $\GE^\uni$ and $G^\uni$ (and $\fg^\nil$). We plan to return to this in future work.
\end{remark}

\subsection{Tannaka duality}
In this subsection, we will use Tannaka duality to express the stack $\uGE$ in terms of its associated vector bundles. The primary reference will be the paper \cite{Lurie_Tannaka}, however see also \cite{Nori_fg} for the original  approach.

Let $S$ be a $k$-scheme. We denote by $\Rep_k(G)$ denote the symmetric monoidal category of finite dimensional representations of $G$ over $k$. Given a $G$-bundle $\cP_G$ on $S$, the associated vector bundle construction affords a symmetric monoidal functor
\begin{align*}
\ass(\cP_G)\colon\Rep(G) &\to \Vect(S)\\
V &\mapsto \cP_G \times^G V
\end{align*}
where the right hand side denotes the exact category of vector bundles on $S$ (with monoidal structure given by tensor product). Moreover, $\ass(\cP_G)$ is continuous (meaning it preserves all small colimits), exact, and sends finite dimensional representations to vector bundles on $S$.

The basic idea of Tannakian reconstruction is that the bundle $\cP_G$ can be recovered from the functor $\ass(\cP_G)$. 

More precisely, let $\Fun^\otimes(\Rep_k(G),\Vect(S))$ denote the groupoid of exact tensor functors (note that $\Vect(S)$ is an exact category, so this makes sense).

\begin{theorem}
The associated bundle construction defines an equivalence of groupoids
\[
\ass\colon \Bun_G(S) \xrightarrow{\sim} \Fun^\otimes(\Rep_k(G),\Vect(S)) 
\]
\end{theorem}
\begin{proof}
By \cite[Theorem 5.11]{Lurie_Tannaka} the associated bundle construction defines an equivalence of groupoids
\begin{equation}\label{eq:Lurie}
\ass\colon\Bun_G(S) \xrightarrow{\sim} \Fun^{\otimes}_{tame}(\QC(BG),\QC(S))
\end{equation}
where the right hand side denotes the groupoid of continuous (i.e. colimit preserving), tame tensor functors.  By definition, a functor is \emph{tame} (see \cite[Definition 5.9]{Lurie_Tannaka}) if it preserves flat objects, and short exact sequences of flat objects. 

As every object of $\QC(BG)$ (which is identified with the category of $\cO(G)$-comodules) is flat, the right hand side consists of exact functors which take values in flat objects of $\QC(S)$. 
Moreover, $\QC(BG)$ is the Ind-completion of $\Rep_k(G)$, and thus the data of an exact, continuous functor from $\QC(G)$ is equivalent to specifying an exact functor from $\Rep_k(G)$. 
By continuity of the tensor product, this equivalence preserves symmetric monoidal structures. 
Finally, by construction, every $\ass(\cP_G)(V)$ is a vector bundle for every finite dimensional representation $V$.
Thus we can identify the groupoid of tensor functors in \eqref{eq:Lurie} with those in the statement of the theorem, as required.
\end{proof}

Let $E$ be a curve of arithmetic genus 1 as usual, and let $E_S = E\times S$ denote the base-change to an arbitrary test scheme $S$. 

Note that a $G$-bundle on $E_S$ is semistable if and only if all its associated vector bundles are semistable and of degree $0$ (see \cref{r:semistability associated}). Thus we may identify the sublocus $\uGE = \Bun_G^{0}(E)^\ss$ in terms of Tannaka duality:

\begin{corollary}\label{c:tannaka}
	For each test scheme $S$, the associated bundle construction defines an equivalence of groupoids
	\[
	\uGE(S) \xrightarrow{\sim} \Fun^\otimes(\Rep_k(G),\Vect^{\ss,0}(E_S))
	\]
	where the right hand side denotes the groupoid of exact tensor functors.
\end{corollary}

\subsection{Fourier-Mukai transform}
Let $J$ be a one-dimensional commutative group scheme, for e.g. $J(E)$ for $E$ elliptic curve, $\Gm$, or $\Ga$. 
Given a test scheme $S$, we define the category of $S$-families of torsion sheaves on $J$:
\[ 
\Tor(J)(S):=\left\{\cP\in \Coh(J_S)\biggl| \begin{array}{l} \cP \text{ flat over }S \\ \Supp(\cP)\to S \text{ is finite} \end{array}\right\}. 
\]
The category $\Tor(J)(S)$ carries a monoidal structure given by convolution:
\[
\cP_1 \ast \cP_2 = m_\ast(\cP_1 \boxtimes \cP_2)
\]
where $m\colon J_S \times_S J_S \to J_S$ is the multiplication map. This construction defines a presheaf of tensor categories on $\Sch_k$.

\begin{theorem}\label{t:fourier mukai}
	Let $E$ be an irreducible projective curve of arithmetic genus one.
	The assignment 
	\[
	\cL \mapsto \cO_{[\cL]}
	\]
	taking a degree $0$ line bundle on $E$ to the corresponding skyscraper sheaf on $J(E)$ 
	extends to an equivalence of symmetric monoidal categories
	\[
	\left(\Vect^{0,\ss}(E_S), \otimes \right) \simeq \left( \Tor(J(E))(S), \ast \right)
	\]
\end{theorem}

\begin{proof}
The equivalence of abelian categories is classical in the case of a smooth elliptic curve; in our situation where $E$ is an integral curve of arithmetic genus 1, it may be deduced from \cite{teodorescu_semistable_1999}, Theorem 1.3 (for the absolute case), Theorem 1.9 (for the relative case). More generally, it is shown in \emph{loc. cit.} that there is a canonical equivalence between degree 0 semistable torsion-free sheaves on $E$ and torsion sheaves on $E$. It follows readily from the construction (see Proposition 1.8 of \emph{loc. cit.}) that when the torsion-free sheaf happens to be a vector bundle, the corresponding torsion sheaf is supported on the smooth locus of $E$, which is identified with $J(E)$ (using the section $x_0$). 

In fact, this construction is an example of a Fourier-Mukai transform. We have a canonical Poincar\'e line bundle $\cP$ on $E \times J(E)$ (normalized at $x_0 \in E$) with the defining property that $\cP|_{E \times \{\cL\}} \cong \cL$ for any degree zero line bundle $\cL \in J(E)$. This kernel extends to a torsion-free sheaf $\widetilde{\cP}$ on $E\times E$. Taking $\widetilde{\cP}$ as a Fourier-Mukai kernel defines a functor
\[
\F:D^b(E) \to D^b(E)
\]
which according to Burban--Kreussler \cite{burban_fourier_2005}, Theorem 2.21 agrees with the functor of \cite{teodorescu_semistable_1999} when restricted to semistable torsion-free sheaves of degree zero, and thus further restricts to the desired functor on semistable degree 0 vector bundles. 

From this perspective, the claim about symmetric monoidal structures is a special case of a more general claim that the Fourier-Mukai transform $D^b(E) \to D^b(J(E))$ induced by the Poincar\'e bundle intertwines the tensor product on $E$ with convolution (also known as Pontryagin product) induced by the group operation on $J(E)$. See e.g. \cite{bartocci_fourier_2009} for the case of abelian varieties, which applies to our setting when $E$ is a smooth elliptic curve. 

Alternatively, in the cuspidal and nodal cases, one may directly apply the results of \cite{FriedMorgIII}, Corollary 2.1.4, 2.2.4 which give an equivalence between degree $0$ semistable vector bundles on $E$ of rank $n$ and conjugacy classes of  $n\times n$ matrices (respectively, invertible matrices). It may be readily checked that, taking all ranks $n$ together, these equivalences define a symmetric monoidal equivalence between all degree zero semistable vector bundles on $E$, and torsion modules for $k[t]$ (respectively $k[t,t^{-1}]$).
\end{proof}

In particular, we obtain a description of the moduli stack of $GL_n$-bundles in terms of torsion sheaves. Each $S$-family of torsion sheaves has a well defined length: the degree of $\Supp(\cP) \to S$. For each test scheme $S$ we let $\Tor_n(J)(S)$ denote the maximal subgroupoid of $\Tor(J)(S)$ whose objects are torsion sheaves of length $n$.
\begin{corollary}\label{C:gln case}
	There are equivalences of stacks
	\[
	\underline{\GL}_{n,E} \cong \Vect_n^{0,\ss}(E) \cong \Tor_n(J(E))
	\]
\end{corollary}

Putting together the Tannakian statement \cref{c:tannaka} with the Fourier-Mukai statement \cref{t:fourier mukai} we obtain:

\begin{corollary}\label{C:ssbundles as functors}
	There is an equivalence of stacks
	\[
	\uGE \simeq \Fun^\otimes(\Rep_k(G),\Tor(J(E))).
	\]
\end{corollary}

\subsection{Torsion sheaves and effective divisors}
We wish to identify the loci of unipotent bundles in terms of their associated torsion sheaves. 
This will be expressed in terms of the \emph{Norm map} which associates a Cartier divisor to a torsion sheaf.

First let us recall that for a smooth curve $J$, we have isomorphisms
\[
\Div_n(J) \cong \Hilb_n(J) \cong \Sym^n(J)
\]
where:
\begin{itemize}
	\item $\Div_n(J)$ denotes the moduli of effective Cartier divisors of degree $n$, 
	\item $\Hilb_n(J)$ is the Hilbert scheme of points of length $n$ on $J$,
	\item $\Sym^n(J)$ is the quotient $J^n/\fS_n$.
\end{itemize}
(See e.g. Milne \cite{MR861976}[Theorem 3.13].)
Moreover, these are all smooth varieties of dimension $n$.

Given an object $\cP$ of $\Tor_n(J)(T)$ we will associate a relative effective Cartier divisor $Div_n(\cP)$ on $J_T$ which measures the support of $\cP$ with multiplicity. 

We will sketch one construction of $Div(\cP)$ from \cite{mumford_geometric_1994}[Section 5.3] (see also \cite{knudsen_projectivity_1976}). 
Suppose we are given $\cP$ an $S$-family of torsion sheaves on $J$. 
In particular, $\cP$ is a coherent sheaf on $J_S$ which is flat over $S$ and with no support in depth $0$. 
We consider the determinant line $\det(\cP)$, which may be defined locally in terms of a free resolution. 
As $\cP$ has no generic support, $\det(\cP)$ carries a canonical section $\cO_{J_S} \to \det(\cP)$ defined away from the support of $\cP$. 
This rational section defines the Cartier divisor $Div(\cP)$.

This construction gives rise to a morphism of stacks:
\[
Div\colon \Tor_n(J) \to \Div_n(J) \cong \Sym^n(J).
\]
Note that $\Sym^n(J(E))$ is isomorphic to the base of the characteristic polynomial map $\cM_E(GL_n)$. 
In fact, the following lemma explains that the morphism $Div$ is a realization of the characteristic polynomial map via the equivalence of $GL_n$-bundles and torsion sheaves of length $n$.
\begin{lemma}\label{L:div=char}
	The equivalence of \cref{C:gln case} fits in to a commutative square:
	\[
	\xymatrix{
		\underline{GL}_{n,E} \ar[d] \ar[r]^\sim & \Tor_n(J(E)) \ar[d]^\nu \\
		\cM_E(GL_n) \ar[r]^\sim & \Sym^n(J(E))
	}
	\]
\end{lemma}
\begin{proof}
	We have already constructed all the arrows in the diagram. To see that the diagram commutes, it is sufficient to check the commutativity on the dense substacks consisting of semisimple objects. This in turn reduces to checking the assertion for $n=1$ where it is clear.
\end{proof}

\subsection{Unipotent bundles}
Recall that $\one \in J$ denotes the unit for the group structure. 
Let $\one_n \in \Sym^n(J)$ correspond to the effective Cartier divisor on $J$ given by the unique length $n$ subscheme of $J$ whose underlying reduced scheme is $\one\times S$. 
We will also use the notation $\one_n$ to denote the corresponding length $n$ subscheme of $J$. 

The space of unipotent torsion sheaves on $J$, denoted $\Tor_n(J)^\uni$ is defined to be the fiber of $\Tor_n(J)$ over the point $\one_n$ in $\Sym^n(J)$. 
In other words, an $S$-family of torsion sheaves is called unipotent if the corresponding divisor is equal to $\one_{n,T}$.

We write $\Tor_n(J)^\wedge_\uni$ for the completion of $\Tor_n(J)$ along the substack $\Tor_n(J)^\uni$. 
The goal of the remainder of this subsection is to show that these stacks only depend on the formal neighborhood of $\one \in J$.

Just as above, we may define the presheaf of categories $\Tor_n(\widehat{J})$ and the presheaf of sets $\Sym^n(\widehat{J})$ parameterizing $S$-families of torsion sheaves (respectively effective Cartier divisors) in $\widehat{J}$. Again, there is an associated divisor map denoted abusively also by $Div$:
\[
Div\colon \Tor_n(\widehat{J}) \to \Sym^n(\widehat{J}).
\]
We denote by $\Tor_n(\widehat{J})^\uni$ the fiber over $\one_n$.

\begin{lemma}
We have natural isomorphisms giving rise to a commutative diagram
\[
\xymatrix{
\Tor_n(\widehat{J}) \ar[r]^\sim \ar[d] & \Tor_n(J)^\wedge_{\uni} \ar[d]\\
\Sym^n(\widehat{J}) \ar[r]^\sim & \Sym^n(J)^\wedge_{\one_n}.
}
\]
In particular, there is an equivalence 
\[
\Tor_n(J)^\uni \cong \Tor_n(\widehat{J})^\uni.
\]
\end{lemma}
\begin{proof}		
First, we claim that for a test scheme $S$, the image of $\Tor_n(\widehat{J})(S)$ in $\Tor_n(J)(S)$ consists of $S$-families of torsion sheaves on $J$ which are set-theoretically supported on the closed subset $\one_S \subseteq J_S$. Indeed, note that an $S$-family of torsion sheaves on $\widehat{J}$ is, by definition, given by an $S$-family of torsion sheaves on some infinitesimal thickening of $\one \in J$. Giving such a family is indeed equivalent to an $S$-family of torsion sheaves on $J$ which is set-theoretically supported on $\one_S \subseteq  J_S$. Similarly, one shows that the image of $\Sym^n(\widehat{J})(S)$ in $\Sym^n(J)(S)$ consists of $S$-families of divisors which are set-theoretically supported in $\one_{S}$.

To show that $\Sym^n(\widehat{J}) \cong \Sym^n(J)^\wedge_{\one_n}$ observe that an $S$-point of the completion $\Sym^n(J)^\wedge_{\one_n}$ corresponds to an $S$-family of degree $n$ effective divisors in $\Sym^n(J)(S)$ whose restriction to $S_{\red}$ is equal to the divisor $\one_{n,S_{\red}}$. 
This is equivalent to saying that the set-theoretic support of the divisor is contained in $\one_S$ as required. 

Finally, to show that $\Tor_n(\widehat{J}) \cong \Tor_n(J)^\wedge_\uni$, it suffices to prove that the diagram
\[
\xymatrix{
	\Tor_n(\widehat{J}) \ar[d] \ar[r] &  \Tor_n(J) \ar[d] \\
	\Sym^n(\widehat{J}) \ar[r] &  \Sym^n(J)
}
\]
is cartesian. This follows from the observation that an $S$-family of torsion sheaves on $J$ is set-theoretically supported on $\one_S$ if and only if the associated divisor is set-theoretically supported on $\one_S$.
\end{proof} 

\begin{remark}
	In fact, an $S$-family of torsion sheaves is unipotent if and only if its (scheme-theoretic) support is contained in the subscheme $\one_{n,S}$. 
	This is a consequence of the Cayley-Hamilton theorem (see \cref{E:torsion example}).
\end{remark}

\begin{example}\label{E:torsion example}
Let $J=\Ga = \Spec(k[t])$, and let $S=\Spec(R)$ for a Noetherian local ring $R$. 
Then the groupoid $\Tor_n(\Ga)(R)$ consists of $R[t]$-modules $M$ which are free over $R$ of rank $n$.

Given such a module $M$ and choosing an $R$-basis, the action of $t$ is expressed by a matrix $A_M$. Then the associated divisor is given by the function $\det(t-A_M)$ (a polynomial of degree $n$ with coefficients in $R$). The coefficients of $\chi_A(t) = \det(t-A_M)$ define the corresponding element of $R^n = \bA^n(R) \cong \Sym^n(\Ga)(R)$. Note that, by the Cayley-Hamilton theorem,  $M_A$ is scheme-theoretically supported in the divisor $\chi_A(t)$. An $R$-point is nilpotent (respectively, formally nilpotent) if $\chi_A(t) = t^n$ (respectively, the non-leading coefficients of $\chi_A(t)$ are nilpotent in $R$). 
\end{example}

\subsection{Unipotent cones}

The next result explains that the subfunctors of $\uGE$ corresponding to unipotent (respectively infinitesimally unipotent) bundles can be understood via the equivalence of \cref{C:ssbundles as functors} as those functors which factor through the subcategories of unipotent (respectively, infinitesimally unipotent) torsion sheaves.

\begin{proposition}\label{p:unipotent iff associated unipotent}
	Let $\cP$ denote an object of $\uGE(S)$ with associated functor 
	\[ F\colon\Rep(G) \to \QC(J(E)\times S) .\]
	Then $\cP$ is contained in the sub-groupoid 
	of unipotent (respectively infinitesimally unipotent) bundles if and only if, for each $V \in \Rep(G)$, $F(V)$ is a unipotent (respectively, infinitesimally unipotent) family of torsion sheaves.
\end{proposition}
\begin{proof}
	First note that for any morphism of groups $G \to G'$, there is a commutative diagram:
	\[
	\xymatrix{
		\uGE \ar[d] \ar[r] & \uGpE  \ar[d] \\
		\MGE \ar[r] & \MGpE
	}
	\]
	It follows that if an object of $\uGE(S)$ is unipotent (respectively infinitesimally unipotent) then its image in $\uGpE(S)$ is unipotent (respectively infinitesimally unipotent).
	
	In particular, we obtain such a diagram with $G'=GL(V)$ for each representation $V$ of $G$. Thus (noting \cref{L:div=char}) it follows that if $\cP_G$ is (infinitesimally) unipotent then all the associated torsion sheaves $F(V)$ are (infinitesimally) unipotent.
	
	Note that if $V$ is a faithful representation, then the map 
	\[
	\MGE \to \cM_E(GL_n) \cong \Sym^n(J(E))
	\]
	has the property that the set-theoretic fiber of the basepoint $\one_n$ consists only of the basepoint $\one_G$ of $\MGE$ (i.e. a semisimple $G$-bundle is trivial if and only if the associated vector bundle is trivial, at a set-theoretic level). It follows then that if $F(V)$ is infinitesimally unipotent, then $\cP$ is infinitesimally unipotent.
	
	This argument doesn't quite work for the non-infinitesimal case, as the map 
		\[
	\MGE \to \cM_E(GL_n) \cong \Sym^n(J(E))
	\]
	is not injective on $S$-points for a general (possibly non-reduced) scheme $S$.
	
	However, \cref{l:embedding of coarse moduli} below implies that it is enough to take a sufficiently large collection of representations $V_1, \ldots, V_m$ (for example the fundamental representations) to obtain that the map
			\[
	\MGE \to \prod_i \cM_E(GL_{d_i}) \cong \Sym^{d_1}(J(E))\times \ldots \times \Sym^{d_m}(J(E))
	\]
	is a closed embedding in a formal neighbourhood of the basepoint (and thus injective on $T$-points set-theoretically supported on the basepoint).
\end{proof}

\begin{lemma}\label{l:embedding of coarse moduli}
		Let $V_1, \ldots, V_m$ be such that their classes generate $R(G)$ as a ring . For example, we can take the collection of fundamental representations.
	Then the corresponding map
	\begin{equation}\label{eq:embedding E}
	\MGE \to \Sym^{d_1}(J(E)) \times \ldots \times \Sym^{d_m}(J(E))
	\end{equation}
	is a closed embedding in a formal neighbourhood of the basepoint $\one_{G} \in \MGE$.
\end{lemma}
\begin{proof} 
By assumption, the map of representation rings 
\begin{equation}\label{eq:embedding Gm}
R(GL(V_1)) \otimes \ldots \otimes R(GL(V_r)) \to R(G)
\end{equation}
is surjective. Thus the corresponding map of varieties
\[
T\git W \to \Sym^{d_1}(\Gm) \times \ldots \times \Sym^{d_r}(\Gm)
\]
is a closed embedding. 
In particular it is a closed embedding after completing at the basepoint of $T\git W$.

Choosing an isomorphism of formal schemes (not necessarily respecting the group structure) $\widehat{\Gm} \cong \widehat{J(E)}$ gives an identification of the maps in (\ref{eq:embedding E}) and (\ref{eq:embedding Gm}) in a formal neighbourhood of the basepoint as required.
\end{proof}

We are now ready to prove \cref{t:unipotent} (and consequently \cref{thm:unipotent G bundles intro} from the introduction). 

\begin{theorem}\label{T:formal equivalence}
Any equivalence of formal groups $\widehat{J(E_1)} \cong \widehat{J(E_2)}$ defines equivalences:
\[
\xymatrix{	
(\uGEone)^\wedge_\uni \ar[r]^\sim & (\uGEtwo)^\wedge_\uni \\
\uGEone^\uni \ar[r]^\sim \ar[u] & \uGEtwo^\uni {\ar[u]} .
}
\]
\end{theorem}

\begin{proof}[Proof of \cref{T:formal equivalence}]
The first part of the theorem says that there is an equivalence of stacks of infinitesimally unipotent bundles:
	\[
	\xymatrix{	
		(\uGEone)^\wedge_\uni \ar[r]^\sim & {(\uGEtwo)^\wedge_\uni}.
	}
	\]
By \cref{C:ssbundles as functors}, we have an identification 
\[
\uGEi \cong \Fun^\otimes(\Rep(G),\Tor(\widehat{J(E_i)})).
\]
The identification of formal groups $\widehat{J(E_1)} \cong \widehat{J(E_2)}$ defines, for each test scheme $S$, an equivalence of symmetric monoidal categories $\Tor(\widehat{J(E_1)})(S) \simeq \Tor(\widehat{J(E_2)})(S)$, and thus we obtain the required equivalence.

The second part of the theorem means that this equivalence preserves the substacks of unipotent bundles. But according to \cref{p:unipotent iff associated unipotent}, the unipotent bundles may be recognized as those functors whose corresponding tensor functors factor through the subcategory of unipotent torsion sheaves. As the subcategory of unipotent torsion sheaves is preserved under the equivalence
\[
\Tor(\widehat{J(E_1)}) \simeq \Tor(\widehat{J(E_2)}).
\]
we obtain the required result.
\end{proof}

\newpage
\section{Examples}\label{S:examples}
In this section we compute (partially) some examples. 
The reference and notation for root systems that we used is from the appendix of \cite{Bour456}.

We recall that $\uGE$ stands for the stack of \emph{semistable} $G$-bundles of degree $0$ on $E$.
Denote by $J$ the Jacobian of $E$. 

\begin{example}
	Take $G=\PGL(2)$. We have two possible closed sets: empty set, full set. 
	The Weyl group acts on $(\bG_{m})_E=J$ by
	$\cL\mapsto \cL^{-1}$.
	
	We have a decomposition
	\[\uGE = \cN/G\bigsqcup  \left(J\setminus\{\cO\}\right)/T\rtimes\mathfrak{S}_2.\]

	Observe that the $\PGL(2)$-bundles $\cO\oplus\cL$ where $\cL\in J[2]$ and $\cL\neq \cO$, have automorphism group $T\rtimes\mathfrak{S}_2$ which is disconnected.
	This is to be expected because the centralizer of a semisimple element in a non-simply-connected group doesn't have to be connected.
\end{example}

\begin{example}
	Take $G=\SL(2)$. Then we have 
	\[\uGE=\left( J[2]\times\cN/G \right)\bigsqcup (J \setminus J[2])/T \rtimes\mathfrak{S}_2\]
	and here all the bundles have connected automorphism groups.
	
\end{example}

\begin{example}
	Take $G=\GL(2)$, $T=\bG_m^2$. Then we have
	\[	\uGE=J\times\cN/G\bigsqcup (\TE\setminus\diag)/(T\rtimes\mathfrak{S}_2) \]
	and here also all the automorphism groups are connected.
\end{example}

\begin{example}\label{Eg:type G2}
	Let $G$ be a group of type $G_2$. 	The root system is \[\Phi=\pm\{\alpha,\beta,\alpha+\beta,2\alpha+\beta,3\alpha+\beta,3\alpha+2\beta\}\]
	where $\alpha$ is the short root.
	
	The closed subsets not contained in any proper Levi are $\Phi$ and  \[\Sigma=\pm\{\beta,3\alpha+\beta,3\alpha+2\beta\}.\]
	We have $G(\Sigma) = \SL(3)$ which is a pseudo-Levi subgroup.
	The other closed sets (up to conjugation by $W$) are $\{\alpha\}$, $\{\beta\}$ and $\emptyset$ and one can easily see that $G(\alpha) \simeq  \GL(2)\simeq G(\beta)$ and $G(\emptyset)=T\simeq\bG_m^2$. 
	It is also an exercise to check that $N_G(G(\alpha))=G(\alpha)$ and similarly for $G(\beta)$.
	
	The roots give us an isomorphism $\TE\simeq J^2$ where the first coordinate corresponds to $\alpha$ and the second to $\beta$.
	The partition of $\uGE$ is therefore
	\begin{align*}
	\uGE =  \cN_G/G & \sqcup \left( J[3]\times \cN_{\SL(3)} \right)/\SL(3) \\ 
				& \sqcup(J-J[3])\times\cN_{G(\alpha)}/G(\alpha) \\
				& \sqcup J\times\cN_{G(\beta)}/G(\beta)\\
				& \sqcup(\TE-\coord)/T\rtimes W.
	\end{align*} 
	where $\coord\simeq J\times\{\cO_E\}\cup\{\cO_E\}\times J\subset J\times J$ corresponds to the coordinate axes.
\end{example}

Now a more involved example:
\begin{example}\label{Eg:type C3}
	Take $G=\Sp(6)$ which is a group of type $C_3$. 
	The simple roots are $\{\alpha_1,\alpha_2,\alpha_3\}$ with $\alpha_3$ the long root.
	The longest root is $2\alpha_1+2\alpha_2+\alpha_3$. 
	
	We put $\alpha_0:=-(2\alpha_1+2\alpha_2+\alpha_3)$.
	The affine Dynkin diagram is

\[ 	\begin{tikzpicture}
	\draw[double distance=2pt, thick] (0,0)--(1,0);
	\draw[thick] (0.5,0) circle (0pt) node{$>$} ;
	\filldraw[black] (0,0) circle (2pt) node[anchor=south] {$\alpha_0$};
	\filldraw[black] (1,0) circle (2pt) node[anchor=south] {$\alpha_1$};
	\draw[thick, double distance=2pt] (2,0) -- (3,0);
	\filldraw[black] (2,0) circle (2pt) node[anchor=south] {$\alpha_2$};
	\filldraw[black] (3,0) circle (2pt) node[anchor=south] {$\alpha_3$};
	\draw[black, thick] (1,0) -- (2,0);	
	\draw[thick] (2.5,0) circle (0pt) node{$<$} ;
	\end{tikzpicture} \]
	We use the Borel-de Siebenthal algorithm to produce closed sets (see \cref{SS:thm of Borel de Siebenthal}).
	If we remove only $\alpha_0$ or only $\alpha_3$ we get back the group $G$. 
	If we remove the affine vertex and some other vertices we get all the Levi subgroups of $G$.
	
	If we remove one vertex different from $\alpha_0$ or $\alpha_3$ we get
	
	\begin{tabular}{lcc}
		\begin{tikzpicture}
		\filldraw[black] (0,0) circle (2pt) node[anchor=south] {$\alpha_0$};
		\draw[thick, double distance=2pt] (2,0) -- (3,0);
		\filldraw[black] (2,0) circle (2pt) node[anchor=south] {$\alpha_2$};
		\filldraw[black] (3,0) circle (2pt) node[anchor=south] {$\alpha_3$};
		\draw[thick] (2.5,0) circle (0pt) node{$<$} ;
		\end{tikzpicture} &  $\Sigma_1=\pm\left\{
		\begin{array}{l}
		\alpha_2,\alpha_3,\alpha_2+\alpha_3, 2\alpha_2+\alpha_3, \\
		2\alpha_1+2\alpha_2+\alpha_3
		\end{array}\right\}$ & 
		$\rightsquigarrow \SL(2)\times \Sp(4)$
		
	\end{tabular}

	\begin{tabular}{lcc}
		\begin{tikzpicture}
		\draw[double distance=2pt, thick] (0,0)--(1,0);
		\draw[thick] (0.5,0) circle (0pt) node{$>$} ;
		\filldraw[black] (0,0) circle (2pt) node[anchor=south] {$\alpha_0$};
		\filldraw[black] (1,0) circle (2pt) node[anchor=south] {$\alpha_1$};
		\filldraw[black] (3,0) circle (2pt) node[anchor=south] {$\alpha_3$};
		\end{tikzpicture} &  
		$\Sigma_2=\pm \left\{
		\begin{array}{l} 
		\alpha_3,\alpha_1,2\alpha_1+2\alpha_2+\alpha_3,\\
		\alpha_1+2\alpha_2+\alpha_3,2\alpha_2+\alpha_3
		\end{array}
		\right\}$ & $\rightsquigarrow\Sp(4)\times \SL(2)$
		
	\end{tabular}
	
	One can check easily that $s_2s_1$ takes $\Sigma_1$ into $\Sigma_2$, hence also the group $G({\Sigma_1})$ into $G({\Sigma_2})$.
	
	By applying once more the above algorithm we get (discarding the Levi subgroups) the following root system:
	
	\begin{tabular}{lcc}
		\begin{tikzpicture}
		\filldraw[black] (0,0) circle (2pt) node[anchor=south] {$\alpha_0$};
		\filldraw[black] (1.5,0) circle (2pt) node[anchor=south] {$\alpha_0'$};
		\filldraw[black] (3,0) circle (2pt) node[anchor=south] {$\alpha_3$};
		\end{tikzpicture}&$\qquad \Sigma_3=\pm\left\{
		\alpha_3, 2\alpha_1+2\alpha_2+\alpha_3, 2\alpha_2+\alpha_3
		\right\}$ & $\rightsquigarrow\SL(2)^3$
	\end{tabular}
	
	where $\alpha_0'=-(2\alpha_1+\alpha_0)=2\alpha_2+\alpha_3$.
	
	The groups $G(\Sigma)$ are computed by inspecting the root/coroot system and by looking at the center and the fundamental group.
	
	We can also compute the relative Weyl groups and find $W_{G,\Sigma_1}=1$ and $W_{G,\Sigma_3}=\mathfrak{S}_3$. 
	Actually the Weyl group of $\Sp(6)$ is $W=(\bZ/2\bZ)^3\rtimes \mathfrak{S}_3$ and one can check (or see from the diagram) that the Weyl group of $\Sigma_3$ is $W_{\Sigma_3}\simeq \bZ^3$. 
	
	If we iterate once more the algorithm we only get Levi subgroups of $\Phi$, $\Sigma_1$ or $\Sigma_3$.
	
	For the partial order $\succeq$ the closed sets $\Phi,\Sigma_1,\Sigma_3$  are the maximal ones (up to permutation by $W$).
	
	Hence the closed pieces in the partition of $\uGE$ are 
	\[J[2]\times\cN_{\Sp(6)}/\Sp(6)\]
	\[\left( J[2]^2-\diag \right)\times\cN_{\SL(2)\times \Sp(4)}/\SL(2)\times\Sp(4)\]
	\[\left((J[2]^3- \diags) \times \cN_{\SL(2)^3} /\SL(2)^3\right)/\mathfrak{S}_3.\]
	
	Here also, all the bundles have connected automorphism group even though we quotient by $\mathfrak{S}_3$ because its action is free on $J[2]^3\setminus \diags$.
\end{example}

\begin{example}\label{Eg:type D4} 
	The last example we compute (in detail) is a simply connected group of type $D_4$.
	We would like to provide an example to show that the automorphism group of a semisimple bundle can be disconnected.
	
	Let $G=\Spin(8)$ be the simply connected group of type $D_4$ and denote by $T$ a maximal torus. 
	The simple roots are denoted by $\alpha_i,i=1,2,3,4$ and the center of $G$ is isomorphic to $\mu_2\times\mu_2$.

	For convenience, let us spell out the root datum that we used for the computations (for more details one should consult \cite[Planche IV, p.256]{Bour456}):
	\begin{itemize}
		\item the character lattice and the root lattice are
		\[ X^*(T)=\<\omega_1,\omega_2,\omega_3,\omega_4\> \supset \<\alpha_1,\alpha_2,\alpha_3,\alpha_4\> \] 
		where the roots are given in terms of fundamental characters as:
		\begin{align*}
			\alpha_1 &= -2\omega_1+\omega_2\\
			\alpha_2 & = -\omega_1 +2\omega_2-\omega_3-\omega_4\\
			\alpha_3 & = -\omega_2+2\omega_3\\
			\alpha_4 & = -\omega_2+2\omega_4
		\end{align*}
		\item the cocharacter lattice which equals the coroot lattice (because simply connected) is
		\[ X_*(T) = \<\valpha_1,\valpha_2,\valpha_3,\valpha_4\> \]
		\item the longest root is $\alpha_0:=\alpha_1+2\alpha_2+\alpha_3+\alpha_4$.
	\end{itemize}
	Notice that the simple coroots give us an isomorphism $(\valpha_1,\valpha_2,\valpha_3,\valpha_4)\colon\bG_m^4\to T$.
	It is useful ot think of the simple coroots (and fundamental characters) as being the coordinates of $T$.

	The affine Dynkin diagram is
	
	\begin{tabular}{cc}
		\parbox[c]{4cm}
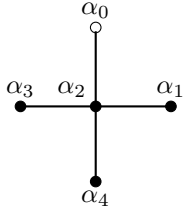
{
			\begin{tikzpicture}
			\filldraw[black] (0,0) circle (2pt) node[anchor=south east] {$\alpha_2$};
			\filldraw[black] (1,0) circle (2pt) node[anchor=south] {$\alpha_1$};
			\filldraw[black] (-1,0) circle (2pt) node[anchor=south] {$\alpha_3$};
			\filldraw[black] (0,-1) circle (2pt) node[anchor=north] {$\alpha_4$};
			\draw[black] (0,1.05) circle (2pt) node[anchor=south] {$\alpha_0$};
			\draw[thick] (0,0) -- (1,0);
			\draw[thick] (0,0) -- (0,1);
			\draw[thick] (0,0) -- (-1,0);
			\draw[thick] (0,0) -- (0,-1);		
			\end{tikzpicture} }& where $\alpha_0=-(\alpha_1+2\alpha_2+\alpha_3+\alpha_4)$.
	\end{tabular}
	
	If we remove a vertex different from $\alpha_2$ we get the whole $\Phi$. Removing $\alpha_2$ we obtain the diagram of a group of type $A_1\times A_1 \times A_1 \times A_1$. 
	
	Therefore (using again Borel--de Siebenthal algorithm, \cref{SS:thm of Borel de Siebenthal}), up to conjugacy, there are only two maximal closed sets, namely $\Phi$ and 
	\[\Sigma:=\pm\{\alpha_1,\alpha_3,\alpha_4,\alpha_1+2\alpha_2+\alpha_3+\alpha_4\}\]
	which is a root system of type $A_1 ^4$.
	All the other closed subsets that we can obtain iterating the algorithm are Levi subgroups of $G$ or of $G(\Sigma)$.
	
	The group $G(\Sigma)$ can be computed to be 
	\[(\SL(2)\times \SL(2)\times \SL(2)\times \SL(2))/\mu_2\]
	where $\mu_2$ is the diagonal central subgroup.
	(This is achieved by computing the root datum for $G(\Sigma)$.)
	Notice that this is not simply connected!	
	
	Let us recall that the Weyl group of $G$ is isomorphic to $\mathfrak{S}_4\ltimes P$ where $P\le (\bZ/2\bZ)^4$ is the hyperplane $\prod x_i=1$ and where the symmetric group acts on it by permutations.
	One sees best the action of $\mathfrak{S}_4\ltimes P$ on characters/cocharacters by introducing additional variables $\varepsilon_i, i=1,2,3,4$ such that $\omega_1=\varepsilon_1$, $\omega_2=\varepsilon_1+\varepsilon_2$, $\omega_3=\frac12(\varepsilon_1+\varepsilon_2+\varepsilon_3+\varepsilon_4)$ and $\omega_4=\frac12(\varepsilon_1+\varepsilon_2+\varepsilon_3-\varepsilon_4)$.
	Using these coordinates the action of $\fS_4$ is by permuting the $\varepsilon_i$ and the action of $P$ is by multiplication (i.e. changing signs).
	
	The Weyl group  $W_\Sigma$ of $G(\Sigma)$ is generated by the permutations $(12)$, $(34)$ together with $(-1,-1,1,1)$, $(1,1,-1,-1)\in P$ where we think of $\bZ/2\bZ=\{\pm1\}$ multiplicatively.

	Another computation shows that the normalizer of $W_\Sigma$ in $W$ is generated by $W_\Sigma$, $P$ and the permutation $(13)(24)$ . 
	The relative Weyl group is
	\[W_{G,\Sigma}\simeq\<(13)(24)\>\times \<(-1,1,-1,1)\>\simeq \mathfrak{S}_2\times \bZ/2\bZ.\]
		
	There are two closed strata in $\uGE$, namely
	\begin{equation}\label{Eq:closed strata in Spin8}
	\begin{aligned}
		&J[2]^2\times \cN_{G}/G\\
		&\left((J[2]^3-J[2]^2) \times \cN_{G(\Sigma)}\right)/G(\Sigma)\rtimes W_{G,\Sigma}.
	\end{aligned}
	\end{equation}
	
	Let us be more explicit about the semisimple parts $J[2]^2$ and $J[2]^3$.
	
	The center of $\Spin(8)$ is $\{t\in T\mid \alpha_i(t) = 1\}$ which, using the identification $\bG_m^4\simeq T$ given by cocharacters becomes 
	\[ \{(z_1,z_2,z_3,z_4)\in T\mid z_2=z_1^2=z_3^2=z_4^2 \text{ and } z_2^2=z_1z_3z_4\}\]
	which can be rewritten as
	\[ \{(z_1,z_2,z_3,z_4)\in T\mid 1=z_2=z_1^2=z_3^2=z_4^2=z_1z_3z_4\} \simeq \mu_2^2.\]
	
	Similarly, the center of $G(\Sigma)$ is
	\[ Z(G(\Sigma)) = \{(z_1,z_2,z_3,z_4)\in T\mid 1=z_2=z_1^2=z_3^2=z_4^2\} \simeq \mu_2^3.\]
	
	So in \cref{Eq:closed strata in Spin8} the $J[2]^2$ corresponds to $Z(\Spin(8))\simeq\mu_2^2$ and $J[2]^3$ corresponds to $Z(G(\Sigma))\simeq \mu_2^3$.
	The complement can also be made explicit
	\[ J[2]^3-J[2]^2=\{(\cL_1,\cO,\cL_3,\cL_4)\mid \cL_i\in J[2] \text{ and }\cL_1\cL_3\cL_4\not\simeq\cO\}. \]
	One can check that the relative Weyl group $W_{G,\Sigma}$ acts on the above locus without fixed points.
	Hence we won't find a semisimple bundle whose automorphism group (as a $\Spin(8)$-bundle) contains $G(\Sigma)$ and is disconnected.
	
	However, we'll produce a semisimple bundle whose automorphism group (as a $\Spin(8)$-bundle) is disconnected with connected component precisely the maximal torus.
	Using the identification $\bG_m^4\simeq T$ given by the simple coroots, a $T$-bundle is a quadruple of line bundles $\cP_T=(\cL_1,\cL_2,\cL_3,\cL_4)$.
	
	To simplify the analysis we impose furthermore $\cL_i^2\simeq \cO$ for $i=1,2,3,4$ (this will ensure later on that some element of the Weyl group stabilizes it).
	In order for $\Aut(\cP_T\timesT \Spin(8))^\circ$ to be $T$ we must have $\alpha_{*}(\cP_T)\not\simeq \cO$ for all roots $\alpha$ (there are 12 positive roots).
	Given the simplifying assumption we've made $\cL_i^2\simeq\cO$, $i=1,2,3,4$, this boils down to the following conditions
	\begin{align*}
	\cL_2\not\simeq&\cO	& \text{and}&	& \cO\not\simeq \cL_1\cL_3\cL_4\not\simeq&\cL_2.
	\end{align*}
		
	Let $\cL_2,\cL_4\in J[2]\setminus\{\cO\}$ be two non-isomorphic line bundles (this is possible if we're not in characteristic $2$).
	Then the $T$-bundle $\cP_T:=(\cO,\cL_2,\cO,\cL_4)$ satisfies the above conditions and hence the automorphism group of $\cP:=\cP_T\timesT \Spin(8)$ has connected component equal to $T$.
	
	The following element in the Weyl group $\sigma:=\sigma_{\alpha_0}\sigma_{\alpha_1}\sigma_{\alpha_3}\sigma_{\alpha_4}$ stabilizes $\cP_T$ (as a $T$-bundle!).
	More precisely, in general we have $\sigma(\cL_1,\cL_2,\cL_3,\cL_4) = (\cL_1\inv,\cL_2\inv,\cL_3\inv,\cL_4\inv)$.
	Hence $\sigma\in \Aut(\cP)$ which implies that $\Aut(\cP)$ is a disconnected group (with connected component equal to $T$).

\end{example}

\newpage
\appendix
\section{Classification of elliptic closed subsets}\label{sec:classification-of-elliptic-closed-subsets}
The focus of this section is on the case when $E$ is an elliptic curve but we'll quickly review the cusp (rational) and nodal (trigonometric) cases.
In \cref{sec:closed-subsets-of-e-type} we gave a general recipe to produce closed subsets of the root system $\Phi$ of $G$ as centralizers of elements in $\ft$, $T$ and $T_E$.
More precisely, we put
\begin{align}
x\in \ft	&\rightsquigarrow	\Sigma_x:=\{\alpha\in\Phi\mid \alpha(x)=0\} \\
t\in T		&\rightsquigarrow	\Sigma_t:=\{\alpha\in\Phi\mid \alpha(t)=1\}\\
\cP\in T_E	&\rightsquigarrow	\Sigma_\cP:=\{\alpha\in\Phi\mid \alpha_*(\cP)\simeq\triv\}.
\end{align}
The collection of these subsets will be denoted (in this section) by $\cA_\rat$, $\cA_\trig$ respectively $\cA_\ell$.

Over the complex numbers, using the analytic uniformizations $(\bX_*(T)\otimes_\bZ \bC)/\bX_*(T) \simeq T$ and $(\bX_*(T)\otimes_\bZ(\bR^2))/\bX_*(T)^{\oplus 2}\simeq E$, the proof of \cref{P:McNSomm-trig} is all that is needed.
However, to deal also with the positive characteristic, one needs to do a little combinatorics of root systems and diagonalizable groups.
The key ingredient is a Lemma from \cite[5.1]{Steinb-endo} that we record as \cref{L:from T to compact Tc} and its elliptic analog \cref{L:from E to two compact Tc}.

Let $a\in \ft$ and consider $\Phi_a\in \cA_\rat$. 
Then $\Phi_a$ is the root system of the Levi subgroup $C_G(k\cdot a) $ (the centralizer of a subtorus is always a Levi subgroup). 
If $L$ is a Levi subgroup of $G$ then for a generic element $a$ in the center of the Lie algebra $\fl$, the closed set $\Phi_a$ is the set of roots of $L$.

The trigonometric situation is a bit more complicated. 
For $t\in T$ the closed subset $\Phi_t$ is the root system of the connected centralizer $C_G(t)^\circ$ which might not be a Levi subgroup if the order of $t$ is finite. 
The reductive subgroups of $G$ thus obtained are called pseudo-Levi subgroups and they are well known in the theory of reductive groups.

Going further to the elliptic case, it turns out that the subsets $\Phi_{\cP}$ are the root systems of intersections of two pseudo-Levi subgroups of $G$.
We will sketch the proofs below after fixing some notation.

Fix a Borel subgroup $B$ containing the maximal torus $T\subset B\subset G$, and let $\Phi\supset\Delta$ be the set of roots and of simple roots respectively.  

A prime number $p$ is said to be \emph{good} for $G$ if $p$ doesn't divide any coefficient of the highest root\footnote{if the root system is not irreducible, consider all the highest roots} (w.r.t. $\Delta$) of $\Phi$.
This is a very tiny restriction on $p$ and $G$. For example, $p>5$ is good for every group and $p\ge 3$ is good for any classical group.

Denote by $\bX=\bX^*(T)$ the character lattice of $T$ and by $\bY$ its dual lattice. One can construct, in a functorial way, the \emph{compact} abelian Lie group $T_c=(\bY\otimes_\bZ\bR)/\bY$ such that $\Hom_\gr(T_c,\bC^\times)=\bX$.

For $t\in T$ we put $\bX_t:=\{\lam\in\bX\mid \lam(t)=1\}$ and similarly for $x\in T_c$ we put $\bX_x = \{\lam\in\bX\mid \lam(x)=1\}$.
We define analogously $\Phi_t$ and $\Phi_x$.

We start with a preparation lemma from \cite[5.1]{Steinb-endo} that is needed to pass from an arbitrary field to $\bR$.
\begin{lemma}\label{L:from T to compact Tc} \hfill
	\begin{enumerate}
		\item For any $t\in T$ there exists $x\in T_c$ such that $\bX_t=\bX_x$.
		\item If $x\in T_c$ is of finite order prime to $p=\car(k)$ then there exists $t\in T$ such that $\bX_t=\bX_x$. 
	\end{enumerate}
\end{lemma}
\begin{proof}
	We sketch the idea of the proof.
	
	The abelian group $\bX/\bX_t$ is finitely generated and injects into $k^\times$ through the map $\lambda\mapsto \lambda(t)$. Since a finite subgroup of $k^\times$ is cyclic we deduce that the torsion part of the abelian group $\bX/\bX_t$ is cyclic.
	To $\bX/\bX_t$ corresponds a sublattice of $\bY$ and hence a subgroup $T'_c\subset T_c$ which is the product of a compact torus and a cyclic group.
	As such it has a topological generator, say $x\in T'_c$.
	By construction we have $\bX_x=\bX_t$.

	Conversely, the finite cyclic group $\bX/\bX_x$ corresponds to a cyclic subgroup $\mu$ of $T$ of order prime to $p$.
	Hence $\mu$ has a generator $t$ of the same order as $x$ and by construction we have $\bX_t=\bX_x$.
\end{proof}
\begin{remark}
	The reason we need the order to be prime to $p=\car(k)$ is that in $\bG_m$ there are no points of order $p$, i.e. $\mu_p$ is infinitesimal.
	
	If the field $k$ has elements of infinite order then $\bG_m$ has a Zariski generator and hence any torus has a Zariski generator. Therefore in (2) above we can replace the assumption $x$ of finite order prime to $p$ by the component group of $\overline{\<x\>}$ has order prime to $p$.
\end{remark}

For $\cP\in T_E$ put $\bX_\cP = \{\lambda\in\bX\mid \lambda_*(\cP)\simeq \cO\}$.
Similarly one can prove

\begin{lemma}	\label{L:from E to two compact Tc}
For any $\cP\in T_E$ there exist $x_1,x_2\in T_c$ such that $\bX_\cP=\bX_{x_1}\cap \bX_{x_2}$.
Conversely, for $x_1,x_2\in T_c$ of finite order prime to $p$ there exists $\cP\in T_E$ such that $\bX_\cP=\bX_{x_1}\cap \bX_{x_2}$.
\end{lemma}
\begin{proof}
	In the proof of \cref{L:from T to compact Tc} we used that a finite subgroup of $k^\times$ must be cyclic.
	We also used that $k^\times$ has a primitive $r$th root of unity if and only if $r$ is prime to $\car(k)$.
	
	The analog for an elliptic curve is: a finite subgroup of $J(E)$ is a product of two cyclic groups.
	The group of torsion points $J(E)[r]$ is isomorphic to $\bZ/r\times \bZ/r$ if and only if $r$ is prime to $\car(k)$.
	
	Hence the quotient $\bX/\bX_\cP$ is a product of a free abelian group and two cyclic groups.
	The rest of the proof is the same.
\end{proof}
We can now easily deduce
\begin{corollary} \hfill
	\begin{enumerate}
	\item For any $t\in T$ there exists $x\in T_c$ such that $\Phi_x=\Phi_t$. 
	\item For any $\cP\in T_E$ there exist $x_1,x_2\in T_c$ such that $\Phi_\cP = \Phi_{x_1}\cap \Phi_{x_2}$.
	\end{enumerate}
\end{corollary}

Conversely, by \cite[Prop. 30,32]{McNSomm} we have
\begin{proposition}\label{P:char(k) good then pseudo Levi same as char 0}
	Assume $\car(k)$ is good for $G$.
	\begin{enumerate}
		\item For any $x\in T_c$ there exists $t\in T$ such that $\Phi_t=\Phi_x$.
		\item For any $x_1,x_2\in T_c$ there exists $\cP\in T_E$ such that $\Phi_\cP=\Phi_{x_1}\cap\Phi_{x_2}$.
	\end{enumerate}
\end{proposition}
By construction, $\Phi_S$ is a closed subset of $\Phi$.
Let 
\[ \widetilde{\Delta} := \Delta \sqcup \{\alpha_0^{(l)}: l  \text{ a connected component of }\Phi\} \]
be the set of simple roots of the corresponding affine root system, where $\alpha_0^{(l)}$ is the negative of the longest root in the corresponding connected component $\Phi_+^{(l)}$. (If the Dynkin diagram is not connected there are several longest roots corresponding to each connected component and we want to add the negative of each of them to $\Delta$.)

Given a subset $S \subset \widetilde{\Delta}$ put $\Phi_S := \bZ S \cap \Phi$. 

Just as Levi subgroups correspond, up to conjugation, to subsets of the simple roots, pseudo-Levi subgroups admit a similar characterization. 

\begin{proposition}\cite[Lemma 29]{McNSomm},\cite[Lemma 5.4]{lusztig_classification_1995}
	\label{P:McNSomm-trig}
	The set $\{\Phi_x\mid x\in T_c\}/W$ consists of subsets of the form $\Phi_S$ for some proper subset $S$ of $\widetilde{\Delta}$ as defined above. 
\end{proposition}
\begin{proof}	
	We recall the proof from loc.cit. for the convenience of the reader.
	Put $\pi\colon\ft_\bR = \bY\otimes_\bZ \bR \to T_c$ the projection.
	
	Let $x\in T_c$.
	If the closure of the subgroup generated by $x$ is a  torus, its centralizer is a Levi subgroup and we're done.
	
	Otherwise, write $x=x's$ with $s$ of finite order and $x'$ that generates a torus.
	We have $\Phi_{x's} = \Phi_{x'}\cap \Phi_s$ and since $\Phi_{x'}$ corresponds to a Levi subgroup we are left to deal with $\Phi_s$.
	So we can suppose $x$ is of finite order in $T_c$.
	
	The affine Weyl group $\bX_*(T)\ltimes W$ acts on $\ft_\bR$ and, up to conjugating $x$ by some element of $W$, we can assume $x$ lies in the image of the fundamental alcove through the map $\pi\colon\ft_\bR\to T_c$. 
	Hence we  we can take $\tilde x\in\pi\inv(x)$ in the fundamental alcove.
	
	Looking at the roots as linear functions on $\ft_\bR$ we have
	\[ \Phi_x=\{\alpha\in\Phi\mid \alpha(\tilde x)\in\bZ\}. \]
	
	Define $S:=\{\alpha\in \widetilde\Delta\} \mid \alpha(\tilde x)\in\{0,1\}\}$. 
	By construction $\Phi_S\subset \Phi_x$. 
	Let us prove that $\Phi_S=\Phi_x$.
	
	Recall that the fundamental alcove in $\ft_\bR$ is defined by the inequalities
	\[ 0\le \alpha(y)\le 1, \text{ for all }\alpha\in\Phi^+. \]
	Denoting by $\gamma^{(l)}=-\alpha_0^{(l)}$ the highest root of $\Phi^{(l)}$, the fundamental alcove can be rewritten as
	\begin{align*}
	0\le \alpha_i(y)\le 1 & \text{ for all }\alpha_i\in\Delta \\
	0\le \gamma^{(l)}(y)\le 1 & \text{ for all connected components } l \text{ of }\Phi.
	\end{align*}
	
	For $\alpha\in\Phi_x^{(l),+}$ we have $0\le \alpha(\tilde x)\le \gamma^{(l)}(\tilde x)\le 1$ hence $\alpha(\tilde x)\in\{0,1\}$.
	
	If $\alpha(\tilde x)=0$, then $\alpha$, being a positive sum of simple roots, is a sum of simple roots all of which must belong to $S$, hence $\alpha\in\Phi_S$. 
	
	If $\alpha(\tilde x)=1$ then $\gamma(\tilde x)-\alpha(\tilde x) = 0$ hence all the simple roots appearing in $\gamma^{(l)}-\alpha$ must be in $S$. 
	Thus $\gamma^{(l)}-\alpha\in\Phi_S$ and since $\gamma^{(l)}(\tilde x)=1$ we also have $\gamma^{(l)}\in\Phi_S$. 
	We deduce $\alpha\in\Phi_S$ and the proof is finished.
\end{proof}

\begin{remark}\label{R:algorithm Borel de Siebenthal}
	One can think of the above construction in the following way. 
	Take the extended Dynkin diagram whose nodes are indexed by $\widetilde{\Delta}$ and remove some non-zero number of nodes (at least one from each connected component). 
	This will then be the (non-extended) Dynkin diagram corresponding to some pseudo-Levi subgroup of $G$. 
	In fact it is known that the Dynkin diagrams corresponding to closed subsets $\Sigma\subset\Phi$ can be obtained by repeatedly applying this procedure, see \cite{BdS} for more details.
\end{remark}

\begin{proposition}\label{P:the set A_ell}
	Assume that $\car(k)$ is good for $G$.
	The collection of closed subsets $\cA_{\ell}$ consists precisely of intersections of two elements of $\cA_\trig$. 
	In other words, the $G(\Sigma)$ for $\Sigma \in \cA_{\ell}$ are precisely the neutral component of intersections of two pseudo-Levi subgroups.
\end{proposition}
\begin{proof}
	It follows from \cref{P:McNSomm-trig} together with \cref{P:char(k) good then pseudo Levi same as char 0}
\end{proof}

\begin{remark}
	The elliptic closed subsets are obtained, up to $W$-conjugation, by applying two times the Borel--de Siebenthal algorithm: take the affine Dynkin diagram and remove some vertices and then consider the closed subset of roots generated by it (which is a root system again).
\end{remark}

\bibliographystyle{alpha}
\bibliography{MyLibrary}


\end{document}

%% file: mynewcommands.tex
\usetikzlibrary{decorations.markings}

\makeatletter
\tikzcdset{
	open/.code     = {\tikzcdset{hook, circled};},
	closed/.code   = {\tikzcdset{hook, slashed};},
	open'/.code    = {\tikzcdset{hook', circled};},
	closed'/.code  = {\tikzcdset{hook', slashed};},
	circled/.code  = {\tikzcdset{markwith = {\draw (0,0) circle (.375ex);}};},
	slashed/.code  = {\tikzcdset{markwith = {\draw[-] (-.0ex,-.4ex) -- (.0ex,.4ex);}};},
	markwith/.code ={
		\pgfutil@ifundefined%
		{tikz@library@decorations.markings@loaded}%
		{\pgfutil@packageerror{tikz-cd}{You need to say %
				\string\usetikzlibrary{decorations.markings} to use arrows with markings}{}}{}%
		\pgfkeysalso{/tikz/postaction = {
				/tikz/decorate,
				/tikz/decoration={markings, mark = at position 0.5 with {#1}}}
		}
	},
}
\makeatother

\newcommand{\rH}{\mathrm H}

\newcommand{\fS}{\mathfrak S}

\newcommand{\ft}{\mathfrak t}

\newcommand{\cA}{\mathcal A}

\newcommand{\cD}{\mathcal D}
\newcommand{\cE}{\mathcal E}

\newcommand{\cL}{\mathcal L}
\newcommand{\cM}{\mathcal M}
\newcommand{\cN}{\mathcal N}
\newcommand{\cO}{\mathcal O}
\newcommand{\cP}{\mathcal P}

\newcommand{\cU}{\mathcal U}
\newcommand{\cV}{\mathcal V}
\newcommand{\cW}{\mathcal W}

\DeclareFontFamily{U}{BOONDOX-calo}{\skewchar\font=45 }
\DeclareFontShape{U}{BOONDOX-calo}{m}{n}{<-> s*[1.05] BOONDOX-r-calo}{}
\DeclareFontShape{U}{BOONDOX-calo}{b}{n}{<-> s*[1.05] BOONDOX-b-calo}{}
\DeclareMathAlphabet{\mathcalboondox}{U}{BOONDOX-calo}{m}{n}

\newcommand{\bZ}{\mathbb{Z}}							  				
							  				
\newcommand{\bP}{\mathbb{P}}						
\newcommand{\bT}{\mathbb{T}}						
\newcommand{\bA}{\mathbb{A}}

\newcommand{\bX}{\mathbb{X}}				
\newcommand{\bY}{\mathbb{Y}}

\newcommand{\bQ}{\mathbb{Q}}
\newcommand{\bC}{\mathbb{C}}
\newcommand{\bR}{\mathbb{R}}

\providecommand{\bP}{\mathbb{P}}

\newcommand{\cQ}{\mathcal{Q}}
\newcommand{\bG}{\mathbb{G}}

\makeatletter
\let\save@mathaccent\mathaccent
\newcommand*\if@single[3]{%
	\setbox0\hbox{${\mathaccent"0362{#1}}^H$}%
	\setbox2\hbox{${\mathaccent"0362{\kern0pt#1}}^H$}%
	\ifdim\ht0=\ht2 #3\else #2\fi
}
\newcommand*\rel@kern[1]{\kern#1\dimexpr\macc@kerna}
\newcommand*\widebar[1]{\@ifnextchar^{{\wide@bar{#1}{0}}}{\wide@bar{#1}{1}}}
\newcommand*\wide@bar[2]{\if@single{#1}{\wide@bar@{#1}{#2}{1}}{\wide@bar@{#1}{#2}{2}}}
\newcommand*\wide@bar@[3]{%
	\begingroup
	\def\mathaccent##1##2{%
		\let\mathaccent\save@mathaccent
		\if#32 \let\macc@nucleus\first@char \fi
		\setbox\z@\hbox{$\macc@style{\macc@nucleus}_{}$}%
		\setbox\tw@\hbox{$\macc@style{\macc@nucleus}{}_{}$}%
		\dimen@\wd\tw@
		\advance\dimen@-\wd\z@
		\divide\dimen@ 3
		\@tempdima\wd\tw@
		\advance\@tempdima-\scriptspace
		\divide\@tempdima 10
		\advance\dimen@-\@tempdima
		\ifdim\dimen@>\z@ \dimen@0pt\fi
		\rel@kern{0.6}\kern-\dimen@
		\if#31
		\overline{\rel@kern{-0.6}\kern\dimen@\macc@nucleus\rel@kern{0.4}\kern\dimen@}%
		\advance\dimen@0.4\dimexpr\macc@kerna
		\let\final@kern#2%
		\ifdim\dimen@<\z@ \let\final@kern1\fi
		\if\final@kern1 \kern-\dimen@\fi
		\else
		\overline{\rel@kern{-0.6}\kern\dimen@#1}%
		\fi
	}%
	\macc@depth\@ne
	\let\math@bgroup\@empty \let\math@egroup\macc@set@skewchar
zen@everymath{\mathgroup\macc@group\relax}%
	\macc@set@skewchar\relax
	\let\mathaccentV\macc@nested@a
	\if#31
	\macc@nested@a\relax111{#1}%
	\else
	\def\gobble@till@marker##1\endmarker{}%
	\futurelet\first@char\gobble@till@marker#1\endmarker
	\ifcat\noexpand\first@char A\else
	\def\first@char{}%
	\fi
	\macc@nested@a\relax111{\first@char}%
	\fi
	\endgroup
}
\makeatother

\newcommand{\Coh}{\operatorname{Coh}}

\newcommand{\Sch}{\mathrm{Sch}}

\newcommand{\Tor}{\mathrm{Tor}}

\newcommand{\Rep}{\operatorname{Rep}}

\makeatletter
\newcommand{\oset}[3][0.2ex]{%
	\mathrel{\mathop{#3}\limits^{
			\vbox to#1{\kern-2\ex@
				\hbox{$\scriptstyle#2$}\vss}}}}
\makeatother

\newcommand{\timesT}{\oset{T}{\times}}

\newcommand{\timesG}{\oset{G}{\times}}
\newcommand{\timesH}{\oset{H}{\times}}

\newcommand{\lam}{\lambda}

\newcommand{\valpha}{{\check{\alpha}}}

\renewcommand{\ss}{\mathsf{ss}}
\newcommand{\car}{\operatorname{char}}
\newcommand{\triv}{\mathsf{triv}}

\newcommand{\nil}{\mathsf{nil}}
\newcommand{\diag}{\mathsf{diag}}
\newcommand{\diags}{\mathsf{diags}}
\newcommand{\coord}{\mathsf{coord}}

\newcommand{\rat}{\mathsf{rat}}
\newcommand{\trig}{\mathsf{trig}}

\renewcommand{\ell}{\mathsf{ell}}
\newcommand{\Ad}{\mathsf{Ad}}
\newcommand{\equi}{\mathsf{eq}}
\newcommand{\Irr}{\mathrm{Irr}}

\providecommand{\b}{\beta}
\newcommand{\<}{\langle}
\renewcommand{\>}{\rangle}

\newcommand{\gr}{\mathsf{gr}}

\newcommand{\reg}{\mathsf{reg}}

\newcommand{\inv}{^{-1}}

\newcommand{\red}{^\mathrm{red}}

\newcommand{\Supp}{\operatorname{Supp}}

\usetikzlibrary{decorations.markings} 
\tikzset{
  closed/.style = {decoration = {markings, mark = at position 0.5 with { \node[transform shape, xscale = .8, yscale=.4] {/}; } }, postaction = {decorate} },
  open/.style = {decoration = {markings, mark = at position 0.5 with { \node[transform shape, scale = .7] {$\circ$}; } }, postaction = {decorate} }
}
\newcommand{\PGL}{\operatorname{PGL}}
\newcommand{\GL}{\operatorname{GL}}
\newcommand{\SL}{\operatorname{SL}}

\newcommand{\Spin}{\operatorname{Spin}}

\newcommand{\LG}{{}^L\!G} 

\newcommand{\fg}{\mathfrak{g}}

\newcommand{\fh}{\mathfrak{h}}
\newcommand{\fl}{\mathfrak{l}}

\newcommand{\Lie}{\operatorname{Lie}}

\newcommand{\Bun}{\operatorname{Bun}}

\DeclareMathOperator{\Aut}{Aut}

\DeclareMathOperator{\Div}{Div}

\DeclareMathOperator{\Fun}{Fun}

\DeclareMathOperator{\Hom}{Hom}

\DeclareMathOperator{\rank}{rank}

\DeclareMathOperator{\Sp}{Sp}

\DeclareMathOperator{\Spec}{Spec}

\DeclareMathOperator{\Sym}{Sym}

\DeclareMathOperator{\Id}{\mathsf{Id}}

\newcommand{\Eis}{\operatorname{Eis}}
\newcommand{\CT}{\operatorname{CT}}

\newcommand{\git}{/\mkern-6mu/}

%% file: stratBunGellipticv4.bbl
\begin{thebibliography}{BZBJ18}

\bibitem[Bal19]{Balaji-singularcurves}
V.~Balaji.
\newblock Torsors on semistable curves and degenerations.
\newblock 2019.

\bibitem[BBR09]{bartocci_fourier_2009}
C.~Bartocci, U.~Bruzzo, and D.~Hern{\'a}ndez Ruip{\'e}rez.
\newblock {\em Fourier-Mukai and Nahm transforms in geometry and mathematical
  physics}, volume 276.
\newblock Springer Science \& Business Media, 2009.

\bibitem[BDS49]{BdS}
A.~Borel and J.~De~Siebenthal.
\newblock Les sous-groupes ferm\'es de rang maximum des groupes de {L}ie clos.
\newblock {\em Comment. Math. Helv.}, 23:200--221, 1949.

\bibitem[BEG03]{BaEvGi_rep-qtori-bdles-ell}
V.~Baranovsky, S.~Evens, and V.~Ginzburg.
\newblock Representations of quantum tori and {$G$}-bundles on elliptic curves.
\newblock In {\em The orbit method in geometry and physics ({M}arseille,
  2000)}, volume 213 of {\em Progr. Math.}, pages 29--48. Birkh\"auser Boston,
  Boston, MA, 2003.

\bibitem[BG96]{BaGi_conj}
V.~Baranovsky and V.~Ginzburg.
\newblock Conjugacy classes in loop groups and {$G$}-bundles on elliptic
  curves.
\newblock {\em Internat. Math. Res. Notices}, (15):733--751, 1996.

\bibitem[Bho01]{Bhosle-nodal}
U.~N. Bhosle.
\newblock Principalg-bundles on nodal curves.
\newblock In {\em Proceedings of the Indian Academy of Sciences-Mathematical
  Sciences}, volume 111, pages 271--291. Springer, 2001.

\bibitem[BK05]{burban_fourier_2005}
I.~Burban and B.~Kreussler.
\newblock Fourier-mukai transforms and semi-stable sheaves on nodal weierstrass
  cubics.
\newblock {\em Journal f{\"u}r die reine und angewandte Mathematik},
  2005(584):45--82, 2005.

\bibitem[BN13]{BZN}
D.~{Ben-Zvi} and D.~Nadler.
\newblock Elliptic {S}pringer theory.
\newblock {\em {arXiv:1302.7053}}, 2013.

\bibitem[BNP13]{BZNP_affch}
D.~{Ben-Zvi}, D.~Nadler, and Anatoly Preygel.
\newblock A spectral incarnation of affine character sheaves.
\newblock {\em {arXiv:1312.7163}}, December 2013.

\bibitem[Bou81]{Bour456}
N.~Bourbaki.
\newblock {\em \'El\'ements de math\'ematique}.
\newblock Masson, Paris, 1981.
\newblock Groupes et alg\`ebres de Lie. Chapitres 4, 5 et 6. [Lie groups and
  Lie algebras. Chapters 4, 5 and 6].

\bibitem[Bri18]{Brion_linearization}
M.~Brion.
\newblock Linearization of algebraic group actions, 2018.

\bibitem[BZBJ18]{ben-zvi_integrating_2018}
D.~Ben-Zvi, A.~Brochier, and D.~Jordan.
\newblock Integrating quantum groups over surfaces.
\newblock {\em Journal of Topology}, 11(4):874--917, 2018.

\bibitem[BZN18]{Betti}
D.~Ben-Zvi and D.~Nadler.
\newblock Betti geometric {L}anglands.
\newblock In {\em Algebraic geometry: {S}alt {L}ake {C}ity 2015}, volume~97 of
  {\em Proc. Sympos. Pure Math.}, pages 3--41. Amer. Math. Soc., Providence,
  RI, 2018.

\bibitem[FM01]{FriedMorgIII}
R.~Glenn Friedman and J.~W. Morgan.
\newblock Holomorphic principal bundles over elliptic curves iii: Singular
  curves and fibrations.
\newblock 2001.

\bibitem[Fra13]{Frat_cusp}
D.~Fratila.
\newblock Cusp eigenforms and the {H}all algebra of an elliptic curve.
\newblock {\em Compositio Mathematica}, 149(06):914--958, 2013.

\bibitem[Gro67]{EGA4.4}
A.~Grothendieck.
\newblock {\'E}l{\'e}ments de g{\'e}om{\'e}trie alg{\'e}brique (r{\'e}dig{\'e}s
  avec la collaboration de jean dieudonn{\'e}): Iv. {\'e}tude locale des
  sch{\'e}mas et des morphismes de sch{\'e}mas, quatrieme partie.
\newblock {\em Publications math{\'e}matiques de l'IHES}, 32:5--361, 1967.

\bibitem[Gro71]{SGA1}
Alexandre Grothendieck.
\newblock Rev{\^e}tement {\'e}tales et groupe fondamental (sga1).
\newblock {\em Lecture Note in Math.}, 224, 1971.

\bibitem[GSB19]{Groj-ShBar}
Ian Grojnowski and Nicholas Shepherd-Barron.
\newblock Del pezzo surfaces as springer fibres for exceptional groups.
\newblock {\em Proceedings of the London Mathematical Society}, 2019.

\bibitem[Gun17]{gunningham_derived_2017}
S.~Gunningham.
\newblock A derived decomposition for equivariant {$D$}-modules.
\newblock {\em arXiv preprint arXiv:1705.04297}, 2017.

\bibitem[Gun18]{gunningham_generalized_2018}
S.~Gunningham.
\newblock Generalized {S}pringer theory for {D}-modules on a reductive lie
  algebra.
\newblock {\em Selecta Mathematica}, 24(5):4223--4277, 2018.

\bibitem[KM76]{knudsen_projectivity_1976}
F.~Knudsen and D.~Mumford.
\newblock The projectivity of the moduli space of stable curves. {I}:
  {Preliminaries} on det and {Div}.
\newblock {\em Mathematica Scandinavica}, 39:19--55, June 1976.

\bibitem[KR08]{kuttler_is_2008}
J.~Kuttler and Z.~Reichstein.
\newblock Is the luna stratification intrinsic?
\newblock {\em Annales de l'Institut Fourier}, 58(2):689--721, 2008.

\bibitem[Las98]{Lasz}
Y.~Laszlo.
\newblock About {G}-bundles over elliptic curves.
\newblock {\em {Ann. Inst. Fourier}}, 48:413–424, 1998.

\bibitem[Lau90]{Lau}
G.~Laumon.
\newblock Faisceaux automorphes li\'es aux s\'eries {d'Eisenstein}.
\newblock In {\em Automorphic forms, Shimura varieties, and {$L$-functions},
  Vol.\ I {(Ann} Arbor, {MI}, 1988)}, volume~10 of {\em Perspect. Math.}, pages
  227--281. Academic Press, Boston, {MA}, 1990.

\bibitem[Li18]{li_derived_2018}
P.~Li.
\newblock Derived categories of character sheaves.
\newblock {\em arXiv preprint:1803.04289}, October 2018.

\bibitem[LN15]{li2015uniformization}
P.~Li and D.~Nadler.
\newblock Uniformization of semistable bundles on elliptic curves.
\newblock 2015.

\bibitem[Lun73]{luna_slices_1973}
D.~Luna.
\newblock Slices \'etales.
\newblock In {\em Sur les groupes alg\'ebriques}, number~33 in M\'emoires de la
  Soci\'et\'e Math\'ematique de France, pages 81--105. Soci\'et\'e
  math\'ematique de France, 1973.

\bibitem[Lur04]{Lurie_Tannaka}
J.~Lurie.
\newblock Tannaka duality for geometric stacks.
\newblock {\em arXiv preprint math/0412266}, 2004.

\bibitem[Lus84]{Lu_gen_spr}
G.~Lusztig.
\newblock Intersection cohomology complexes on a reductive group.
\newblock {\em Inventiones Mathematicae}, 75(2):205--272, 1984.

\bibitem[Lus95]{lusztig_classification_1995}
G.~Lusztig.
\newblock Classification of unipotent representations of simple p-adic groups.
\newblock {\em Int. Math. Res. Not.}, pages 517--589, 1995.

\bibitem[MFK94]{mumford_geometric_1994}
D.~Mumford, J.~Fogarty, and F.~Kirwan.
\newblock {\em Geometric invariant theory}, volume~34 of {\em Ergebnisse der
  Mathematik und ihrer Grenzgebiete (2) [Results in Mathematics and Related
  Areas (2)]}.
\newblock Springer-Verlag, Berlin, third edition, 1994.

\bibitem[Mil86]{MR861976}
J.~S. Milne.
\newblock Jacobian varieties.
\newblock In {\em Arithmetic geometry ({S}torrs, {C}onn., 1984)}, pages
  167--212. Springer, New York, 1986.

\bibitem[MS03]{McNSomm}
G.~J. McNinch and E.~Sommers.
\newblock Component groups of unipotent centralizers in good characteristic.
\newblock {\em J. Algebra}, 260(1):323--337, 2003.
\newblock Special issue celebrating the 80th birthday of Robert Steinberg.

\bibitem[Nor76]{Nori_fg}
M.~V. Nori.
\newblock On the representations of the fundamental group.
\newblock {\em Compositio Math.}, 33(1):29--41, 1976.

\bibitem[Ram96]{Ram-modI}
A.~Ramanathan.
\newblock Moduli for principal bundles over algebraic curves: I.
\newblock {\em Proc. Indian Acad. Sci. (Math. Sci.)}, 106(3):301--328, August
  1996.

\bibitem[RR16]{RiderRussell}
L.~Rider and A.~Russell.
\newblock Perverse sheaves on the nilpotent cone and {L}usztig's generalized
  {S}pringer correspondence.
\newblock In {\em Lie algebras, {L}ie superalgebras, vertex algebras and
  related topics}, volume~92 of {\em Proc. Sympos. Pure Math.}, pages 273--292.
  Amer. Math. Soc., Providence, RI, 2016.

\bibitem[Sch05]{Schmitt-principal-nodal}
A.~H.~W. Schmitt.
\newblock Singular principal g-bundles on nodal curves.
\newblock {\em Journal of the European Mathematical Society}, 7(2):215--251,
  2005.

\bibitem[Sch12]{Sch2}
O.~Schiffmann.
\newblock On the {H}all algebra of an elliptic curve, {II}.
\newblock {\em Duke Mathematical Journal}, 161(9):1711--1750, June 2012.

\bibitem[Sch14]{Schi}
S.~Schieder.
\newblock {The Harder–Narasimhan stratification of the moduli stack of
  $G$-bundles via Drinfeld’s compactifications}.
\newblock {\em Selecta Mathematica}, pages 1--69, 2014.

\bibitem[Ste68]{Steinb-endo}
R.~Steinberg.
\newblock {\em Endomorphisms of Linear Algebraic Groups}.
\newblock Memoirs of the American Mathematical Society. American Mathematical
  Society, 1968.

\bibitem[Sum74]{Sumi_equiv_emb}
H.~Sumihiro.
\newblock Equivariant completion.
\newblock {\em J. Math. Kyoto Univ.}, 14(1):1--28, 1974.

\bibitem[Sun99]{Sun_sstab}
X.~Sun.
\newblock Remarks on semistability of g-bundles in positive characteristic.
\newblock {\em Compositio Mathematica}, 119:41--52, 10 1999.

\bibitem[SV11]{Sch_Vas_McD}
O.~Schiffmann and E.~Vasserot.
\newblock The elliptic {H}all algebra, {C}herednik {H}ecke algebras and
  {M}acdonald polynomials.
\newblock {\em Compositio Mathematica}, 147(01):188--234, 2011.

\bibitem[Teo99]{teodorescu_semistable_1999}
T.~Teodorescu.
\newblock {\em Semistable torsion -free sheaves over curves of arithmetic genus
  one}.
\newblock PhD thesis, United States -- New York, 1999.

\end{thebibliography}
